\documentclass{amsart}

\usepackage{amssymb, amsmath,mathtools}
\usepackage{mathrsfs}
\usepackage{amscd}
\usepackage{verbatim}
\usepackage{stmaryrd}
\usepackage[dvipsnames]{xcolor}
\usepackage{a4wide}

\usepackage{enumerate}

\newcounter{nameOfYourChoice}

\usepackage[colorlinks,linkcolor={blue},citecolor={blue},urlcolor={purple},]{hyperref}

\usepackage[colorinlistoftodos,prependcaption,textsize=tiny]{todonotes}

\usepackage{tikz-cd}
\usetikzlibrary{arrows,positioning}
\tikzset{>=latex,shift left/.style ={commutative diagrams/shift left={#1}},
  shift right/.style={commutative diagrams/shift right={#1}}}

\usepackage{tabularx}
\newcolumntype{L}{>{\arraybackslash}X}
\usepackage{multirow}

\theoremstyle{plain}
\newtheorem{theorem}{Theorem}[section]
\theoremstyle{remark}
\newtheorem{remark}[theorem]{Remark}

\theoremstyle{plain}
\newtheorem{corollary}[theorem]{Corollary}
\newtheorem{lemma}[theorem]{Lemma}
\newtheorem{proposition}[theorem]{Proposition}
\newtheorem{definition}[theorem]{Definition}

\newtheorem{assumption}[theorem]{Assumption}

\numberwithin{equation}{section}

\def\N{{\mathbb N}}

\def\R{{\mathbb R}}

\def\T{{\mathbb T}}

\newcommand{\E}{{\mathbb E}}
\renewcommand{\P}{{\mathbb P}}
\newcommand{\F}{{\mathscr F}}

\newcommand{\g}{\gamma}

\newcommand{\om}{\omega}
\renewcommand{\O}{\Omega}

\renewcommand{\a}{\kappa}

\newcommand{\loc}{{\rm loc}}

\newcommand{\Tor}{\mathbb{T}}
\newcommand{\Dom}{\mathcal{O}}

\newcommand{\A}{\mathcal{A}}
\newcommand{\D}{\mathscr{D}}

\newcommand{\wt}{\widetilde}

\usepackage{stmaryrd}

\newcommand{\Tr}{\mathrm{Tr}}
\newcommand{\Xap}{X^{\mathrm{Tr}}_{\a,p}}
\newcommand{\Xapcrit}{X^{\mathrm{Tr}}_{\a_{\crit},p}}

\newcommand{\one}{{{\bf 1}}}
\newcommand{\embed}{\hookrightarrow}
\newcommand{\s}{\delta}

\renewcommand{\div}{\mathrm{div}}

\renewcommand{\l}{\langle}
\renewcommand{\r}{\rangle}

\newcommand{\crit}{\mathrm{c}}

\newcommand{\reg}{\delta}

\newcommand{\norm}[1]{{\left\vert\kern-0.25ex\left\vert\kern-0.25ex\left\vert #1
    \right\vert\kern-0.25ex\right\vert\kern-0.25ex\right\vert}}

\renewcommand{\emptyset}{\varnothing}

\newcommand{\gap}{\varepsilon}
\newcommand{\Progress}{\mathscr{P}}

\newcommand{\wh}{\widehat}

\newcommand{\rnoise}{g}

\newcommand{\btwod}{b}

\newcommand{\am}{a}
\newcommand{\bm}{b}

\newcommand{\Borel}{\mathscr{B}}
\newcommand{\Constant}{R}

\newcommand{\vone}{v}
\newcommand{\vtwo}{v'}

\newcommand{\hstar}{h_d}

\newcommand{\V}{\mathcal{V}}
\newcommand{\U}{\mathcal{U}}

\newcommand{\AS}{A}
\newcommand{\BS}{B}
\newcommand{\FS}{\Phi}
\newcommand{\GS}{\Gamma}

\newcommand{\ellip}{\nu}

\newcommand{\fp}{\mathsf{f}}
\newcommand{\Fp}{\mathsf{F}}
\newcommand{\gp}{\mathsf{g}}
\newcommand{\X}{\mathscr{X}}
\newcommand{\uu}{\overline{u}}
\newcommand{\vv}{\overline{v}}

\newcommand{\W}{\mathcal{O}}
\newcommand{\psim}{\psi_{\star}}
\newcommand{\xm}{x_{\star}}

\newcommand{\dd}{{\rm d}}

\allowdisplaybreaks

\begin{document}

\author{Antonio Agresti}
\address{Institute of Science and Technology Austria (ISTA), Am Campus 1, 3400 Klosterneuburg, Austria} \email{antonio.agresti92@gmail.com}

\author{Mark Veraar}
\address{Delft Institute of Applied Mathematics\\
Delft University of Technology \\ P.O. Box 5031\\ 2600 GA Delft\\The
Netherlands.} \email{M.C.Veraar@tudelft.nl}

\thanks{The first author has received funding from the European Research Council (ERC) under the Eu\-ropean Union’s Horizon 2020 research and innovation programme (grant agreement No 948819) \includegraphics[height=0.4cm]{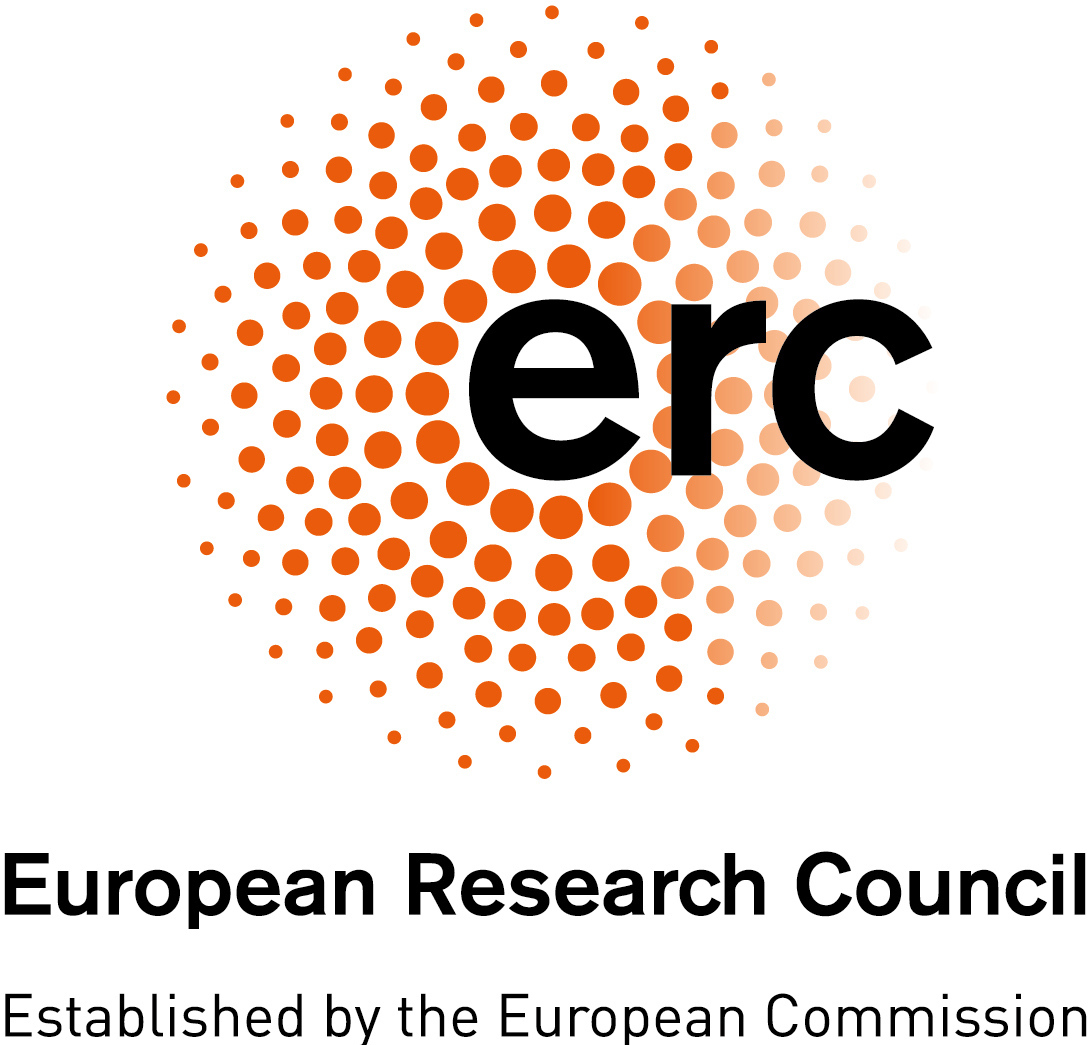}\,\includegraphics[height=0.4cm]{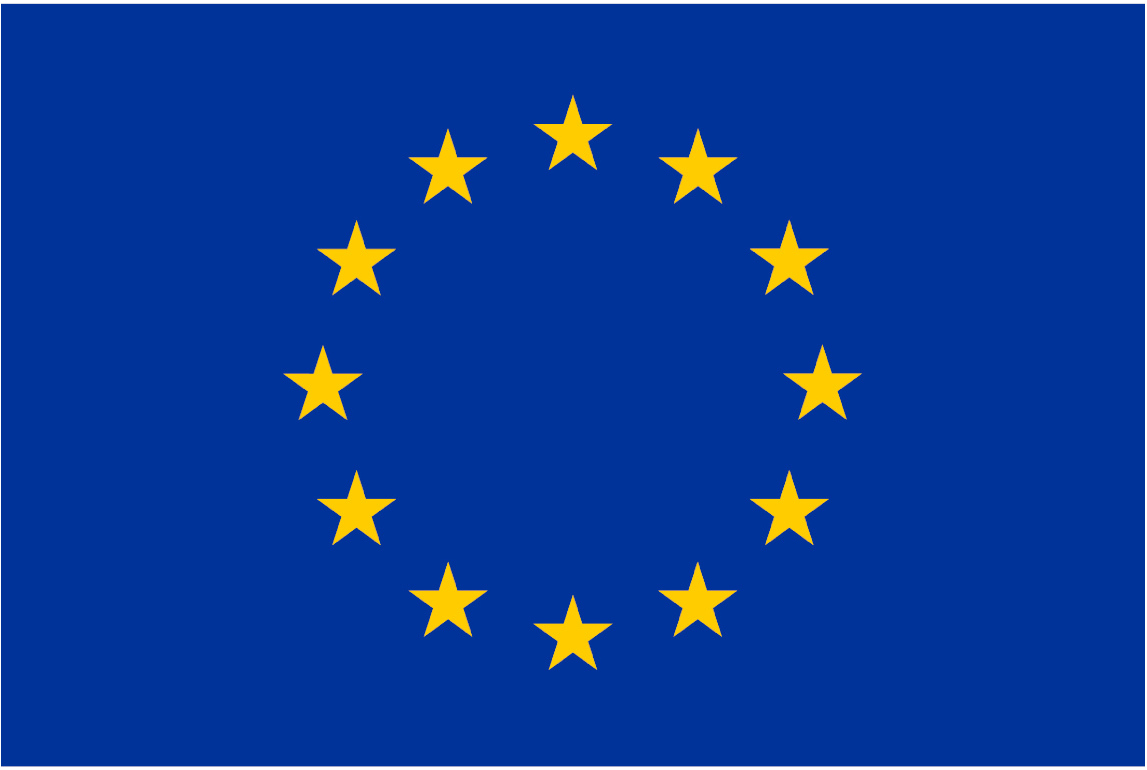}. The second author is supported by the VICI subsidy VI.C.212.027 of the Netherlands Organisation for Scientific Research (NWO)}

\date\today

\title[Reaction-diffusion equations with transport noise]{Reaction-diffusion equations with transport\\ noise and critical superlinear diffusion:\\ Local well-posedness and positivity}



\begin{abstract}
In this paper we consider a class of stochastic reaction-diffusion equations.
We provide local well-posedness, regularity, blow-up criteria and positivity of solutions. The key novelties of this work are related to the use transport noise, critical spaces and the proof of higher order regularity of solutions -- even in case of non-smooth initial data.
Crucial tools are $L^p(L^q)$-theory, maximal regularity estimates and sharp blow-up criteria.
We view the results of this paper as a general toolbox for establishing global well-posedness for a large class of reaction-diffusion systems of practical interest, of which many are completely open. In our follow-up work \cite{AVreaction-global}, the results of this paper are applied in the specific cases of the Lotka-Volterra equations and the Brusselator model.
\end{abstract}

\maketitle
\setcounter{tocdepth}{1}

\tableofcontents

\section{Introduction}
\label{s:intro}

In this paper we investigate local/global existence, uniqueness, (sharp) blow-up criteria, positivity and regularity of solutions to  the following stochastic reaction-diffusion equations with transport noise
\begin{equation}
\label{eq:reaction_diffusion_system_intro}
\left\{
\begin{aligned}
&\dd u_i -\ellip_i\Delta u_i \,\dd t= \Big[\div(F_i(\cdot, u)) +f_{i}(\cdot, u)\Big]\,\dd t + \sum_{n\geq 1}  \Big[(b_{n,i}\cdot \nabla) u_i+ \rnoise_{n,i}(\cdot,u) \Big]\,\dd w_t^n, \\
&u_i(0)=u_{i,0},
\end{aligned}
\right.
\end{equation}
where $i\in \{1,\dots,\ell\}$ for some integer $\ell\geq 1$. For simplicity we restrict ourselves to the $d$-dimensional torus $\Tor^d$, but we expect that many results can be extended to $\R^d$, and even to bounded smooth domains with suitable boundary conditions.
The unknown process is denoted by $u=(u_i)_{i=1}^{\ell}:[0,\infty)\times \O\times \Tor^d\to \R^{\ell}$, $(w^n)_{n\geq 1}$ is a sequence of standard independent Brownian motions on a filtered probability space, $F_i,f_i,g_{n,i}$ are given nonlinearities and
$$
(b_{n,i}\cdot\nabla) u_i:=\sum_{j=1}^d b^j_{n,i} \partial_j u_i.
$$
The nonlinearities $F_i,f_i,g_{n,i}$ are assumed to have polynomial growth. Moreover, the leading operator $\nu_i \Delta$ can be replaced by $\div(a_i\cdot\nabla)$. This is included in the main results of this work, as this is useful for reformulating \eqref{eq:reaction_diffusion_system_intro} with Stratonovich noise instead, see \eqref{eq:stratonovich_correction} below. Lower order terms in the differential operators can be allowed as well, and they can be included in the nonlinearities $f,F$ and $g$.

\subsection{Deterministic setting}\label{ss:mass}
Systems of PDEs of the form \eqref{eq:reaction_diffusion_system_intro} with $b=0$ and $g=0$, are usually called
reaction-diffusion equations. Such equations can be used to model a wide class of physical phenomena ranging from chemical reactions to predatory-prey systems, as well as phase separation processes. Further examples can be found in the standard reference \cite{R84_global} and in the recent survey  \cite{P10_survey}.
In the deterministic case there are many global well-posedness results available (see \cite{CGV19, FMT20, K90, P10_survey, PSY19,R84_global} and references therein). In particular, many important systems with rather weak forms of coercivity are included, but some structure is essential. As a matter of fact, existence of global \emph{smooth} solutions to \eqref{eq:reaction_diffusion_system_intro} (or more generically global well-posedness) under polynomial growth and smoothness assumptions, positivity, and mass preservation, is known to be false \cite[Section 4]{P10_survey}. This shows that even in the deterministic setting, the problem of global well-posedness is rather delicate.
Under additional entropy structures, existence of global \emph{renormalized} solutions has been established in \cite{F15_global_renormalized}. Such solutions have a rather poor regularity in time and space.  Moreover, the uniqueness of such solutions is still open, see also \cite{F17_weak_strong_uniqueness}
 for the weaker notion of weak-strong uniqueness.

Next we discuss a well-known example of reaction-diffusion equations arising in the study of chemical reactions.
For an integer $\ell\geq 1$ and two collections of nonnegative integers $(q_i)_{i=1}^{\ell},(p_i)_{i=1}^{\ell}$ (note that either $q_i=0$ or $p_i=0$ for some $i$ is allowed), consider the (reversible) chemical reaction:
\begin{equation}
\label{eq:chemical_reaction_chemistry_formula}
q_1 U_1 +\dots + q_{\ell} U_{\ell} \xrightleftharpoons[R_-]{R_+}
p_1 U_1+\dots + p_{\ell} U_{\ell},
\end{equation}
where $R_{\pm}$ are the reaction rates and $(U_i)_{i=1}^{\ell}$ are chemical substances. Let $u_i$ be the concentration of the substance $U_i$, and let $\ellip_i>0$ be its diffusivity.
The \emph{law of mass action} postulates that the concentration $u_i$ satisfies the deterministic version of \eqref{eq:reaction_diffusion_system_intro} with
\begin{equation}
\label{eq:chemical_reactions}
f_i(\cdot,u)=(p_i -q_i)\Big(R_+\prod_{j=1}^{\ell} u_j^{q_j}-R_-\prod_{j=1}^{\ell} u_j^{p_j} \Big), \ \ \ i\in \{1,\dots,\ell\}.
\end{equation}
The results of the current paper applies to \eqref{eq:reaction_diffusion_system_intro} with $f_i$ as in  \eqref{eq:chemical_reactions}. From a modeling point of view, especially in the context of chemical reactions, it is natural to ask for \emph{mass conservation} along the flow, i.e. $\int_{\T^d} u(t,x) \, \dd x$ is constant in time.
For the special case of \eqref{eq:chemical_reactions}, the mass conservation turns out to be equivalent to the existence of strictly positive constants $(\alpha_i)_{i=1}^{\ell}$ such that
\[
\sum_{i=1}^{\ell}\alpha_i(q_i-p_i)=0.
\]
Weaker notions of mass conservation are also employed, see property (M) in \cite{P10_survey}.
Although it will be not needed in most of our results below, mass conservation will be used in Theorem \ref{t:global_small_data} to provide a simple proof of global existence of (sufficiently) smooth solutions to \eqref{eq:reaction_diffusion_system_intro} for small initial data.

\subsection{Stochastic setting}
A lot of work has already been done on {\em stochastic} reaction-diffusion equations already (see \cite{C03,Cer05,CR05,DKZ19,F91,KN19,KvN12,M18,S21,S21_dissipative,S21_superlinear_no_sign} and references therein). Unfortunately, little is known for weakly dissipative systems in which the equations are coupled through nonlinear terms such as $\pm u_i^{p_i} u_j^{q_j}$. These type of nonlinearities are very common since they model reaction terms (e.g.\ Lotka-Volterra equations or chemical reactions). In the follow-up paper \cite{AVreaction-global} we prove global well-posedness for some of the important concrete systems, including the above mentioned Lotka-Volterra equations and the Brusselator model. For these applications we refer to \cite[Section 5]{AVreaction-global}.
Although with our methods we can include a general class of equations, the full setting of \cite{FMT20} and \cite[Section 3]{P10_survey} seems still completely out of reach.

The results of the current paper can be seen as a first natural step in the study of global well-posedness of weakly dissipative systems. Some general global well-posedness results will already be included in case of small initial data. Moreover, we obtain several new regularization effects, sharp blow-up criteria, and positivity results. Each of the above turns out to be crucial for proving the global well-posedness results in our follow-up work \cite{AVreaction-global}. Indeed, higher order regularity is needed to apply stochastic calculus pointwise in space. This is essential in checking energy estimates and mass conservation type conditions, which in turn can often be combined with blow-up criteria to show global existence of smooth solutions. Positivity often plays a central role in these calculations as it provides a uniform lower bound and important information on the sign of the nonlinear terms.

Let us briefly explain the main difference of our setting to the given literature on stochastic reaction-diffusion equations. To the best of our knowledge, only few papers consider superlinear diffusion (e.g.\ $\pm u_i^{p_i} u_j^{q_j}$). Also very few papers allow transport noise (i.e. $(b_{n,i}\cdot \nabla) u_i\, \dd w^n$), which seems motivated by small scale turbulence (see Subsection \ref{ss:transport_noise} below). Furthermore, there is very little $L^p$-theory for stochastic reaction-diffusion equations available. $L^p$-theory with $p>2$ is often essential when dealing with either (rough) Kraichnan noise, or nonlinearities of higher order growth and dimensions $d\geq 2$. In particular, to establish global well-posedness, $L^p$-energy estimates are typically needed for large $p$ when working with $d\geq 2$. For applications to concrete systems we refer to \cite[Section 3-5]{AVreaction-global}.
Moreover, in our work, weighted $L^p(L^q)$-theory turns out to be key in proving results on higher order regularity of solutions.

\subsection{A derivation via separation of scales}
\label{ss:transport_noise}
In this subsection we explain where the transport term in \eqref{eq:reaction_diffusion_system_intro} comes from and where it has been considered before.
In fluid dynamics, transport noise is typically used to model turbulent phenomena (see e.g.\ \cite{BCF92,F_intro,F15_book,FL19,MR01,MiRo04}), and it is usually referred to as
Kraichnan's noise due to his seminal works on turbulent flows \cite{K68,K94}.

In engineering applications, highly turbulent flows are often employed to \emph{improve} the development and efficiency of chemical reactions (compared to reactions occurring in more ``regular" flows), see e.g.\
\cite{FB18_turbulence,MS06_turbulence,VC12_turbulence,ZCB20_turbulence}.
We refer to \cite{GE63_turbulence,KD86_turbulence_chemical_reactions,LW76,MC98,SH91_turbulence} for further applications and discussions concerning turbulent flows and chemical reactions.
Here, to motivate the transport noise we follow  the heuristic argument in \cite[Subsection 1.2]{FL19}.
We refer to \cite{DP22,FP21} for (different) situations where the argument below can be made rigorous.
Before going further, let us note that our setting only requires some H\"older smoothness of $b_{n,i}^k$, and thus we are able to include the Kraichnan noise with arbitrary small correlation parameter (see \cite{AV20_NS, MiRo04} and \cite[Section 5]{GY21}). In particular this includes the case where $b_{n,i}^k$ reproduces the  Kolmogorov spectrum of turbulence according to \cite[pp.\ 427 and 436]{MK99_simplified}.

Suppose that \eqref{eq:reaction_diffusion_system_intro} models a chemical reaction taking place in a fluid where $u_i$'s are the concentration of the reactants. As commented below \eqref{eq:chemical_reaction_chemistry_formula}, from a deterministic view-point, one can consider the model:
\begin{equation}
\label{eq:reaction_diffusion_det}
\partial_t u_i =\ellip_i\Delta u_i + (v_i\cdot\nabla) u_i+\div(F_i(\cdot, u)) +f_{i}(\cdot, u),
\end{equation}
where $v_i$ models the transport effects of the fluid, $f_i$ is as in \eqref{eq:chemical_reactions} and $F_i$ is a given nonlinearity (of polynomial growth) modeling conservative source terms. In this situation, as in \cite[Subsection 1.2]{FL19}, we can assume that $v_i$ splits into a ``Small" and  ``Large" part:
\begin{equation}
\label{eq:decomposition_flow_intro}
v_i=v_i^{(L)}+v_i^{(S)}.
\end{equation}
In a turbulent regime, the small component $v_i^{(S)}$ varies in time very rapidly compared to the larger one $v_i^{(L)}$.  Therefore, in some sense, $v_i^{(S)}$ models turbulent phenomena (for instance to thermal fluctuations if the reaction is related to combustion processes).
In practice, there is no efficient way to model the small scale component. Hence, the latter is often modeled as an approximation of white noise:
\begin{equation}
\label{eq:small_scale_ansatz}
v_i^{(S)}=\sum_{n\geq 1} b_{n,i} \dot{w}^n_t,
\end{equation}
where, $(w^n)_{n\geq 1}$ is a sequence of independent standard Brownian motions. In case of incompressible flows, one also has the divergence free-condition:
\begin{equation}
\label{eq:flow_div_condition}
\div\, b_{n,i}:=\sum_{j=1}^d  \partial_j b^j_{n,i}=0,\ \  \text{ for all }n\geq 1.
\end{equation}
Using the ansatz \eqref{eq:small_scale_ansatz} for the small scale behavior of $v_i=v_i^{(L)}+v_i^{(S)}$ in \eqref{eq:reaction_diffusion_det} one obtains \eqref{eq:reaction_diffusion_system_intro}.

The same heuristic argument can be used also in other contexts. For instance, in the case of the famous Lotka-Volterra equations (see e.g.\ \cite[Subsection 5.2]{AVreaction-global}), which model the dynamics of predatory-prey systems, the flow $v_i$ may model migratory phenomena of the $i$-th species. In particular, in \eqref{eq:decomposition_flow_intro} the term
$v^{(L)}_i$ takes into account the large scale movements of the $i$-th species while $v^{(S)}_i$ models small fluctuations of the movements due to local effects (e.g.\ unusual dryness of the fields, adverse weather events and local changes of the territory).

It is worth to mention two other properties of transport noise.
Assume that $u_i$ satisfies \eqref{eq:reaction_diffusion_system_intro} with $g_{n,i}\equiv 0$.
Firstly, if \eqref{eq:flow_div_condition} holds, then at least formally the \emph{total mass} $\sum_{ i=1}^{ \ell}\int_{\Tor^d} u_i(t,x)\,\dd x$  is controlled \emph{pathwisely} along the flow provided $\sum_{i=1}^{\ell} f_i(\cdot,u)\lesssim 1+\sum_{i=1}^{\ell} u_i$, which is typical in deterministic theory, see \eqref{eq:chemical_reaction_chemistry_formula} and \cite[(M)]{P10_survey}. To see this it is enough to take $\int_{\Tor^d} \cdot \,\dd x$ in the first equation of \eqref{eq:reaction_diffusion_system_intro}. Secondly, if the positivity preserving condition of the deterministic theory holds (see e.g.\ \cite[(P)]{P10_survey}), then also the flow induced by \eqref{eq:reaction_diffusion_system_intro} is positive preserving (see Theorem \ref{thm:positivity}). Here we do not need \eqref{eq:flow_div_condition}.

From a mathematical point of view, there is no reason to prefer the It\^o's formulation rather than a Stratonovich one in \eqref{eq:reaction_diffusion_system_intro}.
In our paper, we are able to deal with both situations as we will consider \eqref{eq:reaction_diffusion_system_intro} with $\ellip_i \Delta u_i$ replaced by $\div(a_i\cdot\nabla u_i)+ (r_i\cdot\nabla) u_i$. To see this it is enough to recall that (at least formally in case $g_{n,i}\equiv 0$)
\begin{equation}
\label{eq:stratonovich_correction}
(b_{n,i}\cdot \nabla)u_i \circ \dd w_t^n
=[\div(a_{b,i}\cdot\nabla u_i)+ (r_{b,i} \cdot\nabla)u_i]\, \dd t+ (b_{n,i}\cdot \nabla)u_i \,\dd w_t^n
\end{equation}
where $a_{b,i}:=(\frac{1}{2}\sum_{n\geq 1} b_{n,i}^jb_{n,i}^k)_{j,k=1}^d$ and
$r_{b,i}:=(-\frac{1}{2}\sum_{n\geq 1}(\div\,b_{n,i})b_{n,i}^j)_{j=1}^d$.

In the context of SPDEs, transport noise has attracted much attention in the last decades. Indeed, under structural assumptions on the $b_{n,i}$'s one can show that the solution $u$ to a certain SPDE has better properties than its deterministic counterparts. This phenomena is usually called \emph{regularization by noise}, see e.g.\ \cite{FGP10,F15_book} and the references therein for further details. Let us mention two situations where it occurs:
\begin{itemize}
\item Delayed blow-up \cite{FGL21a,FL19}.
\item Dissipation enhancement and/or stabilization \cite{FGL21_dissipation,GY21,Luo21}.
\end{itemize}

To the best of our knowledge, none of the above phenomena have been shown in the general context of reaction-diffusion equations. In the follow-up paper \cite{Agr22} first steps are being made by the first author.

\subsection{Scaling and criticality}
\label{sss:scaling}
Before we discuss our main results in more detail, it is instructive to analyze the scaling property of \eqref{eq:reaction_diffusion_system_intro} in the following ``toy'' situation:
$$
\ell=1,\quad f(u)=|u|^{h-1}u, \quad F(u)= e |u|^{\frac{h-1}{2}} u,\quad g_n(u)=\theta_n |u|^{\frac{h-1}{2}}u,
$$
where $h>1$, $e\in \R^d$ and $\theta=(\theta_n)_{n\geq 1}\in \ell^2$.
Note that the growth of $f$ and $F,g$ are related. Reasoning as \cite[Subsection 5.2.2]{AV19_QSEE_1}, one can see that solutions to \eqref{eq:reaction_diffusion_system_intro} are (locally) invariant under the mapping
\begin{equation*}
u\mapsto \lambda^{1/(h-1)}u(\lambda \cdot,\lambda^{1/2}\cdot), \qquad \lambda>0,
\end{equation*}
and that the Besov spaces $B^{\frac{d}{q}-\frac{2}{h-1}}_{q,p}(\Tor^d)$ and Lebesgue spaces $L^{\frac{d}{2}(h-1)}(\Tor^d)$ are (locally) invariant under the induced mapping on the initial data $u_0\mapsto\lambda^{1/(h-1)} u_0(\lambda^{1/2}\cdot)=:u_{0,\lambda}$.
More precisely, the homogeneous version of such spaces are invariant under the map $u_0\mapsto u_{0,\lambda}$:
$$
\|u_{0,\lambda}\|_{\dot{B}^{\frac{d}{q}-\frac{2}{h-1}}_{q,p}(\R^d)}
\eqsim \|u_{0}\|_{\dot{B}^{\frac{d}{q}-\frac{2}{h-1}}_{q,p}(\R^d)},
\quad
\|u_{0,\lambda}\|_{L^{\frac{d}{2}(h-1)}(\R^d)}
\eqsim \|u_{0}\|_{L^{\frac{d}{2}(h-1)}(\R^d)}.
$$
The \emph{Sobolev index} of the spaces $B^{\frac{d}{q}-\frac{2}{h-1}}_{q,p}$ and $L^{\frac{d}{2}(h-1)}$
is
 $-\frac{2}{h-1}$ and therefore it is independent of $d,q$ (and $p$ in case of Besov spaces).
This number will appear several times in the paper and it will gives distinction between the ``critical" and ``non-critical" situation.

Although the above choice seems very restrictive, the above can be thought of as a ``toy example'' for the case of $F,f,g$ with polynomial growth of order $h>1$, i.e.\ as $u\to \infty$
$$|F(u)|+\|(g_n(u))_{n\geq 1}\|_{\ell^2}\lesssim |u|^{\frac{h+1}{2}} \qquad \text{ and } \qquad |f(u)|\lesssim |u|^h.
$$

\subsection{Overview}
Below we give an overview of the results of the current paper.
In the manuscript we consider a (slightly) generalized version of \eqref{eq:reaction_diffusion_system_intro}, namely \eqref{eq:reaction_diffusion_system} below.
\begin{itemize}
\item Local well-posedness in critical spaces of \eqref{eq:reaction_diffusion_system} - see Theorem \ref{t:reaction_diffusion_global_critical_spaces} and Proposition \ref{prop:local_continuity}.
\item Instantaneous regularization of solutions to \eqref{eq:reaction_diffusion_system} -  see Theorems \ref{t:reaction_diffusion_global_critical_spaces} and \ref{t:high_order_regularity}.
\item (Sharp) blow-up criteria for  \eqref{eq:reaction_diffusion_system} -  see Theorem \ref{t:blow_up_criteria}.
\item Positivity of solutions to \eqref{eq:reaction_diffusion_system} - see Theorem \ref{thm:positivity}.

\item Global well-posedness for small initial data - see Theorems \ref{t:global_small_data}.
\end{itemize}
Although we formulate the main results only for $d\geq 2$, a detailed explanation on the simpler case $d=1$ case can be found in Section \ref{sec:1dcase}.  The special case $p=q=2$ is presented separately in Section \ref{sec:p=q=2} as it requires a different argument. Finally, there is an appendix on the maximum principle for scalar SPDEs in Appendix \ref{s:maximum}. The latter plays a crucial role in the positivity of the solution to \eqref{eq:reaction_diffusion_system}.

The proofs of the above results are based on our recent theory on stochastic evolution equations \cite{AV19_QSEE_1, AV19_QSEE_2}. It was already applied to stochastic Navier-Stokes equations \cite{AV20_NS} and a large class of SPDEs which fit into a variational setting \cite{AV22_variational}. The current paper is the first in a series of papers in which we apply our new framework to reaction-diffusion equations. In the companion papers \cite{AV22_quasi,AVreaction-global}, based on the analysis worked out in this paper, we prove global well-posedness results in several cases, and extend some of the results to the quasilinear case.
Finally, we mention that the local well-posedness and positivity results proven in the current paper have been already used by the first author in \cite{Agr22}  to prove delay of the blow–up of strong solutions and to establish an enhanced diffusion effect in presence of sufficiently intense transport noise.

\subsection{Notation}
Here we collect some notation which will be used throughout the paper. Further notation will be introduced where needed. We write $A \lesssim_P B$ (resp.\ $A \gtrsim_P B$) whenever there is a constant $C>0$ depending only on $P$  such that $A\leq C B$ (resp.\ $A\geq C B$). We write $C(P)$ if the constant $C$ depends only on $P$.

Let $p\in (1,\infty)$ and $\a\in (-1,p-1)$, we denote by $w_{\a}$ the \emph{weight} $w_{\a}(t)=|t|^{\a}$ for $t\in\R$.
For a Banach space $X$ and an interval $I=(a,b)\subseteq \R$, $L^p(a,b,w_{\a};X)$ denotes the set of all strongly measurable maps $f:I\to X$ such that
$$
\|f\|_{L^p(a,b,w_{\a};X)}:= \Big(\int_a^b \|f(t)\|_{X}^p  w_{\a}(t) \,\dd t\Big)^{1/p}<\infty.
$$
Furthermore, $W^{1,p}(a,b,w_{\a};X)\subseteq L^p(a,b,w_{\a};X)$ denotes the set of all $f$ such that $f'\in L^p(a,b,w_{\a};X)$ (here the derivative is taken in the distributional sense) and we set
$$
\|f\|_{W^{1,p}(a,b,w_{\a};X)}:=
\|f\|_{L^p(a,b,w_{\a};X)}+
\|f'\|_{L^p(a,b,w_{\a};X)}.
$$

Let $(\cdot,\cdot)_{\theta,p}$ and $[\cdot,\cdot]_{\theta}$ be the real and complex interpolation functor, respectively. We refer to  \cite{Analysis1,Tr1,BeLo} for details on interpolation and functions spaces. For each  $\theta\in (0,1)$, we set
$$
H^{\theta,p}(a,b,w_{\a};X) := [L^p(a,b,w_{\a};X),W^{1,p}(a,b,w_{\a};X)]_{\theta}.
$$
In the unweighted case, i.e.\ $\a=0$, we set $ H^{\theta,p}(a,b;X) :=H^{\theta,p}(a,b,w_{0};X) $ and similar.
For $\A\in \{L^p,H^{\theta,p},W^{1,p}\}$, we denote by $\A_{\loc}(a,b,w_{\a};X)$ (resp.\ $\A_{\loc}([a,b),w_{\a};X)$) the set of all strongly measurable maps $f:(c,d)\to X$ such that $f\in\A(c,d,w_{\a};X)$ for all $c,d\in(a,b)$ (resp.\ $f\in\A(a,c,w_{\a};X)$ for all $c\in (a,b)$).

The $d$-dimensional (flat) torus is denoted by $\Tor^d$ where $d\geq 1$.
For $K\geq 1$ and $\theta_1, \theta_2\in (0,1)$, $C^{\theta_1, \theta_2}_{\loc}((a,b)\times \Tor^d;\R^\ell)$ denotes the space of all maps $v:(a,b)\times \Tor^d\to \R^\ell$ such that for all $a<c<d<b$ we have
\[
|v(t, x) - v(t', x')|\lesssim_{c,d} |t-t'|^{\theta_1} + |x-x'|^{\theta_2}, \ \text{ for all }t,t'\in [c,d], \ x,x'\in \Tor^d.
\]
This definition is extended to $\theta_1,\theta_2\geq 1$ by requiring that the partial derivatives $\partial^{\alpha,\beta}v$ (with $\alpha\in \N$ and $\beta\in \N^d$) exist and are in $C^{\theta_1-|\alpha|, \theta_2-|\beta|}_{\loc}((a,b)\times\Tor^d;\R^{\ell})$ for all
$\alpha \leq \lfloor \theta_1 \rfloor$ and $\sum_{i=1}^d \beta_i\leq \lfloor \theta_2 \rfloor$.

We will also need the Besov spaces $B^{s}_{q,p}(\T^d;\R^\ell)$ and Bessel potential spaces $H^{s,q}(\T^d;\R^{\ell})$ to formulate our main results. These spaces can be defined by real and complex interpolation or more directly using Littlewood-Paley decompositions (see \cite[Section 6.6]{Saw} and \cite{SchmTr}).
Throughout the paper section, to abbreviate the notation, we often write $L^q,H^{s,q},B^{s}_{q,p}$ instead of $L^q(\Tor^d;\R^{\ell}),H^{s,q}(\Tor^d;\R^{\ell}),B^{s}_{q,p}(\Tor^d;\R^{\ell})$ if no confusion seems likely.

Finally we collect the main probabilistic notation. In the paper we fix a filtered probability space $(\O,\mathscr{A},(\F_t)_{t\geq 0}, \P)$ and we denote by $\E[\cdot]=\int_{\O}\cdot\,\dd \P$ the expected value. A map $\sigma:\O\to [0,\infty]$ is called a stopping time if $\{\sigma\leq t\}\in \F_t$ for all $t\geq 0$. For two stopping times $\sigma$ and $\tau$, we let
$$
[\tau,\sigma]\times \O:=\{(t,\om)\in [0,\infty)\times \O\,:\, \tau(\om)\leq t\leq \sigma(\om)\}.
$$
Similar definition holds for
$[\tau,\sigma)\times \O$,
$(\tau,\sigma)\times \O$ etc.
Finally, $\Progress$ denotes the progressive $\sigma$-algebra on the above mentioned probability space.

\subsubsection*{Acknowledgments}

The authors thank the referee and Udo B\"ohm for helpful comments.

\section{Statement of the main results}
\label{s:main_results}
In this section we state our main results on local well-posedness, regularity, blow-up criteria, and positivity for systems of reaction-diffusion equations on the $d$-dimensional torus $\Tor^d$. The results will be presented in a very flexible setting. This has the advantage that using the results of this paper, one can address global well-posedness issues in an efficient way by checking sharp blow-up criteria. Regularity and positivity often play a crucial role in dealing with these issues. As mentioned in the introduction, in the stochastic case there are many important cases in which global well-posedness is completely open. Using our new framework we are able to settle some of these problems in \cite{AVreaction-global}.

The proofs of the main results are postponed to Section \ref{s:proof_local_existence_regularity} and are based on our abstract framework developed in \cite{AV19_QSEE_1, AV19_QSEE_2}, and our recent maximal regularity estimates \cite{AV21_SMR_torus}.

Consider the following system of stochastic reaction-diffusion equations:
\begin{equation}
\label{eq:reaction_diffusion_system}
\left\{
\begin{aligned}
&\dd u_i -\div(a_i\cdot\nabla u_i) \,\dd t
= \Big[\div(F_i(\cdot, u)) +f_i(\cdot, u)\Big]\,\dd t \\
&\qquad \qquad \qquad \qquad \ \ \ \
+ \sum_{n\geq 1}  \Big[(b_{n,i}\cdot \nabla) u_i+ \rnoise_{n,i}(\cdot,u) \Big]\,\dd w_t^n, & \text{ on }\Tor^d,\\
&u_i(0)=u_{0,i},  & \text{ on }\Tor^d,
\end{aligned}
\right.
\end{equation}
where $i\in \{1,\dots,\ell\}$ and $\ell\geq 1$ is an integer.
Here $u=(u_i)_{i=1}^{\ell}:[0,\infty)\times \O\times \Tor^d\to \R^\ell$ is the unknown process, $(w^n)_{n\geq 1}$ is a sequence of standard independent Brownian motions on the above mentioned filtered probability space and
$$
\div (a_i\cdot\nabla u_i):=\sum_{j,k=1}^d \partial_j(a^{j,k}_i \partial_k u_i), \qquad
(b_{n,i}\cdot\nabla) u_i:=\sum_{j=1}^d b^j_{n,i} \partial_j u_i.
$$
As explained in the Subsection \ref{ss:transport_noise}, the coefficients $b_{n,i}^j$ model small scale turbulent effects; while the coefficients $a^{j,k}_i$
model inhomogeneous conductivity and may also take into account the It\^o correction in case of Stratonovich noise (see \eqref{eq:stratonovich_correction}). Note that the SPDEs \eqref{eq:reaction_diffusion_system} are coupled only through the nonlinearities $F$, $f$ and $g$, but there is no cross interactions in the diffusion terms $\div (a_i\cdot\nabla u_i)$ and $(b_{n,i}\cdot \nabla) u_i$, which is a standard assumption in reaction-diffusion systems.
The absence of cross-diffusion in the deterministic part $\div(a_i\cdot\nabla u_i)\,\dd t $ can be weakened in all the results below expect for Theorem \ref{thm:positivity}, where in its proof we argue component-wise.

Lower order terms in the leading differential operators in \eqref{eq:reaction_diffusion_system}  can be included as well. Since they can be modeled through the nonlinearities $F$, $f$, and $g$ as well, we do not have to write them explicitly.

\subsection{Assumptions and definitions}
\label{ss:reactionmain}
In this subsection we collect the main assumptions and definitions. Additional assumptions will be employed where needed. Below $\Borel$ and $\Progress$ denotes the Borel and the progressive $\sigma$-algebra, respectively. The space $H^{\alpha,q}(\T^d;Y)$ denotes the Bessel potential space with smoothness $\alpha$ and integrability $q$, defined on $\T^d$ with values in the Banach space $Y$.

\begin{assumption}
\label{ass:reaction_diffusion_global}
Let $d\geq 2$ and $\ell\geq 1$ be integers. We say that Assumption \ref{ass:reaction_diffusion_global}$(p,q,h,\s)$ holds if $p\in (2,\infty)$, $q\in [2,\infty)$, $h>1$, $\s\in [1, 2)$ and for all $i\in \{1,\dots,\ell \}$ the following hold:
\begin{enumerate}[{\rm(1)}]
\item\label{it:reaction_diffusion_global1} For each $j,k\in \{1,\dots,d\}$, $\am^{j,k}_i:\R_+\times \O\times \Tor^d\to \R$, $b_{i}^j:=(\bm^{j}_{n,i})_{n\geq 1}:\R_+\times \O\times \Tor^d\to \ell^2$ are $\Progress\otimes \Borel(\Tor^d)$-measurable.
\item\label{it:regularity_coefficients_reaction_diffusion}
There exist $N>0$ and $\alpha>\max\{\frac{d}{\rho},\s-1\}$ where $\rho \in [2,\infty)$
such that a.s.\ for all $t\in \R_+$ and $j,k\in \{1,\dots,d\}$,
\begin{align*}
\|\am^{j,k}_i(t,\cdot)\|_{H^{\alpha,\rho}(\Tor^d)}+\|(\bm^{j}_{n,i}(t,\cdot))_{n\geq 1}\|_{H^{\alpha,\rho}(\Tor^d;\ell^2)}
&\leq N.
\end{align*}
\item\label{it:ellipticity_reaction_diffusion} There exists $\ellip_i>0$ such that, a.s.\ for all $t\in \R_+$, $x\in \Tor^d$ and $\xi\in \R^d$,
$$
\sum_{j,k=1}^d \Big(a_i^{j,k}(t,x)-\frac{1}{2}\sum_{n\geq 1} b^j_{n,i}(t,x)b^k_{n,i}(t,x)\Big)
 \xi_j \xi_k
\geq  \ellip_i |\xi|^2.
$$
\item\label{it:growth_nonlinearities} For all $j\in \{1,\dots,d\}$,  the maps
\begin{align*}
F_i^j, \ f_i:\R_+\times \O\times \Tor^d\times \R\to \R,&\\
g_i:=(g_{n,i})_{n\geq 1}:\R_+\times \O\times \Tor^d\times \R\to \ell^2,&
\end{align*}
are $\Progress\otimes \Borel(\Tor^d)\otimes \Borel(\R)$-measurable. Set $F_i:=(F_i^j)_{j=1}^d$. Assume that
\begin{equation*}
F_i^j(\cdot,0), \ f_i(\cdot,0)\in L^{\infty}(\R_+\times \O\times \Tor^d),\quad
g_i(\cdot,0)\in L^{\infty}(\R_+\times \O\times \Tor^d;\ell^2),
\end{equation*}
and a.s.\ for all $t\in \R_+$, $x\in \Tor^d$ and $y\in\R$,
\begin{align*}
|f_i(t,x,y)-f_i(t,x,y')|
&\lesssim (1+|y|^{h-1}+|y'|^{h-1})|y-y'|,\\
|F_i(t,x,y)-F_i(t,x,y')|
&\lesssim (1+|y|^{\frac{h-1}{2}}+|y'|^{\frac{h-1}{2}})|y-y'|,
\\ \|g_i(t,x,y)-g_i(t,x,y')\|_{\ell^2}
&\lesssim (1+|y|^{\frac{h-1}{2}}+|y'|^{\frac{h-1}{2}})|y-y'|.
\end{align*}
\end{enumerate}
\end{assumption}

The parameters $p$ and $q$ will be used for temporal and spatial integrability, respectively. Finally, $\delta$ will be related to the order of smoothness of the underlined Sobolev space with integrability $q$.
Although we allow $\s\in [1,2)$, in applications to \eqref{eq:reaction_diffusion_system} it turns out to be enough to consider $\delta\in [1,\frac{h+1}{h})$, see Assumption \ref{ass:admissibleexp} below.

Note that Assumption \ref{ass:reaction_diffusion_global}\eqref{it:regularity_coefficients_reaction_diffusion} and Sobolev embeddings give
\begin{equation*}
\|a_i^{j,k}\|_{C^{\alpha-\frac{d}{\rho}}(\Tor^d)}
+
\|(b^{j}_{n,i})_{n\geq 1}\|_{C^{\alpha-\frac{d}{\rho}}(\Tor^d;\ell^2)}
\lesssim_{\alpha,d,\rho} N.
\end{equation*}

For future convenience, we collect some observations in the following remark.
\begin{remark}\
\label{r:basic_assumptions}
\begin{enumerate}[{\rm(a)}]
\item\label{it:enlarge_delta} If Assumption \ref{ass:reaction_diffusion_global}$(p,q,h,\delta)$ holds for some $\s\in [1, 2)$, then there exists an $\varepsilon>0$ such that it holds for all $\wt{\delta}\in [1, \delta+\varepsilon]$.
\item\label{it:enlarge_h} If Assumption \ref{ass:reaction_diffusion_global}$(p,q,h,\delta)$ holds for some $h$, then it holds for all $\wt{h}\in [h,\infty)$.
\item The growth of $f,F$ and $g$ is chosen in accordance with the scaling argument of Subsection \ref{sss:scaling}.
\item\label{it:d=1} The case $d=1$ is excluded in Assumption \ref{ass:reaction_diffusion_global} to avoid many subcases in our main results. However, it can be deduced by more direct methods (see Section \ref{sec:1dcase}), or from the $d=2$ case by adding a dummy variable (under some restrictions). Often one cannot identify any critical spaces in the case $d=1$.
\item The globally Lipschitz case $h=1$ is excluded in the above. Global well-posedness always holds in this case and can be derived from \cite[Theorem 4.15]{AV19_QSEE_2}. Similar to \eqref{it:d=1}, if $h=1$, then no critical spaces can be identified as no rescaling of solutions can (locally) preserve the structure of \eqref{eq:reaction_diffusion_system}.
\end{enumerate}
\end{remark}

Next we introduce the notion of solution to \eqref{eq:reaction_diffusion_system}. To stress the dependence on $(p,\kappa, \s,q)$ we will keep these parameters in the definition of solutions. The parameter $\a\geq 0$ is used for the power weight $w_{\kappa}(t)=t^{\kappa}$ in time. Finally, let us recall that the sequence $(w^n)_{n\geq 1}$ uniquely induces an $\ell^2$-cylindrical Brownian motion (see e.g.\ \cite[Definition 2.11]{AV19_QSEE_1}) given by $W_{\ell^2}(v):=\sum_{n\geq 1} \int_{\R_+} v_n \,\dd w^n_t$ where $v=(v_n)_{n\geq 1}\in L^2(\R_+;\ell^2)$.

\begin{definition}
\label{def:solution}
Suppose that Assumption \ref{ass:reaction_diffusion_global}$(p,q,h,\s)$ is satisfied for some $h>1$ and let $\a\in [0,\frac{p}{2}-1)$.
\begin{itemize}
\item Let $\sigma$ be a stopping time and let $u=(u_i)_{i=1}^{\ell}:[0,\sigma)\times \O\to H^{2-\s,q}(\Tor^d;\R^\ell)$ be a stochastic process.
We say that $(u,\sigma)$ is a {\em local $(p,\a,\s,q)$-solution} to \eqref{eq:reaction_diffusion_system} if there exists a sequence of stopping times $(\sigma_j)_{j\geq 1}$ such that the following hold for all $i\in \{1,\dots,\ell\}$:
\begin{itemize}
\item $\sigma_j\leq \sigma$ a.s.\ for all $j\geq 1$ and $\lim_{j\to \infty} \sigma_j =\sigma$ a.s.;
\item for all $j\geq 1$ the process $\one_{[0,\sigma_j]\times \O} u_i$ is progressively measurable;
\item a.s.\ for all $j\geq 1$ we have $u_i\in L^p(0,\sigma_j ,w_{\a};H^{2-\s,q}(\Tor^d))$ and
\begin{equation}
\label{eq:integrability_nonlinearity}
\begin{aligned}
\div(F_i(\cdot, u)) +f_i(\cdot, u)\in L^p(0,\sigma_j,w_{\a};H^{-\s,q}(\Tor^d)),&\\
(g_{n,i}(\cdot,u))_{n\geq 1}\in L^p(0,\sigma_j,w_{\a};H^{1-\s,q}(\Tor^d;\ell^2));&
\end{aligned}
\end{equation}
\item a.s.\ for all $j\geq 1$ the following identity holds for all $t\in [0,\sigma_j]$:
\begin{equation}
\label{eq:reaction_diffusion_global_stochastic_integrated_form}
\begin{aligned}
u_i(t)-u_{0,i}
&=\int_{0}^{t} \Big(\div(a_i\cdot\nabla u_i)+ \div(F_i(\cdot, u)) +f_i(\cdot, u)\Big)\,\dd s\\
& \  + \int_{0}^t\Big( \one_{[0,\sigma_j]}\big[ (b_{n,i}\cdot\nabla)u + g_{n,i}(\cdot,u) \big]\Big)_{n\geq 1} \,\dd W_{\ell^2}(s).
\end{aligned}
\end{equation}
\end{itemize}
\item Finally, $(u,\sigma)$ is called a {\em $(p,\a,\s,q)$-solution} to \eqref{eq:reaction_diffusion_system} if for any other local $(p,\a,\s,q)$-solution $(u',\sigma')$ to \eqref{eq:reaction_diffusion_system} we have $\sigma'\leq \sigma$ a.s.\ and $u=u'$ on $[0,\sigma')\times \O$.
\end{itemize}
\end{definition}

Note that a $(p,\a,\s,q)$-solution is \emph{unique} by definition. Later on in Proposition \ref{prop:comp} we will prove a further uniqueness result: a different choice of the coefficients $(p,\kappa,\delta,q,h)$ leads to the same solution.

All the integrals in \eqref{eq:reaction_diffusion_global_stochastic_integrated_form} are well-defined. To see this, fix $i\in \{1,\dots,\ell\}$. By Assumption \ref{ass:reaction_diffusion_global}\eqref{it:regularity_coefficients_reaction_diffusion}, \cite[Proposition 4.1]{AV21_SMR_torus} and $u_i\in L^p(0,\sigma_j,w_{\a};H^{2-\s,q}(\Tor^d))$ a.s.\ for all $j\geq 1$, we get
\begin{equation}
\label{eq:integrability_A_B_u}
\begin{aligned}
\div(a_i\cdot\nabla u_i)\in L^p(0,\sigma_j,w_{\a};H^{-\s,q}(\Tor^d)),&\\
((b_{n,i}\cdot\nabla) u_i)_{n\geq 1}\in L^p(0,\sigma_j,w_{\a};H^{1-\s,q}(\Tor^d;\ell^2)),&
\end{aligned}
\end{equation}
a.s.\ for all $j\geq 1$.
The deterministic integrals are well-defined as $H^{-\s,q}(\Tor^d)$-valued Bochner integrals. For the stochastic integrals, recall that
\begin{equation}
\label{eq:identity_gamma_H}
\g(\ell^2,H^{\zeta,r}(\Tor^d))= H^{\zeta,r}(\Tor^d;\ell^2),  \text{ for all }\zeta\in \R\text{ and }r\in (1,\infty),
\end{equation}
where $\g(\ell^2;X)$ denotes the set of all $\g$-radonifying operators with values in the Banach space $X$ (see \cite[Chapter 9]{Analysis2} and in particular Theorem 9.4.8 there for details). Therefore, due to \eqref{eq:integrability_nonlinearity} and \eqref{eq:integrability_A_B_u}, the stochastic integrals are well-defined as $H^{1-\s,q}(\Tor^d)$-valued stochastic integrals by \eqref{eq:identity_gamma_H}, \cite[Theorem 4.7]{NVW13} and $L^p(0,T;w_{\a})\embed L^2(0,T)$ since $\a<\frac{p}{2}-1$.

\subsection{Local well-posedness and regularity in critical spaces}

Before we state our main local well-posedness result for \eqref{eq:reaction_diffusion_system} in critical spaces, we first introduce the set of \emph{admissible} exponents $(p,q,h,\delta)$.
\begin{assumption}\label{ass:admissibleexp}
Let $d\geq 2$. We say that Assumption \ref{ass:admissibleexp}$(p,q,h,\delta)$ holds if
$h>1$, $\s\in [1,\frac{h+1}{h})$,
$p\in (2,\infty)$, and $q\in [2,\infty)$ satisfy
\begin{align}\label{eq:condqadm}
\frac{1}{p}+\frac{1}{2}\Big(\reg+\frac{d}{q}\Big) \leq \frac{h}{h-1}, \ \ \ \ \text{and} \ \ \  \
\frac{d}{d-\reg} <q<\frac{d(h-1)}{h+1-\reg(h-1)}.
\end{align}
\end{assumption}

In the above assumption we avoided the case $p=2$ since this is an exceptional case, which can be included provided $q=2$. The latter situation is discussed in Section \ref{sec:p=q=2}.
The following lemma characterizes for which exponents $h$ we can find $(p,q,\delta)$ such that Assumption \ref{ass:admissibleexp}$(p,q,h,\delta)$ holds. Recall that we may always enlarge $h$ if needed (see Remark \ref{r:basic_assumptions}\eqref{it:enlarge_h}).
\begin{lemma}\label{lem:admissibleexp}
Let $d\geq 2$ and set
\begin{equation}
\label{eq:hstar_reaction_diffusion}
\hstar := \left\{
             \begin{aligned}
               &3, \ \ & \text{ if } & \ d=2, \\
               &\frac{1}{2}+\frac{1}{d}+\sqrt{\Big(\frac{1}{2}+\frac{1}{d}\Big)^2+\frac{2}{d}},\ \  & \text{ if } & \ d\geq 3.
             \end{aligned}
           \right.
\end{equation}
Then there exist $(p,q,\delta)$ such that Assumption \ref{ass:admissibleexp}$(p,q,h,\delta)$ holds if and only if $h>\hstar$.
\end{lemma}
\begin{proof}
Since we can take $p$ as large as we want,
the first part of \eqref{eq:condqadm} is equivalent to $\frac{d(h-1)}{2h-\reg(h-1)}<q$.
Therefore, we can find admissible $(p,q,h,\delta)$ if and only if
there exist $\reg\in [1,\frac{h+1}{h})$ and $q\geq 2$ such that
\begin{align}\label{eq:rangeq}
\max\Big\{\frac{d}{d-\reg},\frac{d(h-1)}{2h-\reg(h-1)}\Big\} <q<\frac{d(h-1)}{h+1-\reg(h-1)}.
\end{align}
By elementary considerations one can see that the range of $q$'s in \eqref{eq:rangeq} is nontrivial if and only if  $h>\frac{d+1}{d-1}$. Since additionally $q\geq 2$, admissibility is equivalent to
\begin{align}\label{eq:admequiv}
h>\frac{d+1}{d-1} \ \ \text{and} \ \  \frac{d(h-1)}{h+1-\reg(h-1)}>2  \ \text{for some}  \ \reg\in \Big[1,\frac{h+1}{h}\Big).
\end{align}
Taking $\reg\uparrow \frac{h+1}{h}$, the second part of \eqref{eq:admequiv} becomes $h^2-(1+\frac{2}{d})h-\frac2d>0$, which is equivalent to
$h>\frac{1}{2}+\frac{1}{d}+\sqrt{\big(\frac{1}{2}+\frac{1}{d}\big)^2+\frac{2}{d}}=:\wt{h}_d$.
In case $d\geq 3$, one can check that $\frac{d+1}{d-1}\leq \wt{h}_d$. In case $d=2$, one has $3=\frac{d+1}{d-1}>\wt{h}_d$. Hence, admissibility is equivalent to $h>h_d= \wt{h}_d$ if $d\geq 3$, and $h>h_d=3$ if $d=2$.
\end{proof}

The numbers $\hstar$ in \eqref{eq:hstar_reaction_diffusion} are connected to the {\em Fujita exponent} $1+\frac{2}{d}$ introduced in the seminal paper \cite{F66} in the study of the \emph{blowing-up} of positive (smooth) solutions to the PDE: $\partial_t u-\Delta u=u^{1+h}$. In the next remark we compare this to our setting.

\begin{remark}[Stochastic Fujita exponent]\label{rem:stochFujita}
Note that $h_d$ in Lemma \ref{lem:admissibleexp} satisfies $1+\frac{2}{d}< h_d \leq 1+\frac{4}{d}$. In particular it is always larger than the classical Fujita exponent $1+\frac{2}{d}$ (note that $h>1+\frac{2}{d}$ corresponds to the fact that the scaling invariant space $L^{\frac{d}{2}(h-1)}$ has integrability $>1$). Moreover, $h_d$
is decreasing in $d$, $h_d\downarrow 1$ as $d\to \infty$, and
\[h_2 = 3,  \ h_3 = 2, \  h_4 \approx 1.781, \  h_5 \approx 1.643, \ h_6\approx 1.549.\]
In case $h\leq \hstar$, then one can still apply Theorem \ref{t:reaction_diffusion_global_critical_spaces} by using one of the following strategies:
\begin{itemize}
\item enlarge $h$ in Assumption \ref{ass:reaction_diffusion_global}, see Remark \ref{r:basic_assumptions}\eqref{it:enlarge_h};
\item add dummy variables to increase the dimension $d_2>d$ in order to have $h>h_{d_2}$ (here we are using that $\lim_{d\to \infty} \hstar=1$).
\end{itemize}

Via Theorem \ref{t:blow_up_criteria}
one can show non-explosion (in probability) on large time intervals for solutions to \eqref{eq:reaction_diffusion_system} in case of small initial data and admissible exponents without further conditions, see Section \ref{sec:globsmall}. Therefore, by Lemma \ref{lem:admissibleexp}, one can allow nonlinearities as in \cite{F66} for $h>\hstar$. Such a threshold $\hstar$ seems optimal for these results to hold in presence of a non-trivial transport noise term, i.e.\ $(b\cdot\nabla) u\,\dd w$. Therefore, it seems natural to call $\hstar$ the {\em stochastic Fujita exponent}.

Recently, there has been an increasing attention in extending \cite{F66} to the stochastic framework, see e.g.\ \cite{C09,C11,ChKh15, FLN19_explosion} and the references therein. In the latter works, equations on $\R^d$ are considered, but transport noise does not appear. In view of the scaling argument in Subsection \ref{sss:scaling}, we expect that the same stochastic Fujita exponent $h_d$ appears in the $\R^d$-case of \eqref{eq:reaction_diffusion_system}.
\end{remark}

The main result of this section is the following local existence and regularity for \eqref{eq:reaction_diffusion_system} in critical spaces, and it will be proved in Subsection \ref{sss:reaction_diffusion_local}. Recall that $B^{s}_{q,p}(\T^d;\R^\ell)$ denotes the Besov space with smoothness $s\in \R$, integrability $q$, and microscopic parameter $p$. To abbreviate notation we write $B^{s}_{q,p}$ and $H^{s,q}$ for the spaces $B^{s}_{q,p}(\T^d;\R^\ell)$ and $H^{s,q}(\T^d;\R^\ell)$.

\begin{theorem}[Local existence and uniqueness in critical spaces, and regularity]
\label{t:reaction_diffusion_global_critical_spaces}
Let Assumptions \ref{ass:reaction_diffusion_global}$(p,q,h,\s)$ and \ref{ass:admissibleexp}$(p,q,h,\s)$ be satisfied. Set $\a:=\a_{\crit}:=p\big(\frac{h}{h-1}-\frac{1}{2}(\reg+\frac{d}{q})\big)-1$. Then for any
\begin{equation}
\label{eq:initial_data_from_critical_space}
u_0\in L^0_{\F_0}(\O;B^{\frac{d}{q}-\frac{2}{h-1}}_{q,p}),
\end{equation}
the problem \eqref{eq:reaction_diffusion_system} has a (unique) $(p,\a_{\crit},\s,q)$-solution $(u,\sigma)$ such that
 $\sigma>0$ a.s.\ and
\begin{align}
\label{eq:regularity_u_reaction_diffusion_critical_spaces}
u&\in C([0,\sigma);B^{\frac{d}{q}-\frac{2}{h-1}}_{q,p}) \  \text{ a.s.\ }
\\ \label{eq:regularity_u_reaction_diffusion_critical_spaces2}
u&\in H^{\theta,p}_{{\rm loc}}([0,\sigma),w_{\a_{\crit}};H^{2-\s-2\theta,q}) \ \text{ a.s.\ for all }\theta\in [0,1/2).
\end{align}
Finally, $u$ instantaneously regularizes in space and time:
\begin{align}
\label{eq:reaction_diffusion_H_theta}
u&\in H^{\theta,r}_{\rm loc}(0,\sigma;H^{1-2\theta,\zeta})  \ \ \text{a.s.\ for all }\theta\in  [0,1/2), \ \  r,\zeta\in (2,\infty),\\
\label{eq:reaction_diffusion_C_alpha_beta}
u&\in C^{\theta_1,\theta_2}_{\rm loc}((0,\sigma)\times \Tor^d;\R^{\ell}) \ \ \text{a.s.\ for all }\theta_1\in  [0,1/2), \ \ \theta_2\in (0,1).
\end{align}
\end{theorem}

The standard set of initial data  in the theory of reaction-diffusion equations is $L^{\infty}(\Tor^d;\R^{\ell})$  (see e.g.\ \cite{P10_survey}), and it is always included as a special case in the above result (see Remark \ref{r:reaction_diffusion_critical_spaces}\eqref{it:L_xi_data}).

The regularity  \eqref{eq:reaction_diffusion_H_theta}-\eqref{eq:reaction_diffusion_C_alpha_beta} can be improved by imposing further smoothness conditions on $(a,b,F,f,g)$, but keeping the same space of initial data for $u_0$ (see Theorem \ref{t:high_order_regularity} below).
We will prove later on that if Theorem \ref{t:reaction_diffusion_global_critical_spaces} is applicable for two sets of exponents $(p,q,h,\s)$, then the corresponding solutions coincide, see Proposition \ref{prop:comp}.

For future reference, we collect several observations in the following remark.
\begin{remark}\
\label{r:reaction_diffusion_critical_spaces}
\begin{enumerate}[{\rm(a)}]
\item\label{it:Sob_index}
As we have seen in Subsection \ref{sss:scaling}, the space
 $B^{\frac{d}{q}-\frac{2}{h-1}}_{q,p}$ in \eqref{eq:initial_data_from_critical_space} has the right local scaling for \eqref{eq:reaction_diffusion_system}. For this reason it is often called \emph{critical} for \eqref{eq:reaction_diffusion_system}. It coincides with the abstract notion of criticality which will be considered in Section \ref{s:proof_local_existence_regularity}. Note that the \emph{Sobolev index} of the initial value space is $(\frac{d}{q}-\frac{2}{h-1})-\frac{d}{q}=-\frac{2}{h-1}$, which is independent of $q$ and $\s$. Moreover, by Sobolev embeddings
$$
 B^{\frac{d}{q}-\frac{2}{h-1}}_{q,p}(\Tor^d;\R^{\ell}) \embed B^{\frac{d}{r}-\frac{2}{h-1}}_{r,s}(\Tor^d;\R^{\ell}) \  \text{ for all }r\geq q \text{ and }s\geq p.
$$
\item\label{it:negative_smoothness}
The freedom in the choice of $\delta$ allows us to reduce the smoothness of the above critical spaces. Indeed, choosing $\s\uparrow \frac{h+1}{h}$ and letting $q\uparrow\frac{d(h-1)}{h+1-(\frac{h+1}{h})(h-1)}=\frac{dh(h-1)}{h+1}$ it follows that we can treat initial data with smoothness $\frac{d}{q} - \frac{2}{h-1}\downarrow -\frac{1}{h}$.
\item\label{it:L_xi_data} By increasing $h$ (see Remark \ref{r:basic_assumptions}\eqref{it:enlarge_h}) we can suppose that
\begin{equation*}
\text{ either }\quad h>1+\frac{4}{d}
 \quad \text{ or } \quad
\Big[ h= 1+\frac{4}{d},\ \text{ and }\    d\geq 3\Big].
\end{equation*}
Setting $q=\frac{d}{2}(h-1)$, Theorem \ref{t:reaction_diffusion_global_critical_spaces} gives local well-posedness for \eqref{eq:reaction_diffusion_system} for the important class of initial data in
\begin{equation*}
u_0\in L^{0}_{\F_0}(\O;L^{q}(\Tor^d;\R^{\ell})).
\end{equation*}
Indeed, even if Assumptions \ref{ass:reaction_diffusion_global}$(p,q,h,\s)$ and \ref{ass:admissibleexp}$(p,q,h,\delta)$ hold with $\delta=1$, they self-improve to some $\s>1$ (see Remark \ref{r:basic_assumptions}\eqref{it:enlarge_delta}) and $p\geq\max\{ q,\frac{2}{2-\s}\}$.
Thus since
$L^{q}\embed B^{0}_{q,p}
=B^{\frac{d}{q}-\frac{2}{h-1}}_{q,p},
$
local well-posedness with initial data from the space $L^0_{\F_0}(\O;L^q)$ follows from Theorem \ref{t:reaction_diffusion_global_critical_spaces}.
In the above setting, one can also prove that $u\in C([0,\sigma);L^q)$ a.s.\ by using the local continuity w.r.t.\ $u_0$ (see Proposition \ref{prop:local_continuity} below) and a stopped version of the arguments used in \cite[Proposition 6.3]{AVreaction-global} (see also the comments below its statement).
\end{enumerate}
\end{remark}

The next rather technical local continuity result will be used in the proof of positivity of solutions $(u,\sigma)$ (see Theorem \ref{thm:positivity} below).
\begin{proposition}[Local continuity]
\label{prop:local_continuity}
Let Assumptions \ref{ass:reaction_diffusion_global}$(p,q,h,\s)$ and \ref{ass:admissibleexp}$(p,q,h,\s)$ be satisfied. Set $\a:=\a_{\crit}:=p\big(\frac{h}{h-1}-\frac{1}{2}(\reg+\frac{d}{q})\big)-1$.
Assume that $u_0$ satisfies \eqref{eq:initial_data_from_critical_space} and let $(u,\sigma)$ be the $(p,\a_{\crit},\s,q)$-solution to \eqref{eq:reaction_diffusion_system} provided by Theorem \ref{t:reaction_diffusion_global_critical_spaces}. There exist constants $C_0,T_0,\varepsilon_0>0$ and stopping times $\sigma_0,\sigma_1\in (0,\sigma]$ a.s.\ for which the following assertion holds:

For each $v_0\in L^p_{\F_0}(\O;B^{\frac{d}{q}-\frac{2}{h-1}}_{q,p})$ with
$\E\|u_0-v_0\|_{B^{\frac{d}{q}-\frac{2}{h-1}}_{q,p}}^p\leq \varepsilon_0$,
the $(p,\a_{\crit},\s,q)$-solution $(v,\tau)$ to \eqref{eq:reaction_diffusion_system} with initial data $v_0$, has the property that there exists a stopping time $\tau_0\in (0,\tau]$ a.s.\ such that for all $t\in [0,T_0]$ and $\gamma>0$, one has
\begin{align}
\label{eq:local_continuity_1}
\P\Big(\sup_{r\in [0,t]}\|u(r)-v(r)\|_{B^{\frac{d}{q}-\frac{2}{h-1}}_{q,p}}\geq \gamma, \  \sigma_0\wedge \tau_0>t\Big)
&\leq \frac{C_0}{\gamma^p}
\E \|u_0-v_0\|_{B^{\frac{d}{q}-\frac{2}{h-1}}_{q,p}}^p,\\
\label{eq:local_continuity_2}
\P\Big(\|u-v\|_{L^p(0,t,w_{\a_{\crit}};H^{2-\s,q})} \geq\gamma, \  \sigma_0\wedge \tau_0>t\Big)
&\leq \frac{C_0}{\gamma^p}
\E \|u_0-v_0\|_{B^{\frac{d}{q}-\frac{2}{h-1}}_{q,p}}^p,\\
\label{eq:local_continuity_3}
\P(\sigma_0\wedge \tau_0\leq t)&\leq C_0
\big[\E \|u_0-v_0\|_{B^{\frac{d}{q}-\frac{2}{h-1}}_{q,p}}^p+ \P(\sigma_1\leq t)\big].
\end{align}
\end{proposition}
The stopping time $\tau_0$ depends on $(u_0, v_0)$. To some extend, the estimates \eqref{eq:local_continuity_1}-\eqref{eq:local_continuity_2} show that $(u,\sigma)$
depends continuously on the initial data $u_0$, while \eqref{eq:local_continuity_3} gives a measure of the size of the time interval on which the continuity estimates \eqref{eq:local_continuity_1}-\eqref{eq:local_continuity_2} hold.
The key point is that the right-hand side of \eqref{eq:local_continuity_3} depends on $v_0$, but not on $v$.
In particular, $\{\tau_0\leq t\}$ has small probability if $t\sim 0$ and $v_0$ is close to $u_0$.
We actually prove a slightly stronger result than Proposition \ref{prop:local_continuity}, see Remark \ref{r:refined_local_existence}.

\subsection{Blow-up criteria}
Below we state some \emph{blow-up criteria} for the solution to \eqref{eq:reaction_diffusion_system} provided by Theorem \ref{t:reaction_diffusion_global_critical_spaces}.
Roughly speaking, blow-up criteria ensure that, if there exists a fixed time $T>0$ such  that the stopping time $\sigma$ satisfies $\P(\sigma<T)>0$, then the norm of $u$ in an appropriate space explodes. Blow-up criteria are often used to prove that a certain solution $(u,\sigma)$ is \emph{global} in time, i.e.\ $\sigma=\infty$ a.s. In practice, to prove global existence, it is enough to prove that the norm of $u$ in the above mentioned function space \emph{cannot} explode. In our follow-up paper \cite{AVreaction-global}, we will use this strategy to prove that solutions provided by Theorem \ref{t:reaction_diffusion_global_critical_spaces} are \emph{global} in time in several situations. A version of such results for small initial data can be found in Section \ref{sec:globsmall}.

Blow-up criteria are most powerful when they are formulated in function spaces which are as rough (i.e.\ large) as possible. On the other hand, the regularity cannot be arbitrarily low, since at least the nonlinearities need to be well-defined.
Hence, from a scaling point of view it is natural to ask for blow-up criteria involving function spaces with Sobolev index $-\frac{2}{h-1}$, because such critical threshold (see Subsection \ref{sss:scaling}) is generically optimal for local and global well-posedness of (S)PDEs (see \cite[Section 2.2]{CriticalQuasilinear} for deterministic evidence on this). Our general theory from \cite{AV19_QSEE_2} leads to the following criteria which at the moment is the best we can expect with abstract methods.

\begin{theorem}[Blow-up criteria]
\label{t:blow_up_criteria}
Let the assumptions of Theorem \ref{t:reaction_diffusion_global_critical_spaces} be satisfied and let $(u,\sigma)$ be the $(p,\a_{\crit},\s,q)$-solution to \eqref{eq:reaction_diffusion_system}. Suppose that $p_0\in (2,\infty)$, $q_0\in [2,\infty)$, $h_0\geq h$, $\s_{0}\in [1,2)$  are such that Assumptions \ref{ass:reaction_diffusion_global}$(p_0,q_0,h_0,\s_0)$ and \ref{ass:admissibleexp}$(p_0,q_0,h_0,\s_0)$ hold.
Set
$$
\beta_0:=\frac{d}{q_0}-\frac{2}{h_0-1} \ \ \text{ and } \ \
\g_0:= \frac{d}{q_0}+\frac{2}{p_0}-\frac{2}{h_0-1}.
$$
Then for all $0<s<T<\infty$,
\begin{enumerate}[{\rm(1)}]
\item\label{it:blow_up_not_sharp}
$
\displaystyle{
\P\Big(s<\sigma<T,\, \sup_{t\in [s,\sigma)}\|u(t)\|_{B^{\beta_0}_{q_1,\infty}} <\infty\Big)=0
}
$
for all $q_1>q_0$.
\item\label{it:blow_up_sharp}
$
\displaystyle{
\P\Big(s<\sigma<T,\, \sup_{t\in [s,\sigma)}\|u(t)\|_{B^{\beta_0}_{q_0,p_0}}+
\|u\|_{L^{p_0}(s,\sigma;H^{\gamma_0,q_0})} <\infty\Big)=0.
}
$
\end{enumerate}
\end{theorem}
Note that the norms in the blow-up criteria are well-defined thanks to  \eqref{eq:reaction_diffusion_H_theta}-\eqref{eq:reaction_diffusion_C_alpha_beta} and $s>0$. In particular, the parameter $s$ makes it possible to consider rough initial data. It is possible to take $(p,q,\s,h) = (p_0,q_0,\s_0,h_0)$, but the extra flexibility turns out to be very helpful in proving global well-posedness.

The proof of Theorem \ref{t:blow_up_criteria} will be given in Subsection \ref{ss:proof_blow_up_criteria}. As a consequence we also obtain:
\begin{corollary}\label{cor:blow_up_criteria}
Let the assumptions of Theorem \ref{t:reaction_diffusion_global_critical_spaces} be satisfied and let $(u,\sigma)$ be the $(p,\a_{\crit},\s,q)$-solution to \eqref{eq:reaction_diffusion_system}.
Suppose that $p_0\in (2,\infty)$, $q_0\in [2, \infty)$, $h_0\geq \max\{h, 1+\frac{4}{d}\}$, $\s_{0}\in (1,2)$  are such that Assumptions \ref{ass:reaction_diffusion_global}$(p_0,q_0,h_0,\s_0)$ and \ref{ass:admissibleexp}$(p_0,q_0,h_0,\s_0)$ hold.
Let $\zeta_0 = \frac{d}{2}(h_0-1)$. The following hold for each $0<s<T<\infty$:
\begin{enumerate}[{\rm(1)}]
\setcounter{enumi}{\value{nameOfYourChoice}}
\item\label{it:blow_up_not_sharp_L}
If $q_0 = \zeta_0$, then for all $\zeta_1>q_0$
\[\P\Big(s<\sigma<T,\, \sup_{t\in [s,\sigma)}\|u(t)\|_{L^{\zeta_1}} <\infty\Big)=0.
\]
\item\label{it:blow_up_sharp_L}
If $q_0>\zeta_0$, $p_0\in \big(\frac{2}{\s_0-1},\infty\big)$, $p_0\geq q_0$, and $\frac{d}{q_0}+\frac{2}{p_0}=\frac{2}{h_0-1}$, then
\[
\P\Big(s<\sigma<T,\, \sup_{t\in [s,\sigma)}\|u(t)\|_{L^{\zeta_0}}+
\|u\|_{L^{p_0}(s,\sigma;L^{q_0})} <\infty\Big)=0.
\]
\end{enumerate}
\end{corollary}
Although Theorem \ref{t:blow_up_criteria} is more general, in the follow-up work \cite{AVreaction-global} on global well-posedness we mainly use Corollary \ref{cor:blow_up_criteria}.
Considering $T+\varepsilon$ instead of $T$ in Corollary \ref{t:blow_up_criteria}\eqref{it:blow_up_not_sharp_L} and letting $\varepsilon\downarrow 0$, we find
\begin{equation}
\label{eq:include_endpoint_T}
\P\Big(s<\sigma\leq T,\, \sup_{t\in [s,\sigma)}\|u(t)\|_{L^{\zeta_1}} <\infty\Big)=0.
\end{equation}
Note that \eqref{eq:include_endpoint_T} contains also information on the set $\{\sigma=T\}$.
The same also holds for Corollary \ref{t:blow_up_criteria}\eqref{it:blow_up_sharp_L} and the assertions in Theorem \ref{t:blow_up_criteria}.
Such variants of the blow-up criteria can be sometimes useful (see e.g.\ Theorem \ref{t:global_small_data}).

\begin{remark}\
\label{r:blow_up}
\begin{enumerate}[{\rm(a)}]
\item Keeping in mind the parabolic scaling, the spaces $L^{\infty}(s,\sigma;B^{\beta_0}_{q_0,\infty})$ and $L^{p_0}(s,\sigma;H^{\g_0,q_0})$ have (space-time) Sobolev index $-\frac{2}{h-1}$. Thus, from a scaling point of view, Theorem \ref{t:blow_up_criteria}\eqref{it:blow_up_sharp} is optimal, while \eqref{it:blow_up_not_sharp} is only sub-optimal. A similar remark holds for Corollary \ref{cor:blow_up_criteria}.
\item In Theorem \ref{t:blow_up_criteria}\eqref{it:blow_up_not_sharp} and Corollary \ref{cor:blow_up_criteria}\eqref{it:blow_up_not_sharp_L}, $p_0$ does not appear, and thus it can be taken arbitrarily large.
\item\label{it:blow_up_negative_smoothness}
 Choosing $q_0,p_0$ large enough and $\s_0>1$, one has $\beta_0,\g_0< 0$. Thus Theorem \ref{t:blow_up_criteria} yields blow-up criteria in Sobolev spaces of \emph{negative} smoothness. To see how far below zero one can get, as in Remark \ref{r:reaction_diffusion_critical_spaces}\eqref{it:negative_smoothness}, we take $\s_0\uparrow \frac{h_0+1}{h_0}$, $q_0\downarrow \frac{dh_0(h_0-1)}{h_0+1}$. This gives $\beta_0,\g_0\downarrow -\frac{1}{h_0}$.
%
 \item Under the assumptions of Theorem \ref{t:blow_up_criteria} for $p_0$ large enough (depending on $h$) one also has
\[\P\Big(s<\sigma<T,\, \|u\|_{L^{p_0}(s,\sigma;H^{\gamma_0,q_0})} <\infty\Big)=0 \ \ \ \text{for all} \ 0<s<T.\]
To prove this one can argue as in the proof of Theorem \ref{t:blow_up_criteria} below by using \cite[Theorem 4.11]{AV19_QSEE_2} instead. We leave the details to the reader.
\end{enumerate}
\end{remark}

\subsection{Positivity}
In this subsection we investigate the positive preserving property of the stochastic reaction-diffusion equations \eqref{eq:reaction_diffusion_system}. Existence of positive solutions to stochastic reaction-diffusion equations has been studied by many authors see e.g.\ \cite{Ass99,CPT16,CES13,Mar19,MS19} and the references therein.
To the best of our knowledge, positivity of solutions to \eqref{eq:reaction_diffusion_system} is not known in our setting (e.g.\ rough data, transport noise and $(t,\om)$-dependent coefficients).  Considering $(t,\om)$-dependence of the coefficients is also very useful in applications to \emph{quasilinear} SPDEs, in which case $a^{j,k}_i(t,\om,x):=A^{j,k}_i(u(t,\om,x))$ and $A^{j,k}_i(\cdot)$ is a nonlinear map. These applications will be considered in \cite{AV22_quasi}.

The strategy of proof which we use seems to be new in the stochastic setting, but folklore for deterministic reaction-diffusion equations. It is based on a linearization argument, and on the maximum principle. The stochastic version of the maximum principle we use is for \emph{linear} scalar SPDEs and due to \cite{Kry13} (see Lemma \ref{lem:maxprinciple} for a slight variation of the latter). To apply this to obtain positivity in the case of nonlinear systems, an essential ingredient is the instantaneous regularization \eqref{eq:reaction_diffusion_H_theta}-\eqref{eq:reaction_diffusion_C_alpha_beta} of solutions to \eqref{eq:reaction_diffusion_system} proven in Section \ref{sss:reaction_diffusion_local}.

Below we write $v\geq 0$ for $v\in \D'(\Tor^d)$ provided
\begin{equation*}
\l \varphi,v\r \geq 0 \text{ for all }\varphi\in \D(\Tor^d)\text{ such that }\varphi\geq 0\text{ on }\Tor^d.
\end{equation*}
If $v\in L^1(\Tor^d)$, then the above coincides with its natural meaning. Recall that positive distributions can be identified with finite positive measures.  For an $\R^\ell$-valued distribution $v=(v_i)_{i=1}^{\ell}\in \D'(\Tor^d;\R^{\ell})$, we say that $v\geq 0$ provided $v_i\geq 0$ for all $i\in \{1,\dots,\ell\}$.

Our main result on positivity is the following.
\begin{theorem}[Positivity]
\label{thm:positivity}
Let the assumptions of Theorem \ref{t:reaction_diffusion_global_critical_spaces} be satisfied.
Let $(u,\sigma)$ be the $(p,\a_{\crit},\s,q)$-solution to \eqref{eq:reaction_diffusion_system} provided by Theorem \ref{t:reaction_diffusion_global_critical_spaces}. Suppose that
$$
u_0\geq 0 \text{ a.s.,}
$$
and that there exist progressive measurable processes $c_1,\dots,c_{\ell} :\R_+\times \O\to \R$ such that for all $i\in \{1,\dots,\ell\}$, $n\geq 1$, $y=(y_{i})_{i=1}^{\ell}\in [0,\infty)^{\ell}$ and
a.e.\ on $\R_+\times \O\times \T^d$,
\begin{align}
\label{eq:positivity_f}
f_i(\cdot,y_1,\dots,y_{i-1},0,y_{i+1},\dots,y_{\ell})&\geq 0,\\
\label{eq:positivity_F}
F_{i}(\cdot,y_1,\dots,y_{i-1},0,y_{i+1},\dots,y_{\ell})&= c_i(\cdot),\\
\label{eq:positivity_g}
g_{n,i}(\cdot,y_1,\dots,y_{i-1},0,y_{i+1},\dots,y_{\ell})&=0.
\end{align}
Then a.s. for all $x\in \Tor^d$ and $t\in [0,\sigma)$, one has
$
u(t,x)\geq 0.
$
\end{theorem}

By  \eqref{eq:reaction_diffusion_C_alpha_beta}, the pointwise expression $u(t,x)$  is well-defined in the above. The condition \eqref{eq:positivity_f} is standard in the theory of (deterministic) reaction-diffusion equations (see e.g.\ \cite[eq.\ (1.7)]{P10_survey}), while \eqref{eq:positivity_g} is (almost) optimal since it excludes the additive noise case (in which case positivity \emph{cannot} be preserved). Condition \eqref{eq:positivity_F} might be new. For $\ell=1$ it holds trivially if $F$ is not depending on $x\in \T^d$. In case $\ell = 2$ it is for instance fulfilled for
\[F_i(t,\omega, x,y_1, y_2) = \psi_i(t,\omega, x) \phi_{i,1}(x, y_1) \phi_{i,2}(x, y_2)\]
if $\phi_{i,i}(x,0)=0$ and $\psi_i$ is $\Progress\otimes\Borel(\T^d)$-measurable.

The proof of Theorem \ref{thm:positivity} will be given in Section \ref{ss:positivityproof}. From the proof it will be clear that it is possible to replace $\Tor^d$ by a domain $\Dom\subseteq \R^d$ if one assumes Dirichlet boundary conditions (for instance), and $b_{n,i}|_{\partial\Dom}=0$.

\section{Proofs of the main results}
\label{s:proof_local_existence_regularity}

\subsection{Local well-posedness and regularity}
\label{sss:reaction_diffusion_local}
The aim of this subsection is to prove local well-posedness and smoothness of $(p,\a,\s,q)$-solutions to \eqref{eq:reaction_diffusion_system}. In particular, the next result contains Theorem \ref{t:reaction_diffusion_global_critical_spaces} as a special case.

\begin{proposition}[Local existence, uniqueness, and regularity]
\label{prop:reaction_diffusion_global}
Let Assumption \ref{ass:reaction_diffusion_global}$(p,q,h,\s)$ be satisfied. Suppose that
$q>\max\{\frac{d}{d-\reg},\frac{d(h-1)}{2h-\reg (h-1)}\}$
and that $\a\in [0,\frac{p}{2}-1)$ satisfies one of the following conditions:
\begin{align}
\label{eq:reaction_diffusion_globali}
q<\frac{d(h-1)}{\reg} \ \  \ \text{and} \ \ \ \frac{1+\a}{p}+\frac{1}{2}(\reg+\frac{d}{q})\leq \frac{h}{h-1};
\\ \label{eq:reaction_diffusion_globalii} q\geq\frac{d(h-1)}{\reg}  \ \ \ \text{and} \ \ \
\frac{1+\a}{p}\leq \frac{h}{h-1}(1-\frac{\reg}{2}).
\end{align}
Then for any
$u_0\in L^0_{\F_0}(\O;B^{2-\reg-2\frac{1+\a}{p}}_{q,p})$,
\eqref{eq:reaction_diffusion_system} has a (unique) $(p,\a,\s,q)$-solution satisfying a.s.\ $\sigma>0$ and for all $\theta\in [0,\frac{1}{2})$
\begin{equation}
\label{eq:regularity_u_reaction_diffusion_critical_spaces_1}
u\in H^{\theta,p}_{\rm loc}([0,\sigma),w_{\a};H^{2-\reg-2\theta,q})\cap C([0,\sigma);B^{2-\reg-2\frac{1+\a}{p}}_{q,p}).
\end{equation}
Moreover, $u$ instantaneously regularizes
\begin{align}
\label{eq:reaction_diffusion_H_theta_1}
u&\in H^{\theta,r}_{\rm loc}(0,\sigma;H^{1-2\theta,\zeta})   \ \ \text{a.s.\ for all }\theta\in  [0,1/2), \ \  r,\zeta\in (2,\infty),\\
\label{eq:reaction_diffusion_C_alpha_beta_1}
u&\in C^{\theta_1,\theta_2}_{\rm loc}((0,\sigma)\times \Tor^d;\R^{\ell}) \ \ \ \text{a.s.\ for all }\theta_1\in  [0,1/2), \ \ \ \theta_2\in (0,1).
\end{align}
\end{proposition}

The weight $\a$ is called {\em critical} if equality holds in the above condition on $\a$ in \eqref{eq:reaction_diffusion_globali} or \eqref{eq:reaction_diffusion_globalii}, i.e.
\begin{align*}
\text{in \eqref{eq:reaction_diffusion_globali}:} \ \ \a&=\a_{\crit}=p\Big(\frac{h}{h-1}-\frac{1}{2}\big(\s+\frac{d}{q}\big)\Big)-1,
\\ \text{in \eqref{eq:reaction_diffusion_globalii}:} \ \ \a&=\a_{\crit}=\frac{p h}{h-1} \Big(1-\frac{\s }{2}\Big) - 1.
\end{align*}
Moreover, the space of initial data $B^{2-\reg-2\frac{1+\a}{p}}_{q,p}$ is called critical as well. For details on criticality we refer to \cite[Section 4]{AV19_QSEE_1}. This explains the subscript `$\crit$' in Theorem \ref{t:reaction_diffusion_global_critical_spaces}. This abstract notion of criticality turns out to be the one that leads to scaling invariant space in many examples.

Before we prove the above result, let us first show how Theorem \ref{t:reaction_diffusion_global_critical_spaces} can be deduced from Proposition \ref{prop:reaction_diffusion_global}.

\begin{proof}[Proof of Theorem \ref{t:reaction_diffusion_global_critical_spaces}]
The upper bound $q<\frac{d(h-1)}{h+1-\reg(h-1)}$ and $\s< \frac{h+1}{h}$ imply
$q<\frac{d(h-1)}{\s}$.
In particular, this is the first case of Proposition \ref{prop:reaction_diffusion_global}. Thus, it remains to check the inequality
\begin{equation*}
\frac{1+\a}{p}+\frac{1}{2}\Big(\reg+\frac{d}{q}\Big)\leq \frac{h}{h-1}.
\end{equation*}
Since $\a_{\crit}=p\big(\frac{h}{h-1}-\frac{1}{2}(\s+\frac{d}{q})\big)-1$, the assumptions $\frac{1}{p}+\frac{1}{2}\Big(\reg+\frac{d}{q}\Big)\leq \frac{h}{h-1} $ and $q<\frac{d(h-1)}{h+1-\reg(h-1)}$ imply $\a_{\crit}\geq 0$ and $\a_{\crit}<\frac{p}{2}-1$, respectively. In other words $\a_{\crit}$ belongs to the admissible range $[0,\frac{p}{2}-1)$.
Hence, the assumptions of Theorem \ref{t:reaction_diffusion_global_critical_spaces} imply that Proposition \ref{prop:reaction_diffusion_global} is applicable with $\a=\a_{\crit}$. It remains to show that the space of initial data $u_0$ is the one claimed in Theorem \ref{t:reaction_diffusion_global_critical_spaces}. To this end, note that
$B^{2-\s-2\frac{1+\a_{\crit}}{p}}_{q,p}=B^{\frac{d}{q}-\frac{2}{h-1}}_{q,p}$ as desired.
\end{proof}

Next we prove Proposition \ref{prop:reaction_diffusion_global}.
The idea is to reformulate the system of SPDEs \eqref{eq:reaction_diffusion_system} as a stochastic evolution equations (SEE in the following) and then use the results in \cite{AV19_QSEE_1,AV19_QSEE_2}. To this end, we need two ingredients:
\begin{itemize}
\item Stochastic maximal $L^p(L^q)$-regularity for the linearized problem (see e.g.\ \cite[Section 3]{AV19_QSEE_1} for the definition);
\item Estimates for the nonlinearities.
\end{itemize}
Recently, we obtained stochastic maximal $L^p(L^q)$-regularity for second order systems on the $d$-dimensional torus \cite{AV21_SMR_torus}.
Required estimates for the nonlinearities will be formulated in Lemma \ref{l:estimate_nonlinearities} below.

Before we state the lemma we reformulate \eqref{eq:reaction_diffusion_system} as an SEE. To this end, throughout this subsection we set
\begin{equation}
\label{eq:def_X_theta}
X_0=H^{-\s,q}, \ \ \ X_1= H^{2-\s,q},\  \ \ \text{and} \ \ \ X_{\lambda}:=[X_0,X_1]_{\lambda}=H^{-\s+2\lambda,q},
\end{equation}
where $\lambda\in (0,1)$, and a.s.\ for all $t\in \R_+$, $v\in X_1$,
\begin{equation}
\begin{aligned}
\label{eq:ABFG_def}
\AS(t)v&= \div(a(t)\cdot \nabla v ),  &\qquad \BS(t)v&=\big((b_n(t)\cdot\nabla)v\big)_{n\geq 1},\\
\FS(t,v)&= \div(F(t,v))+f(t,v), &\qquad \GS(t,v)&= \big(\rnoise_n(t,v)\big)_{n\geq 1}.
\end{aligned}
\end{equation}
With the above notation, \eqref{eq:reaction_diffusion_system} can be rewritten as a \emph{semilinear} SEE on $X_0$:
\begin{equation}
\label{eq:SEE}
\left\{
\begin{aligned}
&\dd u- \AS(t)u \,\dd t = \FS(t,u)\,\dd t+ (\BS(t)u+ \GS(t,u))\,\dd W_{\ell^2}(t), \ \ \ t\in \R_+,\\
&u(0)=u_0,
\end{aligned}
\right.
\end{equation}
where $W_{\ell^2}$ is the $\ell^2$-cylindrical Brownian motion induced by $(w^n)_{n\geq 1}$, see the text before Definition \ref{def:solution}.
Recall that $\g(\ell^2,X_{1/2})=\g(\ell^2,H^{1-\s,q})=H^{1-\s,q}(\ell^2)$, cf.\ \eqref{eq:identity_gamma_H}.

\begin{lemma}
\label{l:estimate_nonlinearities}
Let Assumption \ref{ass:reaction_diffusion_global}$(p,q,h,\s)$ be satisfied.
Let $\FS,\GS$ be as in \eqref{eq:ABFG_def}.
Suppose that
$q>\max\{\frac{d}{d-\reg}\,,\,\frac{d(h-1)}{2h-\reg (h-1)}\}$. Set $\rho_1=h-1$, $\rho_2=\frac{h-1}{2}$ and
\begin{align*}
\beta_1&:=
\left\{
\begin{aligned}
&\frac{1}{2}\Big(\s+\frac{d}{q} \Big) \Big(1-\frac{1}{h} \Big),  \qquad  \ &\text{ if }& \
q< \frac{d(h-1)}{\s},\\
&\frac{\s}{2},  \qquad\ &\text{ if }&\  q\geq \frac{d(h-1)}{\s},
\end{aligned}
\right.
\\
\beta_2&:=
\left\{
\begin{aligned}
&\frac{1}{h+1}+\frac{1}{2}\Big(\s+\frac{d}{q} \Big) \frac{h-1}{h+1} &\text{ if }& \ q< \frac{d(h-1)}{2(\s-1)},\\
&\frac{\s}{2}, &\text{ if } &\ q\geq \frac{d(h-1)}{2(\s-1)}.
\end{aligned}
\right.
\end{align*}
Then $\beta_1,\beta_2\in (0,1)$ and for each $v,v'\in X_1$
\begin{align*}
\|\FS(\cdot,v)-\FS(\cdot,v')\|_{X_0}&\lesssim \textstyle{\sum}_{j\in \{1,2\}}  (1+\|v\|_{X_{\beta_j}}^{\rho_j}+ \|v'\|_{X_{\beta_j}}^{\rho_j}) \|v-v'\|_{X_{\beta_j}},\\
\|\FS(\cdot,v)\|_{X_0}&\lesssim \textstyle{\sum}_{j\in \{1,2\}}  (1+\|v\|_{X_{\beta_j}}^{\rho_j}) \|v\|_{X_{\beta_j}},\\
\|\GS(\cdot,v)-\GS(\cdot,v')\|_{\g(\ell^2,X_{1/2})}&\lesssim  (1+\|v\|_{X_{\beta_2}}^{\rho_2}+ \|v'\|_{X_{\beta_2}}^{\rho_2}) \|v-v'\|_{X_{\beta_2}},\\
\|\GS(\cdot,v)\|_{\g(\ell^2,X_{1/2})}&\lesssim   (1+\|v\|_{X_{\beta_2}}^{\rho_2})
\|v\|_{X_{\beta_2}}.
\end{align*}
\end{lemma}

Since $\beta_j<1$, the above result shows that $\FS$ and $\GS$ are \emph{lower-order} nonlinearities.

\begin{proof}
Since $f(\cdot,0),F^j(\cdot,0)\in L^{\infty}$ and $(g_{n,i}(\cdot,0))_{n\geq 1}\in L^{\infty}$ by Assumption \ref{ass:reaction_diffusion_global}\eqref{it:growth_nonlinearities}, it is enough to estimate the differences $\FS(\cdot,v)-\FS(\cdot,v')$ and $\GS(\cdot,v)-\GS(\cdot,v')$.
We break the proof into two steps.

\emph{Step 1: Estimate for $\FS$}. Let us write
$\FS=\FS_0+\FS_1$ where
$$
\FS_0(\cdot,v):=f(\cdot,v) \qquad \text{ and }\qquad \FS_1(\cdot,v)=\div(F(\cdot,v)).
$$

\emph{Substep 1a: Estimate for $\FS_0$}. By Assumption \ref{ass:reaction_diffusion_global}\eqref{it:growth_nonlinearities}, a.e.\ on $\R_+\times \O$ and for all $\vone,\vtwo\in X_1$,
\begin{equation}
\begin{aligned}
\label{eq:estimate_rone_reaction_diffusion}
\|\FS_0(\cdot,\vone)-\FS_0(\cdot,\vtwo)\|_{H^{-\reg,q}}
&\stackrel{(i)}{\lesssim} \|f(t,\cdot,\vone)-f(t,\cdot,\vtwo)\|_{L^{\xi}}	\\
&\lesssim \Big\|(1+|\vone|^{h-1}+|\vtwo|^{h-1})|\vone-\vtwo|\Big\|_{L^{\xi}}\\
&\stackrel{(ii)}{\lesssim} (1+\|\vone\|^{h-1}_{L^{h\xi}}+\|\vtwo\|^{h-1}_{L^{h\xi}})\|\vone-\vtwo\|_{L^{h\xi}}\\
&\stackrel{(iii)}{\lesssim} (1+\|\vone\|^{h-1}_{H^{\theta,q}}+\|\vtwo\|^{h-1}_{H^{\theta,q}})
\|\vone-\vtwo\|_{H^{\theta,q}},
\end{aligned}
\end{equation}
where in $(i)$ we used Sobolev embedding with $-\frac{d}{\xi}= -\reg-\frac{d}{q}$ and $q>\frac{d}{d-\reg}$ to ensure $\xi\in (1,\infty)$. Estimate $(ii)$ follows from H\"{o}lder's inequality. In $(iii)$ we used Sobolev embedding with $\theta-\frac{d}{q}\geq -\frac{d}{h\xi}$, and where we need $\theta<2-\s$ to ensure that $\Phi_0$ is of lower-order (see \eqref{eq:def_X_theta}). To choose $\theta$ we consider two cases:
\begin{itemize}
\item \emph{Case $q<\frac{d(h-1)}{\reg}$}. In this situation we set $\theta=\frac{d}{q}-\frac{d}{h\xi} = \frac{d(h-1)}{hq} -\frac{\delta}{h}>0$. Note that $\theta<2-\reg$ follows from the assumption $q>\frac{d(h-1)}{2h-\reg (h-1)}$;
\item \emph{Case $q\geq\frac{d(h-1)}{\reg}$}. Here we set $\theta=0$. Since $\reg<2$ by Assumption \ref{ass:reaction_diffusion_global}, we also have $\theta=0<2-\reg$.
\end{itemize}
In both of the above cases, $X_{\beta_1}=H^{\theta,q}$ (see \eqref{eq:def_X_theta}). Thus \eqref{eq:estimate_rone_reaction_diffusion} gives
\begin{equation}
\label{eq:F_0_estimate}
\|\FS_0(t,\cdot,\vone)-\FS_0(t,\cdot,\vtwo)\|_{X_0}\lesssim (1+\|\vone\|^{\rho_1}_{X_{\beta_1}}+\|\vtwo\|^{\rho_1}_{X_{\beta_1}})
\|\vone-\vtwo\|_{X_{\beta_1}}.
\end{equation}

\emph{Substep 1b: Estimate for $\FS_1$}. As in substep 1a, by Assumption \ref{ass:reaction_diffusion_global}\eqref{it:growth_nonlinearities} we have, a.e.\ on $\R_+\times \O$ and for all $\vone,\vtwo\in X_1$,
\begin{equation}
\begin{aligned}
\label{eq:estimate_rtwo_rnoise_reaction_diffusion}
\|\FS_1(\cdot,v)-\FS_1(\cdot,v')\|_{H^{-\reg,q}} &
\stackrel{(iv)}{\lesssim}
\|F(t,\cdot,\vone)-F(t,\cdot,\vtwo)\|_{L^{\eta}}\\
&
\lesssim
\Big\|(1+|\vone|^{\frac{h-1}{2}}+|\vtwo|^{\frac{h-1}{2}})|\vone-\vtwo|\Big\|_{L^{\eta}}\\
&
\stackrel{(v)}{\lesssim}
(1+\|\vone\|^{\frac{h-1}{2}}_{L^{\frac{h+1}{2}\eta}}+\|\vtwo\|^{\frac{h-1}{2}}_{L^{\frac{h+1}{2}\eta}})
\|\vone-\vtwo\|_{L^{\frac{h+1}{2}\eta}}\\
&
\stackrel{(vi)}{\lesssim}
(1+\|\vone\|^{\frac{h-1}{2}}_{H^{\phi,q}}+\|\vtwo\|^{\frac{h-1}{2}}_{H^{\phi,q}})
\|\vone-\vtwo\|_{H^{\phi,q}},
\end{aligned}
\end{equation}
where in $(iv)$ we used $\div:H^{1-\reg,q}\to H^{-\reg,q}$ boundedly, and Sobolev embedding with $-\frac{d}{\eta}=1-\reg-\frac{d}{q}$, where $\eta\in (1,q)$ since $q>\frac{d}{d-\reg}$. In $(v)$ we used H\"{o}lder's inequality, and in $(vi)$ the Sobolev embedding with $\phi\in [0,2-\reg)$ and $\phi-\frac{d}{q}\geq -\frac{2d}{\eta(h+1)}$.
As in substep 1a, to choose $\phi$ we distinguish two cases.
\begin{itemize}
\item \emph{Case $q<\frac{d(h-1)}{2(\reg-1)}$}. In this situation we have $\phi:=\frac{d}{q}-\frac{2d}{\eta(h+1)}=\frac{d}{q}\frac{h-1}{h+1}+2\frac{1-\reg}{h+1}>0$. Note that $\phi<2-\reg$ since $q>\frac{d(h-1)}{2h-\reg (h-1)}$ by assumption;
\item\label{it:case_not_sharp_rone} \emph{Case $q\geq \frac{d(h-1)}{2(\reg-1)}$}. Here we set $\phi=0$ and thus $\phi<2-\s$.
\end{itemize}
Again one can check $X_{\beta_2}=H^{\phi,q}$ in both cases. Thus \eqref{eq:estimate_rtwo_rnoise_reaction_diffusion} gives
\begin{equation}
\label{eq:F_1_estimate}
\|\FS_1(\cdot,\vone)-\FS_1(\cdot,\vtwo)\|_{X_0}
\lesssim (1+\|\vone\|^{\rho_2}_{X_{\beta_2}}+\|\vtwo\|^{\rho_2}_{X_{\beta_2}})
\|\vone-\vtwo\|_{X_{\beta_2}}.
\end{equation}
The required estimate for $\FS(\cdot,v)-\FS(\cdot,v')$ follows from \eqref{eq:F_0_estimate} and \eqref{eq:F_1_estimate}, which completes Step 1.

\emph{Step 2: Estimate for $\GS$}. Here we prove that $\GS(\cdot,v)-\GS(\cdot,v')$ satisfies the same bound of $\FS_1(\cdot,v)-\FS_1(\cdot,v')$ in \eqref{eq:estimate_rtwo_rnoise_reaction_diffusion}. Thus the required estimate for $\GS$ follows as in Substep 1b.
Indeed, a.e.\ on $\R_+\times \O$ and for all $\vone,\vtwo\in X_1$,
\begin{equation}
\label{eq:estimate_rnoise_reaction_diffusion_proof_local}
\begin{aligned}
\|\GS(\cdot,\vone)-\GS(\cdot,\vtwo)\|_{\g(\ell^2,X_{1/2})}
&\stackrel{(vii)}{\lesssim} \|\rnoise(t,\cdot,\vone)-\rnoise(t,\cdot,\vtwo)\|_{\g(\ell^2,L^{\eta})}\\
&\stackrel{(viii)}{\eqsim} \|\rnoise(t,\cdot,\vone)-\rnoise(t,\cdot,\vtwo)\|_{L^{\eta}(\ell^2)}\\
&\stackrel{(ix)}{\lesssim}\Big\|(1+|\vone|^{\frac{h-1}{2}}+|\vtwo|^{\frac{h-1}{2}})|\vone-\vtwo|\Big\|_{L^{\eta}},
\end{aligned}
\end{equation}
where in $(vii)$ we used Sobolev embeddings with $-\frac{d}{\eta}=1-\reg-\frac{d}{q}$, in $(viii)$ \eqref{eq:identity_gamma_H} and in $(ix)$ Assumption \ref{ass:reaction_diffusion_global}\eqref{it:growth_nonlinearities}. Comparing \eqref{eq:estimate_rnoise_reaction_diffusion_proof_local} with the second line in \eqref{eq:estimate_rtwo_rnoise_reaction_diffusion}, one can check that the claimed estimate for $\GS$ follows as in Substep 1b.
\end{proof}

Next we prove Proposition \ref{prop:reaction_diffusion_global}.
For the reader's convenience, the proof will be divided into two parts. In Part (A) we prove the existence of a $(p,\a,\s,q)$-solution to \eqref{eq:reaction_diffusion_system} with pathwise regularity as in \eqref{eq:regularity_u_reaction_diffusion_critical_spaces_1} and in Part (B) we prove \eqref{eq:reaction_diffusion_H_theta_1}-\eqref{eq:reaction_diffusion_C_alpha_beta_1}.

\begin{proof}[Proof of Proposition \ref{prop:reaction_diffusion_global} Part (A) -- Local existence and uniqueness]
We break the proof of Part (A) into two steps. Recall that $(\AS,\BS,\FS,\GS)$ are defined in \eqref{eq:ABFG_def}. In the following, we use the definition of criticality of \cite{AV19_QSEE_1} for the trace space of initial data (see e.g.\ \cite{ALV21} for details on trace theory)
\begin{equation}
\label{eq:Besov_spaces}
\Xap:= (X_0, X_1)_{1-\frac{1+\a}{p},p}  = (H^{-\s,q},H^{2-\s,q})_{1-\frac{1+\a}{p},p}= B^{2-\s-2\frac{1+\a}{p}}_{q,p},
\end{equation}
where we used \cite[Theorem 6.4.5]{BeLo}.

\emph{Step 1: The assumptions (HF) and (HG) of \cite[Section 4.1]{AV19_QSEE_1} hold with $(F,G)$ replaced by $(\FS,\GS)$. Moreover, the trace space $\Xap=B^{2-\s-2\frac{1+\a}{p}}_{q,p}$ is critical for \eqref{eq:reaction_diffusion_system} if and only if one of the following conditions holds:}
\begin{itemize}
\item $q<\frac{d(h-1)}{\reg}$ and $\frac{1+\a}{p}+\frac{1}{2}(\reg+\frac{d}{q})= \frac{h}{h-1}$;
\item $q\geq\frac{d(h-1)}{\reg}$ and $\frac{1+\a}{p}= \frac{h}{h-1}(1-\frac{\reg}{2})$.
\end{itemize}
To prove the claim of this step, by Lemma \ref{l:estimate_nonlinearities} it is suffices to show that
\begin{equation}
\label{eq:critical_condition_j}
\frac{1+\a}{p}\leq \frac{\rho_j+1}{\rho_j}(1-\beta_j) \ \ \ \text{ for }j\in \{1,2\},
\end{equation}
where $\rho_j,\beta_j$ are as in Lemma \ref{l:estimate_nonlinearities}. Note that $\frac{d(h-1)}{\s}<\frac{d(h-1)}{2(\s-1)}$ for all $h>1$ and $\s\in [1,2)$. Therefore, to check \eqref{eq:critical_condition_j}, we can split into the following three cases:
\begin{enumerate}[{\rm(a)}]
\item\label{it:sharp_case_local_existence_proof} \emph{Case $q< \frac{d(h-1)}{\s}$}. In this situation one can check that the inequalities in \eqref{eq:critical_condition_j} for $j\in \{1,2\}$ are equivalent to the following restriction:
\begin{equation*}
\frac{1+\a}{p}\leq \frac{h}{h-1}-\frac{1}{2}\Big(\s+\frac{d}{q}\Big).
\end{equation*}
\item \emph{Case $ \frac{d(h-1)}{\s}\leq q <\frac{d(h-1)}{2(\s-1)}$}. Then \eqref{eq:critical_condition_j} for $j\in \{1,2\}$ holds if and only if
\begin{equation*}
\frac{1+\a}{p}\leq \frac{h}{h-1} -\frac{1}{2}\Big(\s+\frac{d}{q}\Big),\  \ \text{ and }\ \
\frac{1+\a}{p}\leq \frac{h}{h-1}\Big(1-\frac{\s}{2}\Big).
\end{equation*}
Note that $q\geq \frac{d(h-1)}{\s}$ implies $\frac{h}{h-1}(1-\frac{\s}{2})\leq \frac{h}{h-1} -\frac{1}{2}(\s+\frac{d}{q})$. Therefore, it is enough to assume the second of the above conditions.
\item\label{it:blow_up_large_case} \emph{Case $ q\geq  \frac{d(h-1)}{2(\s-1)}$}. Then \eqref{it:sharp_case_local_existence_proof},  \eqref{eq:critical_condition_j} for $j\in \{1,2\}$ leads to the same condition
$$
\frac{1+\a}{p}\leq \frac{h}{h-1}\Big(1-\frac{\s}{2}\Big).
$$
\end{enumerate}
One can check that the conditions in the cases \eqref{it:sharp_case_local_existence_proof}-\eqref{it:blow_up_large_case} coincide with the one assumed in Proposition \ref{prop:reaction_diffusion_global}. Moreover, criticality holds if and only if the estimates in cases \eqref{it:sharp_case_local_existence_proof}-\eqref{it:blow_up_large_case} hold with equality.

\emph{Step 2: There exists a (unique) $(p,\a,\s,q)$-solution $(u,\sigma)$ to \eqref{eq:reaction_diffusion_system} such that
\begin{equation*}
u\in H^{\theta,p}_{\loc}([0,\sigma);H^{2\theta-\delta,q}) \text{ a.s.\ for all }\theta\in [0,1/2).
\end{equation*}
}
To prove existence and uniqueness for \eqref{eq:reaction_diffusion_system} we will apply \cite[Theorem 4.8]{AV19_QSEE_1}. Indeed, our notion of $(p,\a,\s,q)$-solution to \eqref{eq:reaction_diffusion_system} (see Definition \ref{def:solution}) is equivalent to the notion of $L^{p}_{\a}$-maximal local solution given in \cite[Definition 4.4]{AV19_QSEE_1} (see also \cite[Remark 5.6]{AV19_QSEE_2}). By \cite[Theorem 5.2 and Remark 5.6]{AV21_SMR_torus}, the linearized problem with leading operator $(A,B)$ (see \eqref{eq:ABFG_def}) has stochastic maximal $L^p$-regularity. More precisely, we have $(A,B)\in \mathcal{SMR}^{\bullet}_{p,\a}(T)$ for all $T\in (0,\infty)$ with $X_0=H^{-\s,q}$ and $X_1=H^{2-\s,q}$ (see \cite[Definition 3.5]{AV19_QSEE_1} for the definition). Now existence and uniqueness follows from Step 1 and \cite[Theorem 4.8]{AV19_QSEE_1}.
\end{proof}

In order to complete the proof of Proposition \ref{prop:reaction_diffusion_global} it remains to show the regularity results \eqref{eq:reaction_diffusion_H_theta_1}-\eqref{eq:reaction_diffusion_C_alpha_beta_1}. For this we will use our new bootstrap technique of \cite[Section 6]{AV19_QSEE_2}. The structure of the proof of the regularity will be follows:
\begin{itemize}
\item Bootstrap regularity in time via \cite[Proposition 6.8]{AV19_QSEE_2} (see Step 1a) and
 \cite[Corollary 6.5]{AV19_QSEE_2} (see Substep 1b).
\item Bootstrap integrability in space via \cite[Theorem 6.3]{AV19_QSEE_2} applied recursively considering \eqref{eq:reaction_diffusion_system} in the $(H^{-\s,q_j},H^{2-\s,q_j},r,\alpha)$-setting where $(q_k)_{k\geq 1}$ is a sequence of increasing numbers $q_k\uparrow \infty$ with $q_1=q$ (see Step 2).
\item
Bootstrap differentiability in space via \cite[Theorem 6.3]{AV19_QSEE_2} by shifting the scale from $Y_j=H^{2j-\s,q}$ to $\wh{Y}_j=H^{2j-1,q}$ (see Step 3).
\end{itemize}

In each of the steps in the proof below and without further mentioning it, we use the stochastic maximal $L^r_{w_{\alpha}}$-regularity result of \cite[Theorem 5.2 and Remark 5.6]{AV21_SMR_torus} for $(A,B)$. By Assumption \ref{ass:reaction_diffusion_global}  the latter holds on $X_0 = H^{-s, \zeta}$ and all $r\in (2, \infty)$, $\zeta\in [2, \infty)$, $\a\in [0,\frac{r}{2}-1)$, and $s$ such that $1\leq s\leq \delta+\gamma$, for some (small) $\gamma>0$.

\begin{proof}[Proof of Proposition \ref{prop:reaction_diffusion_global} Part (B) -- Instantaneous regularization \eqref{eq:reaction_diffusion_H_theta_1}-\eqref{eq:reaction_diffusion_C_alpha_beta_1}. ]
Let $(u,\sigma)$ denote the $(p,\a,\s,q)$-solution to \eqref{eq:reaction_diffusion_system} provided by Part (A).

\emph{Step 1: For all $r\in (2,\infty)$,}
\begin{equation}
\label{eq:u_regularity_step_4_regularity_in_time}
u\in \bigcap_{\theta\in [0,1/2)} H^{\theta,r}_{\loc}(0,\sigma;H^{2-\s-2\theta,q}), \ \text{ a.s. }
\end{equation}

The proof of \eqref{eq:u_regularity_step_4_regularity_in_time} consists of two sub-steps, where Step 1a is not needed if $\a>0$.

\emph{Step 1a: If $\a=0$, then \eqref{eq:u_regularity_step_4_regularity_in_time} holds for some $r>p$. }
Here we apply \cite[Proposition 6.8]{AV19_QSEE_2}. Let $(\beta_j)_{j\in \{1,2\}}$ be as in Lemma \ref{l:estimate_nonlinearities}. Note that  $\beta_1, \beta_2\in (0,1)$ and $p\in (2, \infty)$ under the assumption of Proposition \ref{prop:reaction_diffusion_global}. Fix  $r\in (p,\infty)$ and $\alpha\in (0,\frac{r}{2}-1)$ such that
\begin{equation}
\label{eq:choice_a_0_regularity}
\frac{1}{p}=\frac{1+\alpha}{r} , \qquad
\text{ and }
\qquad
\frac{1}{r}\geq \max_j \beta_j - 1+\frac{1}{p}.
\end{equation}
With the above choice, Step 1 of the proof of Proposition \ref{prop:reaction_diffusion_global} Part (A) ensures that \cite[Proposition 6.8]{AV19_QSEE_2} is applicable with $Y_j =X_j= H^{2-\s,q}$ and $(r,\alpha)$ as above. This yields \eqref{eq:u_regularity_step_4_regularity_in_time} for all $r\in (p,\infty)$ and $\alpha>0$ satisfying \eqref{eq:choice_a_0_regularity}.

\emph{Step 1b: \eqref{eq:u_regularity_step_4_regularity_in_time} holds for all $r\in (2,\infty)$.} Let
$$
\text{ either }\quad [(r,\alpha)=(p,\a) \text{ if }\a >0] \quad \text{ or }\quad
[(r,\alpha)\text{ as in Step 1a if }\a=0].
$$
In all cases $\alpha>0$.  Let $\wh{r}\in (2,\infty)$ be arbitrary and let $\wh{\alpha}\in [0,\frac{\wh{r}}{2}-1)$ be such that $\frac{1+\wh{\alpha}}{\wh{r}}<\frac{1+{\alpha}}{{r}}$. Set $Y_j:=H^{2j-\s,q}$ for $j\in \{0,1\}$. Combining Step 1 of the proof of Proposition Part (A) and Step 1a if $\a=0$, one can check that the assumptions of \cite[Corollary 6.5]{AV19_QSEE_2} hold and this yields the claim of Step 1b.

\emph{Step 2: For all $r,\zeta\in (2,\infty)$,}
\begin{equation*}
u\in \bigcap_{\theta\in [0,1/2)} H^{\theta,r}_{\loc}(0,\sigma;H^{2-\s-2\theta,\zeta}) \ \text{ a.s. }
\end{equation*}

It suffices to consider $r\in (p,\infty)$ such that $\frac{1}{r}+\frac{1}{2}(\reg+\frac{d}{q})< \frac{h}{h-1}$. The latter condition is nonempty since in each of the case of \eqref{eq:reaction_diffusion_globali} and \eqref{eq:reaction_diffusion_globalii} one can check that $\frac{1}{2}(\reg+\frac{d}{q})< \frac{h}{h-1}$.
To prove the above claim for $u$ it is enough to show the existence of $\gap>0$ depending only $(r,\s,q,h,d)$ such that for all $\zeta\in [q,\infty)$,
\begin{equation}
\label{eq:implication_step_5_reaction_diffusion}
u\in \bigcap_{\theta\in [0,1/2)} H^{\theta,r}_{\loc}(0,{\sigma};H^{2-\s-2\theta,\zeta}) \text{ a.s.}
\ \Longrightarrow  \
u\in \bigcap_{\theta\in [0,1/2)} H^{\theta,r}_{\loc}(0,{\sigma};H^{2-\s-2\theta,\zeta+\gap})\text{ a.s.}
\end{equation}
Indeed, by Step 1 we know that the RHS\eqref{eq:implication_step_5_reaction_diffusion} holds with $\zeta =q$ and $r$ as above. Thus the claim of this step follows by iterating \eqref{eq:implication_step_5_reaction_diffusion}.

To prove \eqref{eq:implication_step_5_reaction_diffusion} suppose that $u\in \bigcap_{\theta\in [0,1/2)} H^{\theta,r}_{\loc}(0,{\sigma};H^{2-\s-2\theta,\zeta}) \text{ a.s.}$ We will apply \cite[Theorem 6.3]{AV19_QSEE_2}.
Since $\frac{1}{r}+\frac{1}{2}(\reg+\frac{d}{q})< \frac{h}{h-1}$ by assumption, there exists $\alpha>0$ (depending only on $(r,\s,q,h,d)$) such that $\frac{1+\alpha}{r}+\frac{1}{2}(\reg+\frac{d}{q})< \frac{h}{h-1}$.
By Step 1 of the proof of Proposition \ref{prop:reaction_diffusion_global} Part (A) we know that (HF) and (HG) of \cite[Section 4.1]{AV19_QSEE_1} hold in the $(H^{-\s,\zeta},H^{2-\s,\zeta},\alpha,r)$-setting with $\zeta\in [q,\infty)$, and the corresponding trace space is \emph{not} critical for \eqref{eq:reaction_diffusion_system} in this setting. Next we check the assumptions of \cite[Theorem 6.3]{AV19_QSEE_2} with the choice
$$
Y_i =H^{2j-\s,\zeta},\qquad \wh{Y}_i =H^{2j-\s,\zeta+\gap}, \qquad
r=\wh{r}, \qquad \alpha=\wh{\alpha}
$$
where $\gap$ will be chosen below.
It is easy to see that conditions (1) and (2) of \cite[Theorem 6.3]{AV19_QSEE_2} are satisfied. To check condition (3) of
\cite[Theorem 6.3]{AV19_QSEE_2} note that $\wh{Y}_i\embed Y_i$ and the assumption \cite[(6.1)]{AV19_QSEE_2} is satisfied due to \cite[Lemma 6.1(1)]{AV19_QSEE_2}. It remains to check
\begin{equation}
\label{eq:trace_embedding}
Y_{r}\embed \wh{Y}_{\wh{\alpha},\wh{r}}=\wh{Y}_{\alpha,r}.
\end{equation}
The latter will require $\gap$ to be small enough. Recall that $Y_{r}=B^{2-\s-\frac{2}{r}}_{\zeta,r}$ and
$\wh{Y}_{\alpha,r}=B^{2-\s-2\frac{1+\alpha}{r}}_{\zeta+\gap,r}$ by \eqref{eq:Besov_spaces}. By Sobolev embedding \eqref{eq:trace_embedding} holds provided
\begin{equation}
\label{eq:Sob_embedding_index_gap}
2-\s-\frac{2}{r}-\frac{d}{\zeta}\geq 2-\s-2\frac{1+\alpha}{r}-\frac{d}{\zeta+\gap}
\quad \Leftrightarrow\quad
\frac{1}{\zeta}-\frac{1}{\zeta+\gap}\leq \frac{2\alpha}{dr}.
\end{equation}
For \eqref{eq:Sob_embedding_index_gap} we can for instance take
$\gap=\frac{2\alpha}{d r}>0$.

\emph{Step 3: For all $r,\zeta\in (2,\infty)$,}
\begin{equation}
\label{eq:regularization_final_step}
u\in  \bigcap_{\theta\in [0,1/2)} H^{\theta,r}_{\loc}(0,\sigma;H^{1-2\theta,\zeta}) \ \text{ a.s. }
\end{equation}
Note that, if $\s=1$, then \eqref{eq:regularization_final_step} follows from Step 2. Thus below we may assume $\s\in (1,2)$. It suffices to prove \eqref{eq:regularization_final_step} for $r$ and $\zeta$ large. Therefore, we may suppose that
\[\zeta\geq \max\Big\{\frac{d(h-1)}{\delta},q\Big\}, \ \ r>\max\Big\{p,\frac{2}{2-\delta}\Big\}, \  \ \text{and} \ \ \frac{1}{r} + \frac{\delta-1}{2}<\frac{h}{2(h-1)}.\]
For the latter note that $\frac{\delta-1}{2}<\frac{h}{2(h-1)}$ always holds.

As in the previous step, we use \cite[Theorem 6.3]{AV19_QSEE_2} to improve the differentiability in space. To prove \eqref{eq:regularization_final_step},
for $j\in \{0,1\}$, we let
\begin{equation}
\label{eq:Y_wh_Y_final_regularization}
Y_j = H^{2j-\s,\zeta}, \quad \wh{Y}_j = H^{2j-1,\zeta}, \quad \wh{r} =r, \quad \alpha=0, \quad
\wh{\alpha}= \frac{r(\s-1)}{2}.
\end{equation}
Moreover, $\wh{\alpha}\in [0,\frac{\wh{r}}{2}-1)$ since $\wh{r}>\frac{2}{2-\s}$. %

We claim that and \eqref{eq:reaction_diffusion_globalii} is satisfies in the $(Y_0,Y_1,r,\alpha)$-setting and $(\wh{Y}_0,\wh{Y}_1,\wh{r},\wh{\alpha})$-setting, and both are not critical. Indeed, for $Y$ and $\wh{Y}$ this follows from
\[\frac{1}{r}<1-\frac{\delta}{2}<\frac{h}{h-1}\Big(1-\frac{\delta}{2}\Big) \ \ \text{and} \ \ \frac{1+\wh{\alpha}}{\wh{r}} = \frac{1}{r} + \frac{\delta-1}{2}< \frac{h}{2(h-1)}, \]
respectively.
To apply \cite[Theorem 6.3]{AV19_QSEE_2} it remains to check condition (3) there, which states
\begin{equation}
\label{eq:final_check_regularization_reaction_diffusion}
\text{ (a) } \ \
Y_r^{\Tr}\embed \wh{Y}_{\wh{\alpha},\wh{r}},
\quad \text{ and }\quad \text{ (b)} \ \
\text{ \cite[(6.1)]{AV19_QSEE_2} holds}.
\end{equation}
The choice of $\wh{\alpha}$ in \eqref{eq:Y_wh_Y_final_regularization} immediately yields \eqref{eq:final_check_regularization_reaction_diffusion}$_{\text{(a)}}$ and both spaces equal $B^{2-\delta-\frac2r}_{\zeta,r}$. To check \eqref{eq:final_check_regularization_reaction_diffusion}$_{\text{(b)}}$
we apply \cite[Lemma 6.2(4)]{AV19_QSEE_2}. To this end note that, for $\varepsilon=\frac{\s-1}{2}$,
$$
\wh{Y}_{1-\varepsilon}=
[\wh{Y}_0,\wh{Y}_1]_{1-\varepsilon}
=Y_1, \quad  \wh{Y}_0= [Y_0,Y_1]_{\varepsilon}=Y_{\varepsilon},  \ \ \text{ and }\ \ \frac{1+\wh{\alpha}}{r}=\varepsilon+\frac{1}{r}.
$$
Since $\alpha=0$ and $\varepsilon<\frac{1}{2}-\frac{1}{r}$ by construction, \cite[Lemma 6.2(4)]{AV19_QSEE_2} applies and thus \eqref{eq:final_check_regularization_reaction_diffusion}$_{\text{(b)}}$ follows. Hence \cite[Theorem 6.3]{AV19_QSEE_2} yields \eqref{eq:regularization_final_step}.

\emph{Step 4: Conclusion}. Note that
\eqref{eq:reaction_diffusion_H_theta_1} is equivalent to \eqref{eq:regularization_final_step}. In addition,   \eqref{eq:reaction_diffusion_C_alpha_beta_1} follows from \eqref{eq:reaction_diffusion_H_theta_1} and Sobolev embedding. Hence the proof of Proposition \ref{prop:reaction_diffusion_global} is completed.
\end{proof}

Next we turn to the local continuity result.
\begin{proposition}[Local continuity]
\label{prop:local_continuity_general}
Let the assumptions of Proposition \ref{prop:reaction_diffusion_global} be satisfied. Let $(u,\sigma)$ be the  $(p,\a,\s,q)$-solution to \eqref{eq:reaction_diffusion_system}.
Then there exist positive constants $(C_0,T_0,\varepsilon_0)$ and stopping times $\sigma_0,\sigma_1$ such that $\sigma_0,\sigma_1\in (0,\sigma]$ a.s.\ for which the following assertion holds:

For each $v_0\in L^p_{\F_0}(\O;B^{2-\reg-2\frac{1+\a}{p}}_{q,p})$ with
$\E\|u_0-v_0\|_{B^{2-\reg-2\frac{1+\a}{p}}_{q,p}}^p\leq \varepsilon_0$,
the $(p,\a,\s,q)$-solution $(v,\tau)$ to \eqref{eq:reaction_diffusion_system} with initial data $v_0$ has the property that there exists a stopping time $\tau_0\in (0,\tau]$ a.s.\ such that for all $t\in [0,T_0]$ and $\gamma>0$, one has
\begin{align}
\label{eq:local_continuity_1_prop}
\P\Big(\sup_{r\in [0,t]}\|u(r)-v(r)\|_{B^{2-\s-2\frac{1+\a}{p}}_{q,p}}\geq \gamma, \  \sigma_0\wedge \tau_0>t\Big)
&\leq \frac{C_0}{\gamma^p}
\E \|u_0-v_0\|_{B^{2-\s-2\frac{1+\a}{p}}_{q,p}}^p,\\
\label{eq:local_continuity_2_prop}
\P\Big(\|u-v\|_{L^p(0,t,w_{\a};H^{2-\s,q})}\geq \gamma, \  \sigma_0\wedge \tau_0> t\Big)
&\leq \frac{C_0}{\gamma^p}
\E \|u_0-v_0\|_{B^{2-\s-2\frac{1+\a}{p}}_{q,p}}^p,\\
\label{eq:local_continuity_3_prop}
\P(\sigma_0\wedge \tau_0\leq t)\leq C_0
\big[\E \|u_0-v_0\|_{B^{2-\reg-2\frac{1+\a}{p}}_{q,p}}^p& \ + \P(\sigma_1\leq  t)\big].
\end{align}
\end{proposition}

Let us first show that Proposition \ref{prop:local_continuity} is included.
\begin{proof}[Proof of Proposition \ref{prop:local_continuity}]
The claim follows from Proposition \ref{prop:local_continuity_general} with the choice of $\a_{\crit}=p(\frac{h}{h-1}-\frac{1}{2}(\s+\frac{d}{q}))-1$, as in the proof of Theorem \ref{t:reaction_diffusion_global_critical_spaces}.
\end{proof}

\begin{proof}[Proof of Proposition \ref{prop:local_continuity_general}]
For the proof of Proposition \ref{prop:local_continuity_general} we need some of the arguments in the abstract local well-posedness result of \cite[Theorem 4.5]{AV19_QSEE_1} (see also \cite[Theorem 4.8]{AV19_QSEE_1}). Let $\xi\in W^{1,\infty}(\R)$ be such that $\xi|_{[0,1]}=1$, $\xi|_{[2,\infty)}=0$ and $\xi$ is linear on $[1,2]$.

For $\lambda>0$ consider the following truncated version of \eqref{eq:SEE}:
\begin{equation}
\label{eq:SEE_truncated}
\left\{
\begin{aligned}
&\dd u- \AS(t)u \,\dd t = \xi_{\lambda}(t,u)\FS(t,u)\,\dd t+ (\BS(t)u+\xi_{\lambda}(t,u) \GS(t,u))\, \dd W_{\ell^2}(t), \ \ \ t\in \R_+,\\
&u(0)=u_0,
\end{aligned}
\right.
\end{equation}
where
\begin{equation}
\label{eq:xi_cut_off_semilinear}
\xi_{\lambda}(t,u):=\xi\Big(\frac{1}{\lambda} \|u\|_{\X(t)}\Big)
\end{equation}
where $\X$ is as in \cite[eq.\ (4.14)]{AV19_QSEE_1} with $(\rho_j,\beta_j)$ as in Lemma \ref{l:estimate_nonlinearities}  and $\varphi_j=\beta_j$.
For the choice of the cut--off in \eqref{eq:xi_cut_off_semilinear} we also uses \cite[Remark 4.14]{AV19_QSEE_1} and that the implicit constants in estimates of Lemma \ref{l:estimate_nonlinearities} are independent of $v,v'$.
As noticed in Step 2 of the proof of Proposition \ref{prop:reaction_diffusion_global} Part (A), the $(p,\a,\s,q)$-solution of \eqref{eq:reaction_diffusion_system} is the $L^p_{\a}$-maximal solution in the terminology of \cite[Definition 4.4]{AV19_QSEE_1} with the choice \eqref{eq:def_X_theta}-\eqref{eq:ABFG_def}. Recall that $\Xap$ has been defined in \eqref{eq:Besov_spaces}.
Now Steps 1--2 in the proof of \cite[Theorem 4.5]{AV19_QSEE_1} show the existence of constants $(\lambda_0,T_0,\varepsilon_0)$ for which the following assertion holds: For all $v_0\in L^p_{\F_0}(\O;\Xap)$ such that $\E\|u_0-v_0\|_{\Xap}^p\leq \varepsilon_0$ there exists a local $(p,\a,\s,q)$-solution $(\vv,T_0)$ to \eqref{eq:SEE_truncated} with initial data $v_0$ and $\lambda=\lambda_0$ satisfying
\begin{equation}
\label{eq:stability_estimate_truncated_equation}
\E\|\uu-\vv\|_{C([0,T_0];\Xap)}^p+ \E\|\uu-\vv\|_{L^p(0,T_0,w_{\a};H^{2-\s,q})}^p\leq C_0\E\|u_0-v_0\|_{\Xap}^p,
\end{equation}
where $(\uu,T_0)$ is the local $(p,\a,\s,q)$-solution to \eqref{eq:SEE_truncated} with $\lambda=\lambda_0$ and initial data $u_0$.

As in Step 4 of \cite[Theorem 3.5]{AV19_QSEE_1}, we set
\begin{align}
\label{eq:def_sigma_0}
\sigma_0 &:=\inf\big\{t\in[0,T_0]\,:\, \|\uu\|_{\X(t)}\geq \lambda_0\big\}, \ \ \ \text{and} \  \ \tau_0:= \inf\big\{t\in[0,T_0]\,:\, \|\vv\|_{\X(t)}\geq \lambda_0\big\}.
\end{align}
Note that a.s.\ $\sigma_0>0$ and  $\tau_0  >0$.
Arguing as in \cite[Step 4]{AV19_QSEE_1}, one can check that $(\uu|_{[0,\sigma_0)\times \O},\sigma_0)$ (resp.\ $(\vv|_{[0,\tau_0)\times \O},\tau_0)$) is a local $(p,\a,\s,q)$-solution to \eqref{eq:reaction_diffusion_system} with initial data $u_0$ (resp.\ $v_0$). By maximality of the $(p,\a,\s,q)$-solutions $(u,\sigma)$ and $(v, \tau)$, we have
$\sigma_0\in (0,\sigma]$, $\tau_0\in (0,\tau]$ a.s., and
\begin{equation}
\label{eq:sigma_0_sigma_u_uu}
\uu=u \text{ a.e.\ on }[0,\sigma_0)\times \O, \quad \text{ and } \quad \vv=v \text{ a.e.\ on }[0,\tau_0)\times \O.
\end{equation}

We are ready to prove \eqref{eq:local_continuity_1_prop}. By \eqref{eq:sigma_0_sigma_u_uu}, for all $t\in [0,T_0]$.
\begin{align*}
\P\Big(\sup_{r\in [0,t]}\|u(r)-v(r)\|_{\Xap}\geq \gamma, \  \sigma_0\wedge \tau_0>t\Big)
&\leq
\P\Big(\sup_{r\in [0,t]}\|\uu(r)-\vv(r)\|_{\Xap}\geq \gamma\Big)\\
&\leq \frac{1}{\gamma^p}\E\|\uu-\vv\|_{C([0,t];\Xap)}^p\leq
\frac{C_0^p}{\gamma^p}\E\|u_0-v_0\|_{\Xap}^p,
\end{align*}
where in the last inequality we used \eqref{eq:stability_estimate_truncated_equation} and $t\leq T_0$.
The same argument also yields \eqref{eq:local_continuity_2_prop}.

Next we prove \eqref{eq:local_continuity_3_prop}. For all $t\in [0,T_0]$,
\begin{align*}
\P(\sigma_0\wedge \tau_0\leq t)
&\leq \P\big( \|\uu\|_{\X(t)}+\|\vv\|_{\X(t)}\geq  \lambda_0\big)\\
&\leq \P\big(2\|\uu\|_{\X(t)}+ \|\vv-\uu\|_{\X(t)} \geq \lambda_0 \big)\\
&\leq
\P\Big( \|\vv-\uu\|_{\X(t)} \geq \frac{\lambda_0}{2} \Big) +
\P\Big(\|\uu\|_{\X(t)}\geq \frac{\lambda_0}{4}\Big)\\
&\leq
\frac{2^pC_0}{\lambda_0^p}\E\|u_0-v_0\|_{\Xap}^p
+ \P(\sigma_1\leq t),
\end{align*}
where in the last step we used \eqref{eq:stability_estimate_truncated_equation} and
$
\sigma_1:=\inf\Big\{t\in[0,T_0]\,:\, \|\uu\|_{\X(t)}\geq \frac{\lambda_0}{4}\Big\}.
$
\end{proof}

\begin{remark}
\label{r:refined_local_existence}
The proof of Proposition \ref{prop:local_continuity_general} also yields the following facts.
\begin{enumerate}[{\rm(a)}]
\item By \eqref{eq:stability_estimate_truncated_equation} and \eqref{eq:sigma_0_sigma_u_uu}, the estimates \eqref{eq:local_continuity_1_prop}--\eqref{eq:local_continuity_3_prop} can be also formulated as $L^p(\O)$--estimates. For instance, \eqref{eq:local_continuity_1_prop} holds in the stronger form:
$$
\E\Big[\one_{\{\sigma_0\wedge \tau_0>t\} }\sup_{s\in [0,t]}\|u(s)-v(s)\|_{B^{2-\delta-2\frac{1+\a}{p}}_{q,p}}^p\Big]\leq C_0 \E\|u_0-v_0\|_{B^{2-\delta-2\frac{1+\a}{p}}_{q,p}}^p, \ \  t\in [0,T_0].
$$
\item\label{it:estimate_Htheta_local} Let $(\uu,\vv)$ be as in \eqref{eq:stability_estimate_truncated_equation}, i.e.\ the $(p,\a,\s,q)$--solutions to \eqref{eq:SEE_truncated} with data $(u_0,v_0)$, respectively. By Steps 1--2 of \cite[Theorem 4.5]{AV19_QSEE_1} and maximal $L^p$--regularity estimates (cf.\ \cite{AV21_SMR_torus}), we have the following stronger version of \eqref{eq:stability_estimate_truncated_equation}:
\begin{align*}
\E\|\uu-\vv\|_{H^{\theta,p}(0,T_0,w_{\a};H^{2-\s-2\theta,q})}\lesssim_{\theta} \E\|u_0-v_0\|_{B^{2-\delta-2\frac{1+\a}{p}}_{q,p}}^{p}, \ \ \text{ for all }\theta\in [0,\tfrac{1}{2}).
\end{align*}
Whence \eqref{eq:local_continuity_2_prop} also holds with $L^p(0,t,w_{\a};H^{2-\delta,q})$ replaced by
$H^{\theta,p}(0,t,w_{\a};H^{2-2\theta-\delta,q})$.
\item The proof of Proposition \ref{prop:local_continuity_general} shows that \eqref{eq:local_continuity_1_prop}--\eqref{eq:local_continuity_3_prop} holds also for quasilinear SPDEs as considered in \cite{AV19_QSEE_1} but taking $F_L=G_L\equiv 0$. The above proofs need the following modifications: \eqref{eq:xi_cut_off_semilinear} needs to be replaced with
$
\xi_{\lambda}(t,u)=\xi(\frac{1}{\lambda}[\|u\|_{\X(t)}+\sup_{s\in [0,t]}\|u(s)\|_{\Xap}]) $ for $ t\in [0,T_0]
$,
and \eqref{eq:def_sigma_0} needs to be replaced by
\begin{align*}
\sigma_0 &=\inf\Big\{t\in[0,T_0]\,:\, \|\uu\|_{\X(t)}+\sup_{s\in [0,t]}\|\uu(s)-u_0\|_{\Xap}\geq \lambda_0\Big\},\\
\tau_0&= \inf\Big\{t\in[0,T_0]\,:\, \|\vv\|_{\X(t)}+\sup_{s\in [0,t]}\|\vv(s)-u_0\|_{\Xap}\geq \lambda_0\Big\}.
\end{align*}
The same assertion as in Proposition \ref{prop:local_continuity_general} holds in the quasilinear setting, but the set $\{\tau_0= 0\}$ might have positive measure as we are only imposing smallness on $\E\|u_0-v_0\|_{\Xap}^p$.
\end{enumerate}
\end{remark}

\subsection{Blow-up criteria}
\label{ss:proof_blow_up_criteria}
Here we prove Theorem \ref{t:blow_up_criteria}. The argument follows the one in \cite[Lemma 6.10]{AV19_QSEE_2}. However, Theorem \ref{t:blow_up_criteria} cannot be deduced from such result since in the present situation we are also considering a parameter $h_0$ that is (possibly) different from $h$. Thus we provide a proof below.
For the reader's convenience, we give a (rough) idea of the argument which is based on the fact that solutions to \eqref{eq:reaction_diffusion_system} instantaneously regularizes, cf.\ \eqref{eq:reaction_diffusion_H_theta}-\eqref{eq:reaction_diffusion_C_alpha_beta}. Indeed, for any $s>0$, $u(s)$ is smooth and we may `restart' the system of SPDEs \eqref{eq:reaction_diffusion_system} considering the solution to such problem on $[s,\infty)$ with data $u(s)$, which will be denoted by $v$. Note that, a-priori, we don't know how $u(t)|_{t>s}$ and $v$ relate. Since $u(s)$ is smooth, the restarted problem \eqref{eq:reaction_diffusion_system} can be considered in a different `setting', i.e.\ replacing the parameters $(p,\a,q,\s,h)$ by (possibly) different ones $(p_0,\a_0,q_0,\s_0,h_0)$. With the latter choice, the results in \cite[Section 4]{AV19_QSEE_2} show that $v$ satisfy a blow-up criterium in the $(p_0,\a_0,q_0,\s_0,h_0)$-setting which is the analogue of the one claimed for $u$.  The conclusion follows by showing that $u=v$ on $[s,\infty)$ and thus the blow-up criteria for $v$ `transfers' to $u$.

\begin{proof}[Proof of Theorem \ref{t:blow_up_criteria}]
\eqref{it:blow_up_not_sharp}: \  We begin by collecting some useful facts. Fix $0<s<T<\infty$ and let $(u,\sigma)$ be the $(p,\a_{\crit},\s,q)$-solution to \eqref{eq:reaction_diffusion_system} provided by Theorem \ref{t:reaction_diffusion_global_critical_spaces}.
By \cite[Theorem 4.10(3)]{AV19_QSEE_2}, \eqref{eq:def_X_theta} and \eqref{eq:Besov_spaces} we have
\begin{equation}
\label{eq:blow_up_criteria_u}
\P\Big(\sigma<\infty,\, \sup_{ t\in [0,\sigma)}\|u(t)\|_{B^{\beta}_{q,p}}+ \|u\|_{L^p(0,\sigma;H^{\g,p})}<\infty\Big)=0,
\end{equation}
where
$$
\beta=\frac{d}{q}-\frac{2}{h-1}, \ \ \
 \g=\frac{d}{q}+\frac{2}{p}-\frac{2}{h-1}, \ \text{ and }\
  \a_{\crit}=p\Big(\frac{h}{h-1}-\frac{1}{2}\Big(\reg+\frac{d}{q}\Big)\Big)-1.
$$
Moreover, let us recall that, by \eqref{eq:regularity_u_reaction_diffusion_critical_spaces2} for $\theta_{\crit}:=\frac{\a_{\crit}}{p}<\frac{1}{2}-\frac{1}{p}$ and the weighted Sobolev embeddings (see e.g.\ \cite[Proposition 2.7]{AV19_QSEE_1}), we have
\begin{equation}
\label{eq:L_p_up_to_zero}
u\in H^{\theta_{\crit},p}_{{\rm loc}}([0,\sigma),w_{\a_{\crit}};H^{2-\s-2\theta_{\crit},q})\embed  L^p_{{\rm loc}}([0,\sigma);H^{\gamma,q})\ \text{ a.s.\ }
\end{equation}
Let $(q_0,p_0,\s_0,h_0,\beta_0)$ be as in Theorem \ref{t:blow_up_criteria} and $q_1>q_0$. Set $\a_0=\a_{\crit,0}=p_0\big(\frac{h_0}{h_0-1}-\frac{1}{2}(\reg_0+\frac{d}{q_0})\big)-1$. Fix $\a\in (\a_{\crit,0}, \a_{\crit,1})$  where $\a_{\crit,1}:= p_0(\frac{h_0}{h_0-1}-\frac{1}{2}(\reg+\frac{d}{q_1}))-1$ for $i\in \{0,1\}$.
Set $\beta:=2-\s-2\frac{1+\a}{p_0}$ and note that $\beta<\beta_0$.
By \eqref{eq:reaction_diffusion_C_alpha_beta} with $\theta_1=0$, $\theta=\theta_2\in (\beta,1)$ and the progressive measurability of $u$, we have
$$
\one_{\{\sigma>s\}}u(s)\in L^{0}_{\F_s}(\O;C^{\theta}).
$$
Combining this with
$C^{\theta}=B^{\theta}_{\infty,\infty}\embed B^{\beta}_{q_1,p_0}$ since $\theta>\beta$, we get
\begin{equation*}
\one_{\{\sigma>s\}}u(s)\in L^{0}_{\F_s}(\O;B^{\beta}_{q_1,p_0}), \ \ \text{ where } \ \ \V:=\{\sigma>s\}.
\end{equation*}
Up to a shift argument, Proposition \ref{prop:reaction_diffusion_global} ensures the existence of a $(p_0,\a_{0},\s_0,q_1)$-solution $(v,\tau)$ on $[s,\infty)$ to
\begin{equation}
\label{eq:reaction_diffusion_global_stochastic_s}
\left\{
\begin{aligned}
&\dd v_i-\div(a_i\cdot\nabla v_i) \,\dd t
= \Big[\div(F_i(\cdot, v)) +f_i(\cdot, v)\Big]\,\dd t \\
&\qquad \qquad \qquad \qquad \ \ \
+ \sum_{n\geq 1}  \Big[(b_{n,i}\cdot \nabla) v_i+ \rnoise_{n,i}(\cdot,v) \Big]\,\dd w_t^n, & \text{on }\Tor^d,\\
&v_i(s)=\one_{\{\sigma>s\}}u_i(s),  & \text{on }\Tor^d,
\end{aligned}
\right.
\end{equation}
where $v=(v_i)_{i=1}^{\ell}$. Moreover, the solution $(v,\tau)$ to \eqref{eq:reaction_diffusion_global_stochastic_s} instantaneously regularizes in time and space:
\begin{equation}
\label{eq:v_regularizes}
v\in H^{\theta,r}_{\rm loc}(s,\tau;H^{1-2\theta,\zeta})\quad  \text{a.s.\ for all }\theta\in  [0,1/2), \  r,\zeta\in (2,\infty).
\end{equation}
The notion of $(p_0,\a_{0},\s_0,q_1)$-solutions to \eqref{eq:reaction_diffusion_global_stochastic_s} follows as in Definition \ref{def:solution}.

By Step 1 of Proposition \ref{prop:reaction_diffusion_global} and the fact that $\a<\a_{\crit,1}$ we know that $B^{\beta}_{q_1,p_0}$ is \emph{not} critical for \eqref{eq:reaction_diffusion_global_stochastic_s}. Thus, applying \cite[Theorem 4.10(2)]{AV19_QSEE_2} to \eqref{eq:reaction_diffusion_global_stochastic_s},
\begin{equation*}
\P\Big(\tau<T,\, \sup_{t\in [s, \tau)}\|v(t)\|_{B^{\beta}_{q_1,p_0}}<\infty\Big)=0.
\end{equation*}
Since $\beta<\beta_0$, we have $ B^{\beta_0}_{q_1,\infty}\embed B^{\beta}_{q_1,p_0}$. Hence the previous implies
\begin{equation}
\label{eq:v_blow_up_criteria}
\P\Big(\tau<T,\, \sup_{t\in [s, \tau)}\|v(t)\|_{B^{\beta_0}_{q_1,\infty}}<\infty\Big)=0.
\end{equation}
Recall that $\V=\{\sigma>s\}$.
Since $\tau>s$ a.s., \eqref{eq:v_blow_up_criteria} shows that \eqref{it:blow_up_not_sharp}  follows as soon as we have shown
\begin{equation}
\label{eq:tau_sigma_u_v_equality}
\tau=\sigma \text{ a.s.\ on }\V \quad \text{ and }\quad u=v \text{ a.e.\ on }[s,\sigma)\times \V.
\end{equation}
The remaining part of this step is devoted to the proof of
\eqref{eq:tau_sigma_u_v_equality}. Let us begin by noticing that, by $h_0\geq h$ and \eqref{eq:reaction_diffusion_H_theta}, $(u|_{[s,\sigma)\times \V}, \one_{\V} \sigma+ \one_{\O\setminus\V} s)$ is a $(p_0,\a_{0},\s_0,q_1)$-solution to \eqref{eq:reaction_diffusion_global_stochastic_s}. The maximality of $(v,\tau)$ yields (see the last item in Definition \ref{def:solution})
\begin{equation}
\label{eq:tau_sigma_u_v_equality_correction}
\sigma\leq \tau \text{ a.s.\ on }\V \quad \text{ and }\quad u=v \text{ a.e.\ on }[s, \sigma)\times \V.
\end{equation}
To conclude it is enough to show that $\P(\V\cap \{\sigma<\tau\})=0$. To this end we employ the blow-up criteria in \eqref{eq:blow_up_criteria_u}. Indeed, by \eqref{eq:v_regularizes} and \eqref{eq:tau_sigma_u_v_equality_correction}, we have $u=v\in L^p_{\loc}((s,\sigma];H^{\gamma,q})$ a.s.\ on $ \V\cap \{\sigma<\tau\}$. Combining this with \eqref{eq:L_p_up_to_zero}, we find $u\in L^p( 0,\sigma;H^{\gamma,q})$ a.s.\ on  $ \V\cap \{\sigma<\tau\}$.  Similarly, one can check that $\sup_{t\in [0,\sigma)}\|u(t)\|_{B^{\beta}_{q,p}}<\infty$ a.s.\ on $ \V\cap \{\sigma<\tau\}$, and therefore
\begin{align*}
\P(\V\cap \{\sigma<\tau\})
&= \P\Big(\V\cap \{\sigma<\tau\} \cap \Big\{\sup_{t\in [0,\sigma)}\|u(t)\|_{B^{\beta}_{q,p}}+ \|u\|_{L^p(0,\sigma;H^{\g,p})}<\infty\Big\}\Big)\\
&\leq \P\Big(\sigma<\infty,\, \sup_{t\in [0,\sigma)}\|u(t)\|_{B^{\beta}_{q,p}}+ \|u\|_{L^p(0,\sigma;H^{\g,p})}<\infty\Big)\stackrel{\eqref{eq:blow_up_criteria_u}}{=}0.
\end{align*}

\eqref{it:blow_up_sharp}: The proof is similar to the one of \eqref{it:blow_up_not_sharp}. Indeed, let us consider the $(p_0,\a_{\crit,0},\s_0,q_0)$-solution to \eqref{eq:reaction_diffusion_global_stochastic_s} where  $\a_{\crit,0}=p(\frac{h_0}{h_0-1}-\frac{1}{2}(\reg_0+\frac{d}{q_0}))-1$. Here the subscript `$\crit$' stresses that the corresponding space for the initial data $B^{\beta_0}_{q_0,p_0}$ is critical for \eqref{eq:reaction_diffusion_global_stochastic_s} (cf.\ Step 1 of Proposition \ref{prop:reaction_diffusion_global} and note that the spatial integrability is $q_0$). Compared to Step 1, the only difference is that instead of \eqref{eq:v_blow_up_criteria} we use \cite[Theorem 4.10(3)]{AV19_QSEE_2} (which holds also in critical situations) and it yields
\begin{equation*}
\P\Big(\tau<T,\, \sup_{t\in [s,\tau)}\|v(t)\|_{B^{\beta_0}_{q_0,p_0}}+\|v\|_{L^{p_0}(s,\tau;H^{\g_0,q_0})}<\infty\Big)=0,
\end{equation*}
where $\beta_0,\g_0$ are as in the statement of Theorem \ref{t:blow_up_criteria}.
\end{proof}

\begin{proof}[Proof of Corollary \ref{cor:blow_up_criteria}]
To prove \eqref{it:blow_up_not_sharp_L} we use Theorem \ref{t:blow_up_criteria}\eqref{it:blow_up_not_sharp_L} with an appropriate choice of $(q_0,q_1)$. Recall that $h_0\geq 1+\frac{4}{d}$, $\zeta_0=\frac{d}{2}(h_0-1)$ and let $\zeta_1>\zeta_0$.
Choose $\s_0>1$ small enough so that Assumption \ref{ass:reaction_diffusion_global}\eqref{it:regularity_coefficients_reaction_diffusion} holds. Fix $q_1\leq \zeta_1$ such that
$$
\zeta_0<q_1<\frac{d(h_0-1)}{h_0+1-\reg_0(h_0-1)}.
$$
The above choice is possible since $\s_0>1$. Since $\s_0<2$, we may fix $p_0\in (q_1,\infty)$ such that
$$
\frac{1}{p_0}+\frac{1}{2}\Big(\reg_0+\frac{d}{\zeta_0}\Big)\leq \frac{h_0}{h_0-1} .
$$
One can check the condition in Theorem \ref{t:reaction_diffusion_global_critical_spaces} with $(p,q,\s,h)$ replaced by $(p_0,q_0,\s_0,h_0)$. By $\zeta_1\geq q_1$ and elementary embeddings for Besov spaces,
$
L^{\zeta_1}\embed L^{q_1}\embed B^{0}_{q_1,\infty}.
$
Hence
$$
\Big\{s<\sigma<T,\,\sup_{t\in [s,\sigma)} \|u(t)\|_{L^{\zeta_1}}<\infty\Big\}
\subseteq
\Big\{s<\sigma<T,\,\sup_{t\in [s,\sigma)} \|u(t)\|_{B^{0}_{q_1,\infty}}<\infty\Big\}
$$
Thus that \eqref{it:blow_up_not_sharp_L} follows from Theorem \ref{t:blow_up_criteria}\eqref{it:blow_up_not_sharp} with $q_0=\zeta_0$ noticing that $\beta_0=\frac{d}{\zeta_0}-\frac{2}{h_0-1}=0$.

To prove \eqref{it:blow_up_sharp_L} we use Theorem \ref{t:blow_up_criteria}\eqref{it:blow_up_sharp}.
Let $\delta_0>1$ be as above.
By assumption
$
q_0<\frac{d(h_0-1)}{h_0+1-\reg_0(h_0-1)}
$
and therefore
$$
\frac{d}{q_0}>\frac{2}{h_0-1}-(\s_0-1).
$$
Hence to ensure the existence of $p_0$ such that $\frac{2}{p_0}+\frac{d}{q_0}=\frac{2}{h_0-1}$ we need $p_0>\frac{2}{\s_0-1}$ as required in \eqref{it:blow_up_sharp_L}. Since $q_0>\zeta_0$,
$$
L^{\zeta_0}\stackrel{(i)}{ \embed} B^{\beta_0}_{q_0,q_0}\stackrel{(ii)}{ \embed} B^{\beta_0}_{q_0,p_0}
$$
where in $(i)$ we used the Sobolev embeddings for Besov spaces (recall $\beta_0=\frac{d}{q_0}-\frac{2}{h_0-1}$) and in $(ii)$ the fact that $p_0\geq q_0$ by assumption. Hence
\begin{align*}
&\Big\{s<\sigma<T,\,\sup_{t\in [s,\sigma)} \|u(t)\|_{L^{\zeta_0}}+\|u\|_{L^{p_0}(s,\sigma;L^{q_0})}<\infty\Big\}\\
&\subseteq
\Big\{s<\sigma<T,\,\sup_{t\in [s,\sigma)} \|u(t)\|_{B^{\beta_0}_{q_0,p_0}}+\|u\|_{L^{p_0}(s,\sigma;L^{q_0})}<\infty\Big\}.
\end{align*}
Thus \eqref{it:blow_up_sharp_L} follows from Theorem \ref{t:blow_up_criteria}\eqref{it:blow_up_sharp} by noticing that $\g_0=\frac{2}{p_0}+\frac{d}{q_0}-\frac{2}{h_0-1} =0$.
\end{proof}

Finally we prove a compatibility result for the solutions obtained in different settings.
\begin{proposition}[Compatibility of different settings]\label{prop:comp}
If Proposition \ref{prop:reaction_diffusion_global} is applicable for two sets of exponents $(p_1,\a_1,\s_1,q_1, h_1)$ and $(p_2,\a_2,\s_2,q_2, h_2)$, then the corresponding solutions $(u_1, \sigma_1)$ and $(u_2,\sigma_2)$ coincide, i.e.\ $\sigma_1=\sigma_2$ a.s.\ and $u_1=u_2$ a.e.\ on $[0,\sigma_1)\times \O$.
\end{proposition}

As Theorem \ref{t:reaction_diffusion_global_critical_spaces} is a special case of Proposition \ref{prop:reaction_diffusion_global}, the above compatibility also holds for solutions provided by Theorem \ref{t:reaction_diffusion_global_critical_spaces}. To explain the difficulty in proving the above result, let us consider two settings where Theorem \ref{t:reaction_diffusion_global_critical_spaces} applies with $p_1\neq p_2$ and $(\s_1,q_1, h_1)=(\s_2,q_2, h_2)$. Note that the corresponding critical weights $\a_{\crit,i}:= p_i\big(\frac{h}{h-1}-\frac{1}{2}(\reg+\frac{d}{q})\big)-1$ satisfy $\frac{1+\a_1}{p_1}=\frac{1+\a_2}{p_2}$. In particular, the $L^{p_i}(w_{\a_{\crit,i}})$--spaces on RHS\eqref{eq:integrability_nonlinearity} of Definition \ref{def:solution} \emph{cannot} be embedded one in the other (cf.\ \cite[Proposition 2.1(3) and Remark 2.2]{AV19_QSEE_2}). Hence, a priori it is unclear how to compare the solutions, and use the uniqueness in one of the two settings. To solve this, we use an approximation argument, local continuity, and regularization results.

\begin{proof}[Proof of Proposition \ref{prop:comp}]

{\em  Step 1: Approximation.}
Note that $u_0\in \cap_{i\in \{1,2\}} B^{2-\s_i-2\frac{1+\a_i}{p_i}}_{q_i,p_i}$ a.s.\ as the assumptions of Proposition \ref{prop:reaction_diffusion_global} are verified in both settings.
By localization in $\O$, see \cite[Theorem 4.7(d)]{AV19_QSEE_1}  with
\[\Gamma = \cap_{i\in \{1,2\}}\big\{\|u_0\|_{B^{2-\s_i-2\frac{1+\a_i}{p_i}}_{q_i,p_i}}\leq n\big\}\in \F_0,\]
it is enough to consider the case
$$
u_0\in \cap_{i\in \{1,2\}} L^{p_i}(\O;B^{2-\s_i-2\frac{1+\a_i}{p_i}}_{q_i,p_i}).
$$
In the following we let $(u^{(n)}_0)_{n\geq 1}\subseteq L^{\infty}_{\F_0}(\O;C^1)$ be such that
$u_0^{(n)}\to u_0$ in $L^{p_i}(\O;B^{2-\s_i-2\frac{1+\a_i}{p_i}}_{q_i,p_i})$ for $i\in \{1, 2\}$.
Fix some $r> \max\{p_1, p_2, q_1, q_2, 2d+2, d(h_1-1), d(h_2-1)\}$ and set $h=\max\{h_1, h_2\}$. Then one can check that Assumption \ref{ass:reaction_diffusion_global}$(r,r,h,1)$ holds, and
\eqref{eq:reaction_diffusion_globalii} holds with $(p,\kappa,q,\delta)$ replaced by $(r,0,r,1)$. Therefore, by Proposition \ref{prop:reaction_diffusion_global}, for each $n\geq 1$ there exists a (unique) $(r,0,1,r)$--solution $(u^{(n)},\sigma^{(n)})$ to \eqref{eq:reaction_diffusion_system} such that a.s.\ $\sigma^{(n)}>0$ and
\begin{equation}
\label{eq:u_n_regularity_compatibility}
u^{(n)}\in H^{\theta,r}_{\loc}([0,\sigma^{(n)});H^{1-2\theta,r})\cap C([0,\sigma^{(n)});B^{1-\frac{2}{r}}_{r,r}), \ \ \theta\in [0,1/2).
\end{equation}
By Sobolev embedding (since $1-\frac{2}{r}>-\frac{d}{r}$) we find that $u^{(n)}\in C([0,\sigma^{(n)});C(\T^d;\R^\ell))$ a.s.
Now let us fix $i\in \{1,2\}$ and $n\geq 1$. Consider the $(p_i,\a_i,\s_i,q_i)$--solution $(u^{(n)}_i,\sigma^{(n)}_i)$ provided by Proposition \ref{prop:reaction_diffusion_global} with the parameters $(p_i,\a_i,\s_i,q_i)$ and initial data $u^{(n)}_0$. By Definition \ref{def:solution}, \eqref{eq:u_n_regularity_compatibility}, and the special choice of $r$, one obtains that $(u^{(n)},\sigma^{(n)})$ is a local $(p_i,\a_i,\s_i,q_i)$--solution to \eqref{eq:reaction_diffusion_system}. Hence  $\sigma^{(n)}\leq \sigma_i^{(n)}$ and $u^{(n)}=u_i^{(n)}$ a.e.\ on $[0, \sigma^{(n)})\times \O$ by maximality of $(u_i^{(n)},\sigma_i^{(n)})$. Now reasoning as in the proof of Theorem \ref{t:blow_up_criteria}, by instantaneous regularization of $(p_i,\a_i,\s_i,q_i)$--solutions (i.e.\ \eqref{eq:reaction_diffusion_H_theta_1}--\eqref{eq:reaction_diffusion_C_alpha_beta_1}), we also obtain
\begin{equation}
\label{eq:equality_sigma_sigma_n_correction}
\sigma_i^{(n)}=\sigma^{(n)} \text{ a.s.}, \quad \text{ and }\quad u^{(n)}_i = u^{(n)}\text{ a.e.\ on }[0,\sigma^{(n)})\times \O.
\end{equation}
Thus in the following we write $(u^{(n)},\sigma^{(n)})$ instead of $(u_i^{(n)},\sigma^{(n)}_i)$.

\emph{Step 2: For all $i\in \{1,2\}$, up to extracting a (not relabeled) subsequence of $((u^{(n)},\sigma^{(n)}))_{n\geq 1}$, there exists a stopping time $\tau_i\in (0,\sigma_i)$ such that a.s.\ $\tau_i<\liminf_{n\to \infty} \sigma^{(n)}$ and}
\begin{equation}
\label{eq:continuity_n}
u_i =\lim_{n\to \infty} u^{(n)}\ \text{ a.e.\ on }[0,\tau_i)\times \O.
\end{equation}

Note that the RHS\eqref{eq:continuity_n} makes sense since $\tau_i<\liminf_{n\to \infty} \sigma^{(n)}$.

In this step we fix $i\in \{1,2\}$.
Moreover, we use the notation introduced in the proof of Proposition \ref{prop:local_continuity_general} with the subscript $i$ to keep track of the setting we are considering. For instance $\X_i$ denotes the space introduced in \cite[eq.\ (4.14)]{AV19_QSEE_1} in the $(p_i,\a_i,\s_i,q_i)$--setting.
Here we prove the claim with $\tau_i$ given by (cf.\ the definition of $\sigma_1$ at the end of the proof of Proposition \ref{prop:local_continuity_general})
$$
\tau_i:=\inf\Big\{t\in[0,T_{0,i}]\,:\, \|\uu_i\|_{\X_i(t)}\geq \frac{\lambda_{0,i}}{4}\Big\},
$$
where $\uu_i$ is the $(p_i,\a_i,\s_i,q_i)$--solution on $[0,T_{0,i}]$ solution to \eqref{eq:SEE_truncated} in the $(p_i,\a_i,\s_i,q_i)$--setting, and where $T_{0,i}$ and $\lambda_{0,i}$ are as in the proof of Proposition \ref{prop:local_continuity_general}.

To prove the claim of Step 2, let $\sigma_{0,i}^{(n)}$ be as in \eqref{eq:def_sigma_0} with $(\uu,\X,\lambda_0)$ replaced by $(\uu_i^{(n)},\X_i,\lambda_{0,i})$.  By \cite[Lemma 4.9]{AV19_QSEE_1}, Remark \ref{r:refined_local_existence}\eqref{it:estimate_Htheta_local}  and \cite[Corollary 5.2]{ALV21} it follows that
\begin{align}
\label{eq:convergence_uu_i_uu}
\uu^{(n)}_i \to \uu_i  \ \text{ a.s.\ in }\X_i(T_{0, i})\cap C([0,T_{0, i}];B^{2-\s_i-2\frac{1+\a_i}{p_i}}_{q_i,p_i}).
\end{align}
Therefore, up to extracting a subsequence, we have $\tau_i <\liminf_{n\to \infty} \sigma_{0,i}^{(n)}$.
Recall that (see \eqref{eq:sigma_0_sigma_u_uu})
\begin{eqnarray*}
\sigma_{0,i}^{(n)} \leq \sigma^{(n)}_i \ \  \ & \text{a.s.\ \ and} & \ \ \uu_i^{(n)}=u_i^{(n)} \  \ \text{a.e.\ on $[0,\sigma_{0,i}^{(n)})\times \O$,}
\\ \sigma_{0,i}  \leq \sigma_i \ \ & \text{a.s.\ \ and} & \ \ \uu_i=u_i\  \ \text{a.e.\ on $[0,\sigma_{0,i})\times \O$}.
\end{eqnarray*}
Hence $\tau _{i}<\liminf _{n\to \infty} \sigma _{i}^{(n)}=\liminf _{n\to \infty} \sigma^{(n)}$ by \eqref{eq:equality_sigma_sigma_n_correction}. Finally, Step 1 and \eqref{eq:convergence_uu_i_uu} give \eqref{eq:continuity_n}.

\emph{Step 3: Conclusion}. By Steps 1 and 2, we deduce that
$$
u_1(t)=u_2(t)\ \  \text{ a.s.\ for all }t\in [0,\tau_1\wedge \tau_2).
$$
Set $\tau:=\tau_1\wedge \tau_2\in (0,\sigma_1\wedge \sigma_2]$ and let $s>0$. Using the instantaneous regularization (i.e.\ \eqref{eq:reaction_diffusion_H_theta_1}-\eqref{eq:reaction_diffusion_C_alpha_beta_1}) we have $\one_{\{\tau>s\}}u_1(s)=\one_{\{\tau>s\}} u_2(s)\in C^{\theta}$ for all $\theta\in (0,1)$. Hence, as in Step 1, the conclusion follows by repeating the argument used in Theorem \ref{t:blow_up_criteria}.
\end{proof}

\begin{remark} \
\begin{enumerate}[{\rm(a)}]
\item (\emph{A proof involving $\X$--space}).
Proposition \ref{prop:comp} can be also proved by using embedding results for the $\X$--spaces (cf.\ the proof of \cite[Proposition 6.8]{AV19_QSEE_2} where $\a_i=0$ for some $i\in \{1,2\}$). Besides being technically more difficult, this approach also requires additional assumptions on the parameters $(p_i,\a_i,\s_i,q_i)$. These restrictions can be removed by tedious iteration arguments. Hence we prefer the above more direct argument based on local continuity.
\item
The proof of Proposition \ref{prop:comp} extends verbatim to other situations such as the Navier--Stokes equations with transport noise as analyzed in \cite{AV20_NS}.
\end{enumerate}
\end{remark}

\subsection{Positivity}\label{ss:positivityproof}

Next we will prove the positivity of the solution stated in Theorem \ref{thm:positivity}. For the proof we need the well-posedness and regularity results of Theorem \ref{t:reaction_diffusion_global_critical_spaces}, Proposition \ref{prop:local_continuity}, the blow-up criteria of Theorem \ref{t:blow_up_criteria}, and a maximum principle for linear scalar equations, which is a variation of \cite{Kry13} (see Lemma \ref{lem:maxprinciple} in the appendix).

In case of smooth initial data the proof below can be shortened considerably. In particular, the approximation argument in Step 1 in the proof below can be omitted. Note that Step 1 relies on the rather technical local continuity result of Proposition \ref{prop:local_continuity}.

\begin{proof}[Proof of Theorem \ref{thm:positivity}]
Below we write $Y:=\Xapcrit=B^{\frac{d}{q}-\frac{2}{h-1}}_{q,p}(\Tor^d;\R^{\ell})$ for convenience.

\emph{Step 0: Reduction to the case $u_0\in L^p(\O;Y)$}.
To prove the claim of this step, assume that the claim of Theorem \ref{thm:positivity} holds for $L^p(\O)$--integrable data.
For any $n\geq 1$, set
$\V_n:=\{\|u_0\|_{Y}\leq n\}
$
and let $(u^{(n)},\sigma^{(n)})$ be the $(p,\a_{\crit},\s,q)$-solution to \eqref{eq:reaction_diffusion_system} with initial data $\one_{\V_n}u_0$. Thus, by assumption, Theorem \ref{thm:positivity} holds for $(u^{(n)},\sigma^{(n)})$ and therefore
\begin{equation}
\label{eq:positivity_u_n}
u^{(n)}(t,x)\geq 0 \ \text{ a.s.\ for all $t\in [0,\sigma^{(n)})$ and }x\in \Tor^d.
\end{equation}
By localization (i.e.\ \cite[Theorem 4.7(d)]{AV19_QSEE_1}), we have
$$
\sigma=\sigma^{(n)} \text{ a.s.\ on }\V_n, \quad \text{and }\quad
u=u^{(n)}\text{ a.e.\ on }[0,\sigma)\times \V_n.
$$
The previous identity, the arbitrariness of $n\geq 1$ and \eqref{eq:positivity_u_n} yield the claim of this step.

\emph{Step 1: Reduction to the case $u_0\in L^0(\O;C^{\alpha}(\Tor^d;\R^{\ell}))$ where $\alpha\in (0,1)$}.
Fix $\alpha\in (0,1)$. By Step 0 we can assume that $u_0\in L^p(\O;Y)$. In the current step we assume that Theorem \ref{thm:positivity} holds for initial data from $ L^p(\O;C^{\alpha}(\R^d;\R^{\ell}))$. Note that from \eqref{eq:reaction_diffusion_C_alpha_beta}, we know that $u$ is smooth on $(0,\sigma)\times \T^d$. This will be used several times below.

Set
$$
\A_+:=\{\varphi\in \D(\Tor^d;\R^{\ell})\,:\, \varphi\geq 0\  \text{ and } \ \|\varphi\|_{Y^*}\leq 1\}.
$$
It is important to note that $\A_+$ is separable due to the separability of $\D(\Tor^d;\R^{\ell})$, in order to have measurable sets below.
Recall that $u_0\geq 0$ by assumption. Let $(C_0,T_0,\varepsilon_0, \sigma_0,\sigma_1)$ be as in the statement of Proposition \ref{prop:local_continuity}. Choose a sequence $(u_{0}^{(n)})_{n\geq 1}$ in $L^p(\O;C^{\alpha}(\T^d;\R^{\ell}))$ such that $u_{0}^{(n)}\geq 0$ a.s.\ on $\Tor^d$ and $u_{0}^{(n)}\to u_0$ in $L^p(\O;Y)$.
Without loss of generality we may assume that $\E\|u_0-u_0^{(n)}\|_{Y}^p\leq \varepsilon_0$ for all $n\geq 1$.
Let $(u^{(n)},\sigma^{(n)})$ be the $(p,\a_{\crit},\s,q)$-solution to \eqref{eq:reaction_diffusion_system} with initial data $u_0^{(n)}$.
The reductive assumption ensures that
\begin{equation}
\label{eq:positivity_u_n_2}
u^{(n)}(t,x)\geq 0 \ \text{ a.s.\ for all $t\in (0,\sigma^{(n)})$ and }x\in \Tor^d.
\end{equation}
Note that, for all $t\in (0,T_0]$, $n\geq 1$ and $\gamma>0$,
\begin{align*}
&\P\Big(\inf_{r\in [0, t]}\int_{\Tor^d} u(r)\cdot\varphi\,\dd x \leq -\gamma \text{ for some }\varphi\in \A_+ ,\, \sigma_0>t\Big)\\
&\leq \P\Big(\inf_{r\in [0,t]}\int_{\Tor^d} u(r)\cdot\varphi\,\dd x \leq -\gamma \text{ for some }\varphi\in \A_+  ,\, \sigma_0\wedge \tau_0^{(n)}> t\Big)+\P(\sigma_0\wedge \tau_0^{(n)}\leq t),
\end{align*}
where $\tau_0^{(n)}$ is as in Theorem \ref{t:reaction_diffusion_global_critical_spaces} with $v_0$ replaced by $u_0^{(n)}$. Note that, by combining \eqref{eq:regularity_u_reaction_diffusion_critical_spaces}, \eqref{eq:positivity_u_n_2} and $\|\varphi\|_{Y^*}\leq 1$ for $\varphi\in \A_+$,
$$
\Big\{\inf_{r\in [0,t]}\int_{\Tor^d} u(r)\cdot\varphi\,\dd x \leq -\gamma ,\, \sigma_0\wedge \tau_0^{(n)}\geq t\Big\}\cap
\Big\{\sup_{r\in [0,t]}\| u(r)-u^{(n)}(r)\|_{Y} <\gamma\Big\}=\emptyset.
$$
Hence
\begin{align*}
& \P\Big(\inf_{r\in [0,t]}\int_{\Tor^d} u(r)\cdot\varphi\,\dd x \leq -\gamma \text{ for some }\varphi\in \A_+  ,\, \sigma_0>t\Big)	\\
 &\leq \P\Big(\sup_{r\in [0,t]}\|u(r)-u^{(n)}(r)\|_{Y} \geq \gamma ,\, \sigma_0\wedge \tau_0^{(n)}>t\Big)
 +\P(\sigma_0\wedge \tau_0^{(n)}\leq t)
 \\
&\leq  C_0 (1+\gamma^{-p})\E\|u_0-u_0^{(n)}\|_{Y}^p+C_0\P(\sigma_1\leq t),
\end{align*}
where in the last estimate we used \eqref{eq:local_continuity_1} and \eqref{eq:local_continuity_3}. Letting $n\to \infty$ and $\gamma=k^{-1}\downarrow 0$, the above estimate yields
\begin{equation}
\label{eq:estimate_A_t}
\P(\U_t) \leq
C_0 \P(\sigma_1\leq t),
\end{equation}
where
$$
\U_t:=\Big\{\inf_{r\in [0, t]}\int_{\Tor^d} u(r)\cdot\varphi\,\dd x < 0\text{ for some }\varphi\in \A_+ ,\, \sigma_0>t\Big\}.
$$
Note that
$
\{\sigma_0>t\}=\U_t \cup \V_t
$
where
\begin{align*}
\V_t
&:=\Big\{\inf_{r\in [0, t]}\int_{\Tor^d} u(r)\cdot\varphi\,\dd x \geq   0 \text{ for all }\varphi\in \A_+,\, \sigma_0>t\Big\}\\
&\ = \big\{ u(r,x)\geq    0 \text{  for all } x\in \T^d, r\in (0,t]\big\}\cap\{ \sigma_0>t\},
\end{align*}
where we used the smoothness of $u$.
By definition $\V_s\supseteq \V_t$ as $s\leq t$ and $\V_t\in \F_t$ for all $t\in [0,T_0]$.
The estimate \eqref{eq:estimate_A_t} gives
\begin{equation}
\label{eq:V_t_has_probability_almost_1}
\P(\V_t) = \P(\sigma_0>t) - \P(\U_t)\geq  \P(\sigma_0>t) - C_0 \P(\sigma_1\leq t)\to 1 \ \ \text{as $t\downarrow0$},
\end{equation}
where we used $\sigma_0>0$ and $\sigma_1>0$ a.s.

Fix $t\in [0,T_0]$ and consider $(v,\tau)$ the $(p,\a_{\crit},\s,q)$-solution to \eqref{eq:reaction_diffusion_system} on $[t,\infty)$ with initial data $v(t)=v_t:=\one_{\V_t}u(t)$.
By definition of $\V_t$, we have a.s.\ $v_t\geq 0$ on $\Tor^d$, and by the smoothness of $(u,\sigma)$ (see  \eqref{eq:reaction_diffusion_C_alpha_beta}), we have
$
v_t\in L^0_{\F_t}(\O;C^{\alpha}(\Tor^d;\R^{\ell})) .
$

In particular, by the reductive assumption (applied at initial time $t$ instead of $0$), we have a.s.
\begin{equation*}
v(r,x)\geq 0 \ \text{ for all $r\in [t,\tau)$ and }x\in \Tor^d.
\end{equation*}
As before, by localization  (i.e.\ \cite[Theorem 4.5(4)]{AV19_QSEE_1}), $\tau=\sigma$ a.s.\ on $\V_t$ and $v=u$ a.e.\ on $[t,\tau)\times \V_t$.
It follows that
\begin{align*}
\P\big(u(r,x)\geq 0 \ \forall r\in [t,\sigma) \ \text{and} \ x\in \T^d\big)
& = \lim_{s\downarrow 0}\P\big(\big\{u(r,x)\geq 0 \ \ \forall r\in [t,\sigma) \ \text{and} \ x\in \T^d\big\}\cap \V_s\big)
\\ &= \lim_{s\downarrow 0}\P\big(\big\{v(r,x)\geq 0 \ \forall r\in [t,\tau) \ \text{and} \ x\in \T^d\big\}\cap \V_s\big)
\\ &= \lim_{s\downarrow 0}\P(\V_s) = 1,
\end{align*}
by \eqref{eq:V_t_has_probability_almost_1}.  Therefore, letting $t\downarrow 0$ we obtain that a.s. $u(r,x)\geq 0$ for all $r\in (0,\sigma)$ and $x\in \T^d$.

\emph{Step 2: Reduction to the case $u_0\in L^\infty(\O;C^{\alpha}(\Tor^d;\R^{\ell}))$ where $\alpha\in (0,1)$}. Due to Step 1, the claim of Step 2 follows by localization as in Step 0.

{\em Step 3: Reduction to positivity of a new function $u^+$.}
By the previous steps we may suppose that $u_0\in L^\infty(\Omega;C^{\alpha}(\T^d;\R^\ell))$ for some $\alpha>0$.
For all $(t,\om,x)\in \R_+\times\O\times \Tor^d$, $y\in \R^{\ell}$ and
$i\in \{1,\dots,\ell\}$, let
\begin{align*}
&f_i^+(t,\om,x,y):=f_i(t,\om,x, (y\vee 0)), \qquad
F_i^+(t,\om,x,y):=F_i(t,\om,x, (y\vee 0)),\\
& \quad \quad \quad \quad\quad \quad \quad \quad
g_{n,i}^+(t,\om,x,y):=g_{n,i}(t,\om,x, (y\vee 0)).
\end{align*}
We denote by \eqref{eq:reaction_diffusion_system}$^{+}$ the system of SPDEs \eqref{eq:reaction_diffusion_system} with $(F,f,g)$ replaced by $(F^+,f^+,g^+)$. Since the assignment $y\mapsto y\vee 0$ is globally Lipschitz, $(F^+,f^+,g^+)$ satisfies Assumption \ref{ass:reaction_diffusion_global}\eqref{it:growth_nonlinearities} with the same parameters. Thus, by Theorem \ref{t:reaction_diffusion_global_critical_spaces} there exists a $(p,\a_{\crit},\s,q)$-solution $(u^+,\sigma^+)$ to \eqref{eq:reaction_diffusion_system}$^{+}$. Moreover, Theorem \ref{t:blow_up_criteria}\eqref{it:blow_up_not_sharp_L} implies (with $T\uparrow \infty$)
\begin{equation}
\label{eq:blow_up_u_plus}
\P\Big(s<\sigma^+<\infty,\, \sup_{t\in [s,\sigma^+)} \|u^+(t)\|_{L^{\infty}}<\infty\Big)=0, \ \ \text{ for all }s>0.
\end{equation}
We claim that
\begin{equation}
\label{eq:positivity_on_sigma_j_u_plus}
u^+\geq 0 \ \ \text{ a.e.\ on }[0,\sigma^+)\times \O\times \Tor^d.
\end{equation}

Next we show that \eqref{eq:positivity_on_sigma_j_u_plus} yields the claim of Theorem \ref{thm:positivity}. More precisely we prove that, if \eqref{eq:positivity_on_sigma_j_u_plus} holds, then
\begin{equation}
\label{eq:positivity_equality_on_sigma_j_u_plus}
\sigma^+ =\sigma  \ \text{ a.s.\ } \quad \text{ and }\quad
u^+= u \ \text{ a.e.\ on }[0,\sigma)\times \O\times \Tor^d.
\end{equation}
Suppose that \eqref{eq:positivity_on_sigma_j_u_plus} holds.
Thus, by Definition \ref{def:solution}, $(u^+,\sigma^+)$ is a \emph{local} $(p,\a_{\crit},\s,q)$-solution to the original problem \eqref{eq:reaction_diffusion_system}. Since $(u,\sigma)$ is a $(p,\a_{\crit},\s,q)$-solution to \eqref{eq:reaction_diffusion_system}, we have
\begin{equation*}
\sigma^+ \leq \sigma \ \text{ a.s.\ } \quad \text{ and }\quad
u^+= u \  \text{ a.e.\ on }[0,\sigma^+)\times \O\times \Tor^d.
\end{equation*}
Since $\sigma^+>0$ a.s., to prove \eqref{eq:positivity_equality_on_sigma_j_u_plus} it remains to show that $\P(s<\sigma^+<\sigma)=0$ for all $s>0$. To this end, fix $s\in (0,\infty)$. Note that $u^+=u$ a.s.\ on $[s, \sigma^+\wedge \sigma)$, and $u\in C([s,\sigma^+]\times \Tor^d;\R^{\ell})$ a.s.\ on $\{s<\sigma^+<\sigma\}$ by \eqref{eq:reaction_diffusion_C_alpha_beta}. Thus
\begin{align*}
\P(s<\sigma^+<\sigma)
&=
\P\Big(\{s<\sigma^+<\sigma\}\cap \sup_{t\in [0,\sigma^+)}\|u^+(t)\|_{L^{\infty}}<\infty\Big)\\
&\leq \P\Big(s<\sigma^+<\infty,\, \sup_{t\in [s,\sigma^+)} \|u^+(t)\|_{L^{\infty}}<\infty\Big)
\stackrel{\eqref{eq:blow_up_u_plus}}{=}
0.
\end{align*}
This proves \eqref{eq:positivity_equality_on_sigma_j_u_plus} in case \eqref{eq:positivity_on_sigma_j_u_plus} holds.

{\em Step 4: Proof of \eqref{eq:positivity_on_sigma_j_u_plus}.}
Fix $i\in \{1, \ldots, \ell\}$. As usual, for all $j\geq 1$, we set
$$
\sigma^+_j:=
\inf\Big\{t\in [0,\sigma) \,:\, \|u^+(t)-u_0\|_{L^{\infty}}+ \|u^+\|_{L^2(0,t;H^1)}\geq j\Big\} \wedge j\ \ \text{ where }\ \ \inf\emptyset:=\sigma.
$$
By \eqref{eq:reaction_diffusion_H_theta}-\eqref{eq:reaction_diffusion_C_alpha_beta} (applied with $(F,f,g)$ replaced by $(F^+,f^+,g^+)$) we have $\lim_{j\to \infty}\sigma_j^+= \sigma^+$. Therefore, to show \eqref{eq:positivity_on_sigma_j_u_plus} it is suffices to prove
\begin{equation*}
u^+\geq 0 \ \ \text{ a.e.\ on }[0,\sigma_j^+]\times \O\times \Tor^d.
\end{equation*}
In the following we fix $j\geq 1$ such that $\|u_0\|_{L^\infty(\Omega;L^\infty)}<j$, and we drop it from the notation. Moreover, we set $\tau^+:=\sigma_j^+$. Note that
\begin{equation}
\label{eq:u_j_bound_positivity}
\sup_{t\in [0,\tau^+)}\|u^+(t)\|_{L^{\infty}}\leq 2j \text{ a.s.}
\qquad \text{ and }\qquad
\|u^+\|_{L^2(0,\tau^+;H^1)}\leq j \text{ a.s. }
\end{equation}

Next we turn the nonlinearities into globally Lipschitz function by a cut-off argument. Let $\zeta:\R^\ell \to \R^\ell$ be a smooth map such that $\zeta|_{\{|y|\leq 2j\}}=1$ and $\zeta|_{\{|y|\geq 2j+1\}}=0$. Set
\begin{align*}
&\fp_i(\cdot, y):=f_i^+(\cdot, \zeta(y)), \qquad
\Fp_i(\cdot, y):=F^+_i(\cdot, \zeta(y)), \qquad
\gp_{n,i}(\cdot, y):=g^+_{n,i}(\cdot, \zeta(y)).
\end{align*}
Then $\fp_i,\Fp_i,\gp_{n,i}$ are globally Lipschitz w.r.t.\ $y\in \R^\ell$ uniformly in $(t,\om,x)\in [0,\infty)\times\Omega\times\T^d$.

For a vector $y\in \R^{\ell}$ we set $\wh{y}_i=(y_1,\dots,y_{i-1},0,y_{i+1},\dots,y_{\ell})$.
Note that, a.e.\ on $[0,\tau^+)\times \O\times \Tor^d$,
\begin{align*}
f_i^+(\cdot,u)
=\big[f^+_i(\cdot,u)-f^+_i(\cdot,\wh{u}_i)\big]+f^+_i(\cdot,\wh{u}_i)
\stackrel{\eqref{eq:u_j_bound_positivity}}{=} \big[\fp_i(\cdot,u)-\fp_i(\cdot,\wh{u}_i)\big]+\fp_i(\cdot,\wh{u}_i).
\end{align*}
Below we will exploit that $\fp_i(\cdot,\wh{u}_i)\geq 0$ by \eqref{eq:positivity_f}.
Similarly, by \eqref{eq:positivity_F}--\eqref{eq:positivity_g},
\begin{align*}
\div (F_{i}^+(\cdot,u))&=  \div [\Fp_{i}^+(\cdot,u)-\Fp_i^+(\cdot,\wh{u}_i)],\\
g_{n,i}^+(\cdot,u)&= \gp_{n,i}(\cdot,u)-\gp_{n,i}(\cdot,\wh{u}_i).
\end{align*}
Recall that Lipschitz functions are weakly differentiable. Hence, for a Lipschitz function $R$ writing $R(u)-R(v)=\int_{0}^1 \frac{\dd }{\dd s}[R(u+s(v-u))] \,\dd s=\big( \int_0^1 R'(u+s(v-u))\,\dd s \big) (v-u)$, one can check that there exists bounded $\Progress\otimes \Borel(\Tor^d)$-measurable maps $r_{\fp_i},r_{\Fp_i} ,r_{\gp_i,n}$ (depending on $u$ on $[0,\tau^+)\times \O\times \Tor^d$) such that
\begin{equation}
\label{eq:linearization_positive_r_functions}
\begin{aligned}
&\Fp_i^+(\cdot, u)-\Fp_i^+(\cdot,\wh{u}_i)= r_{\Fp_i}u_i,\quad
\fp_i(\cdot,u)-\fp_i(\cdot,\wh{u}_i)=  r_{\fp_i} u_i, \quad
\gp_{n,i}(\cdot,u)-\gp_{n,i}(\cdot,\wh{u}_i)=  r_{\gp_i,n} u_i
\end{aligned}
\end{equation}
a.e.\ on $[0,\tau^+)\times \O\times \Tor^d$.

Now consider the following linearization of \eqref{eq:reaction_diffusion_system}:
\begin{equation}
\label{eq:v_problem_positivity}
\left\{
\begin{aligned}
\dd v_i -\div(a_i\cdot\nabla v_i) \,\dd t
&= \one_{[0,\tau^+)}\Big[\div(r_{\Fp_i}  v_i)+ r_{\fp_i,n} v_i +\fp_i(\cdot,\wh{u}_i)\Big]\,\dd t \\
&+ \sum_{n\geq 1}  \Big[(b_{n,i}\cdot \nabla) v_i+\one_{[0,\tau^+)} r_{\gp_i,n} v_i\Big]\,\dd w_t^n, & \text{ on }\Tor^d,\\
v_i(0)&=u_{i,0},  & \text{ on }\Tor^d.
\end{aligned}
\right.
\end{equation}
Let $v_i\in L^2(\O;C([0,j];L^2))\cap L^2((0,j)\times \O;H^1)$ be the global $(2,0,1,2)$-solution to the linear problem \eqref{eq:v_problem_positivity} (well-posedness follows from \cite[Chapter 4]{LR15}). By \eqref{eq:linearization_positive_r_functions}, $u^+_i$ is a solution to the problem \eqref{eq:v_problem_positivity} on $[0,\tau^+)$. Therefore, by uniqueness $u^+_i = v_i$ on $[0,\tau^+)$. Thus it remains to show $v_i\geq 0$ on $[0,j]$. By \eqref{eq:positivity_f}, the inhomogeneity satisfies $\one_{[0,\tau^+)}\fp_i(\cdot,\wh{u}_i)\geq 0$ a.e., and the coefficients of the linear parts are bounded,. Therefore, the conditions of the maximum principle of Lemma \ref{lem:maxprinciple} are fulfilled, and thus a.e.\ on
$[0,j]\times \Omega$, $v_{i}\geq 0$ on ${\mathbb T}^d$. Hence,  a.e.\ on
$[0,\tau^+]\times \Omega$, we have $u^+_i\geq 0$ on ${\mathbb T}^d$ as desired.
\end{proof}

\section{Higher order regularity}
In this section we briefly explain higher order regularity of the solution to \eqref{eq:reaction_diffusion_system} provided by Theorem \ref{t:reaction_diffusion_global_critical_spaces}.

The next assumption roughly says that $F, f$ and $(g_n)$ are $C^{\lceil \alpha+1\rceil}$ in the $y$-variable, where $\alpha>0$ is some fixed number.

\begin{assumption}
\label{ass:high_order_nonlinearities}
Let $\alpha>0$, $F, f$ and $g_n$ be as in Assumption \ref{ass:reaction_diffusion_global}\eqref{it:growth_nonlinearities}.
We assume that $F,f$ and $g_n$ are $x$-independent, $C^{\lceil \alpha+1 \rceil}$ in $y$ and, for all $N\geq 1$ there is a $C_N>0$ such that a.s.\
$$
\sum_{j=1}^{\lceil \alpha+1\rceil } |\partial_y^j F_i(t,y)|
+
|\partial_y^j f_i(t,y)|
+
\|(\partial_y^j g_{n,i}(t,y))_{n\geq 1}\|_{\ell^2}
\leq C_N, \ \  |y|\leq N, \ i\in \{1, \ldots, \ell\}, \ t\geq 0.
$$
\end{assumption}

\begin{theorem}[Higher order regularity]
\label{t:high_order_regularity}
Let the assumptions of Theorem \ref{t:reaction_diffusion_global_critical_spaces} be satisfied, where
$(\eta, \rho)$ are such that Assumption \ref{ass:reaction_diffusion_global}\eqref{it:regularity_coefficients_reaction_diffusion} holds, i.e. $\alpha>\max\{d/\rho,\s-1\}$, $\rho\in [2, \infty)$, and there exists an $N$ such that
\begin{align*}
\|a^{j,k}_i(t,\cdot)\|_{H^{\alpha,\rho}(\Tor^d)}
+
\|(\btwod_{n,i}^{j}(t,\cdot))_{n\geq 1}\|_{H^{\alpha,\rho}(\Tor^d;\ell^2)}
\leq N, \ \ t\geq 0, \ i\in \{1, \ldots, \ell\}, \ \text{a.s.}
\end{align*}
Furthermore, suppose that Assumption \ref{ass:high_order_nonlinearities} holds.
Let $(u,\sigma)$ be the $(p,\a_{\crit},\s,q)$-solution to \eqref{eq:reaction_diffusion_system} provided by Theorem \ref{t:reaction_diffusion_global_critical_spaces}. Then a.s.
\begin{align}
\label{eq:H_regularization_NS_improved}
u&\in H^{\theta,r}_{\rm loc}(0,\sigma;H^{1+\alpha-2\theta,\rho}(\Tor^d;\R^{\ell}))
\ \ \text{for all } \theta\in [0,1/2), \ r\in (2,\infty), \ \
\\ u& \in  C^{\theta_1,\theta_2+\alpha-\frac{d}{\rho}}_{\rm loc} ((0,\sigma)\times\Tor^d;\R^{\ell})  \ \ \text{for all }  \theta_1\in [0,1/2), \ \theta_2\in (0,1).
\label{eq:C_regularization_NS_improved}
\end{align}
\end{theorem}
From the above theorem one can see how the regularity of order $\eta$ of the coefficients appears in \eqref{eq:H_regularization_NS_improved} and \eqref{eq:C_regularization_NS_improved}. In particular, \eqref{eq:H_regularization_NS_improved} with $\theta=0$ shows that the regularity of $u$ is one order higher than the regularity of $(a,b,h)$.
In the above, we can also allow $x$-dependency of the nonlinearities $F_i$, $f_i$ and $g_{n,i}$ under suitable smoothness assumptions on the spatial variable.

\begin{remark}
If $u_0\in L^0_{\F_0}(\Omega;B^{1+\alpha - \frac{2}{r}}_{\rho,r}(\Tor^d;\R^{\ell}))$ for some fixed $r\in (2,\infty)$, then one can check from the proofs that the regularity result \eqref{eq:H_regularization_NS_improved} (for the fixed $r$) holds locally on $[0,\sigma)$ instead of $(0,\sigma)$. However, this will not be used in the sequel.
\end{remark}

To prove the result one can argue in the same way as in \cite[Theorem 2.7]{AV20_NS}. Similar as in the proof of \eqref{eq:reaction_diffusion_H_theta_1}, the ingredients in the proof are stochastic maximal $L^p$-regularity (see \cite{AV21_SMR_torus}) and mapping properties for the nonlinearities as we have encountered in the proof of Proposition \ref{prop:reaction_diffusion_global}. Since the proofs go through almost verbatim, details are left to the reader.

\section{Existence and uniqueness for large times in presence of small data}\label{sec:globsmall}
\label{s:global_small_data}
In this section we prove
 that the solution of reaction-diffusion equations provided by Theorem \ref{t:reaction_diffusion_global_critical_spaces} exists on large time intervals whenever the initial data is sufficiently small.

\begin{theorem}[Existence and uniqueness for large times in presence of small data]
\label{t:global_small_data}
Suppose that Assumptions \ref{ass:reaction_diffusion_global}$(p,q,h,\s)$ and  \ref{ass:admissibleexp}$(p,q,h,\s)$ hold and set $\a:=\a_ {\crit}:=p(\frac{h}{h-1}-\frac{1}{2}(\s+\frac{d}{q}))-1$.
Assume that there are $M_{1}, M_2>0$ such that a.s.\ for all $t\geq 0$ and $y\in \R^{\ell}$,
\begin{equation}
\begin{aligned}
\label{eq:condfggrowtheq}
|f(t,x,y)|&\leq M_1 + M_2(|y|+|y|^h),\\
 |F(t,x,y)| + \|(g_n(t,x,y))_{n\geq 1}\|_{\ell^2} &\leq M_1  + M_2(|y|+|y|^{\frac{h+1}{2}}).
\end{aligned}
\end{equation}
Fix $u_0\in L^p_{\F_0}(\Omega;B^{\frac{d}{q}-\frac{2}{h-1}}_{q,p})$.
Let $(u,\sigma)$ be the $(p,\a_{\crit},\s,q)$-solution to \eqref{eq:reaction_diffusion_system} provided by Theorem \ref{t:reaction_diffusion_global_critical_spaces}. For all $\varepsilon\in (0,1)$ and $T\in (0,\infty)$, there exists $C_{\varepsilon,T}>0$, independent of $u_0$ such that
\begin{equation}
\label{eq:smallness_implies_global}
\E\|u_0\|_{B^{\frac{d}{q}-\frac{2}{h-1}}_{q,p}}^p+M_1^p\leq C_{\varepsilon,T}
\quad \Longrightarrow \quad \P(\sigma > T)>1-\varepsilon.
\end{equation}
\end{theorem}

Roughly speaking,
Theorem \ref{t:global_small_data} shows that if $u_0$ and $M_1$ are close to $0$, then $u$ exists up to $T$ with probability $>1-\varepsilon$. Reasoning as in Remark \ref{r:reaction_diffusion_critical_spaces}\eqref{it:L_xi_data}, the above result also implies existence for large time of unique solutions with small data in $L^{\frac{d}{2}(h-1)}(\Tor^d;\R^{\ell})$.
Under the assumptions of Theorem \ref{t:global_small_data}, the proof below also yields the following assertion:

If $(u_0,M_1)$ satisfies the condition on LHS\eqref{eq:smallness_implies_global},
then there exists a stopping time $\tau\in (0,\sigma]$ a.s.\ such that $\P(\tau\geq T)>1-\varepsilon$ and
\begin{equation}
\label{eq:stopping_time_tau_estimate}
\E\Big[\one_{\{\tau\geq T\}}\|u\|^p_{H^{\theta,p}(0,T,w_{\a_{\crit}};H^{2+\s-2\theta,q})}\Big]
\lesssim_{\theta} \E\|u_0\|_{B^{\frac{d}{q}-\frac{2}{h-1}}_{q,p}}^p+M_1^p, \ \text{ for all }\theta\in [0,\tfrac{1}{2}).
\end{equation}
To prove Theorem \ref{t:global_small_data} and \eqref{eq:stopping_time_tau_estimate} one can modify the arguments used in the proof of \cite[Theorem 2.11(1)]{AV20_NS} and \cite[Theorem 2.11(2)]{AV20_NS}, respectively. Instead of repeating the technical iteration argument used in \cite[Theorem 2.11]{AV20_NS}, we present an alternative approach under the additional assumption that $u\geq 0$, and the \emph{mass conservation} property: there exist $\alpha_1, \ldots, \alpha_\ell, C_0>0$ such that, for all $t\geq 0$, $x\in \Tor^d$ and $y\in [0,\infty)^{\ell}$,
\begin{equation}
\label{eq:mass_conservation_ass}
\sum_{i=1}^{\ell}\alpha_i f_i(t,x,y) \leq C_0\Big(1+\sum_{i=1}^{\ell} y_i\Big).
\end{equation}
Both conditions are natural for reaction-diffusion equations, see
Subsection \ref{ss:mass} and \cite{P10_survey}.

Due to assumption \eqref{eq:mass_conservation_ass} we can control the lower order term $M_2|y|$ on the RHS\eqref{eq:condfggrowtheq} by exploiting the mass balance, i.e.\ for all $T<\infty$ and $i\in \{1,\dots,\ell\}$,
\begin{equation}
\label{eq:mass_conservation}
\E\int_{\Tor^d} u_i(\tau,x)\,\dd x\lesssim_T \E\int_{\Tor^d} u_{0,i}(x) \,\dd x, \quad \text{ for any stopping time }\tau\in (0,\sigma\wedge T].
\end{equation}
We refer to Step 1 in the proof of Theorem \ref{t:global_small_data} for the precise statement.

Before going into the proof of the simplified version of Theorem \ref{t:global_small_data}, we introduce some more notation. Recall that $(X_{\lambda},A,B,\FS,\GS)$ and $\a_{\crit}=p(\frac{h}{h-1}-\frac{1}{2}(\s+\frac{d}{q}))-1$ have been introduced in \eqref{eq:def_X_theta}--\eqref{eq:ABFG_def} and Theorem \ref{t:reaction_diffusion_global_critical_spaces}, respectively.
Moreover, $\Xapcrit=B^{\frac{d}{q}-\frac{2}{h-1}}_{q,p}$, and for $\beta_1,\beta_2$ as in Lemma \ref{l:estimate_nonlinearities} (with $q<\frac{d(h-1)}{\s}$ and thus $q< \frac{d(h-1)}{2(\s-1)}$), we let
\begin{equation}
\label{eq:def_X_space}
\X(t):=L^{h p} (0,t,w_{\a_{\crit}};X_{\beta_1})\cap L^{\frac{h+1}{2}p } (0,t,w_{\a_{\crit}};X_{\beta_2}).
\end{equation}
One can readily check that the above space coincides with the one introduced in \cite[Subsection 4.3, eq.\ (4.14)]{AV19_QSEE_1}. By \cite[Lemma 4.19]{AV19_QSEE_1}, there exists $\theta\in [0,\frac{1}{2})$ such that
\[ H^{\theta,p}(0,t;w_{\a_{\crit}};X_{1-\theta})\cap L^p(0,t,w_{\a_{\crit}};X_1)\subseteq \X(t), \quad t>0.\]
In particular, the solution $(u, \sigma)$ provided by Theorem \ref{t:reaction_diffusion_global_critical_spaces} satisfies a.s.\ for all $t\in (0,\sigma)$,  $u\in \X(t)$.

As in \cite{AV20_NS}, we need the following special case of \cite[Lemma 5.3]{AV19_QSEE_2} and the maximal $L^p$--regularity estimates of \cite{AV21_SMR_torus}.

\begin{lemma}
\label{l:linear_estimate_X}
Let Assumption \ref{ass:reaction_diffusion_global}$(p,q,h,\s)$ be satisfied. Fix $T\in (0,\infty)$.
Let $(A,B)$ be as in \eqref{eq:ABFG_def}.
Then there exists $K> 0$ such that for every stopping time $\tau\in [0,T]$, every
$$
v_0\in L^0_{\F_0}(\O;X^{\mathsf{Tr}}_{\a_{\crit},p}),\  f\in L^p_{\Progress}(( 0,\tau)\times \O,w_{\a_{\crit}};X_0),
\  g\in L^p_{\Progress}(( 0,\tau)\times \O,w_{\a_{\crit}};\g(\ell^2,X_{1/2})),
$$
and every $(p, \a_{\crit},q,\delta)$-solution $v\in L^p_{\Progress}(( 0,\tau)\times \O ,w_{\a_{\crit}};X_1)$ to
\begin{equation*}
\left\{
\begin{aligned}
\dd v + A v \,\dd t&=f \,\dd  t+ \big(Bv +  g\big)\, \dd W_{\ell^2},
\\ v(0)&=v_0,
\end{aligned}\right.
\end{equation*}
on $( 0,\tau)\times \O$, the following estimate holds
\begin{align*}
\|v\|_{L^p(\Omega;\X(\tau))}^p \leq K^p\big( \|v_0\|_{L^p(\Omega;X^{\mathsf{Tr}}_{\a_{\crit},p})}^p
+\|f\|_{L^p(( 0,\tau)\times \O,w_{\a_{\crit}};X_{0})}^p
 +\|g\|_{L^p(( 0,\tau)\times \O,w_{\a_{\crit}};\g(\ell^2,X_{1/2}))}^p\big).
\end{align*}
\end{lemma}

\begin{proof}[Proof of Theorem \ref{t:global_small_data} -- Case $u\geq 0$ a.e.\ on $[0,\sigma)\times \O\times \Tor^d$ and the mass conservation \eqref{eq:mass_conservation_ass} holds]
Through the proof we fix $\varepsilon\in (0,1)$ and $T\in (0,\infty)$. Let
\begin{equation*}
\sigma_n:=\inf\big\{t\in [0,\sigma)\,:\,\|u\|_{L^p(0,t,w_{\a_{\crit}};X_1)}+ \|u\|_{\X(t)}\geq n\big\}\wedge T, \ \ n\geq 1,
\end{equation*}
where $\inf\emptyset:=\sigma$. We split the proof into several steps.

{\em Step 1: (Mass conservation). There exists $L>0$, depending only on $(C_0,\alpha_i,T)$ in \eqref{eq:mass_conservation_ass} such that,  for all $t\in [0,T]$ and $n\geq 1$,}
\begin{equation}
\label{eq:mass_conservation_small_global}
\E\int_{\Tor^d} \sum_{i=1}^\ell u_i(t\wedge \sigma_n,x)\,\dd x
\leq L \E\int_{\Tor^d} \sum_{i=1}^\ell u_{0,i}(x)\,\dd x.
\end{equation}
On the RHS\eqref{eq:mass_conservation_small_global} we understood $\displaystyle{\int_{\Tor^d} u_0(x)\,\dd x:=\l \one_{\Tor^d},u_0\r}$. Note that
$
\displaystyle{\int_{\Tor^d} u_0(x)\,\dd x \lesssim \|u_0\|_{\Xapcrit}.}
$

To see \eqref{eq:mass_conservation_small_global} it is enough to stop \eqref{eq:reaction_diffusion_system} at time $t\wedge \sigma_n$, multiply each equation in \eqref{eq:reaction_diffusion_system} by $\alpha_i$ and then sum them up. After integrating over $\T^d$ and canceling the divergence terms and martingale terms, and using the mass conservation \eqref{eq:mass_conservation_ass}, we find that
\begin{align*}
&\E \int_{\Tor^d} \sum_{i=1}^\ell  \alpha_i u_i(t\wedge \sigma_n,x)\,\dd x\\
&=
\E \int_{\Tor^d} \sum_{i=1}^\ell \alpha_i u_{0,i}(x)\,\dd x
+ \E  \int_{\Tor^d}\int_0^{t\wedge \sigma_n} \alpha_i f_i(s, x, u) \dd s\, \dd x
\\ & \leq
\E \int_{\Tor^d} \sum_{i=1}^\ell \alpha_i u_{0,i}(x)\,\dd x + C_0 \E \int_0^{t\wedge \sigma_n}\Big(1+ \int_{\Tor^d} \sum_{i=1}^{\ell} u_i(s,x)\,\dd x\Big) \dd s
\\ & \leq \E \int_{\Tor^d} \sum_{i=1}^\ell \alpha_i u_{0,i}(x)\,\dd x + C_0 \E \int_0^{t}\Big(1+ \int_{\Tor^d} \sum_{i=1}^{\ell} u_i(s\wedge \sigma_n,x)\,\dd x\Big) \dd s,
\end{align*}
where we used the positivity of $u_i$ in the last step, which holds by assumption.
Now \eqref{eq:mass_conservation_small_global} follows from Gronwall's inequality applied to the function $U(s):=\E \int_{\T^d}\sum_{i=1}^{\ell} u_i(s\wedge \sigma_n,x) \dd x$ and the fact that $\alpha_i>0$.

\emph{Step 2: (Estimates for the nonlinearities). Let $K$ be as in Lemma \ref{l:linear_estimate_X}. There exists $c_0,c_1>0$ independent of $M_1$ such that, for all stopping times $0\leq \mu\leq \sigma\wedge T$ a.s., one has}
\begin{align*}
\E\|\FS(\cdot,u)\|_{L^p(0,\mu,w_{\a_{\crit}};H^{2-\s,q})}^p
&+
\E\|\GS(\cdot,u)\|_{L^p(0,\mu,w_{\a_{\crit}};H^{1-\s,q}(\ell^2))}^p \\
&\leq c_0 M_1^p + \frac{1}{2K^p} \E\|u\|_{\X(\mu)}^p+ c_1M_2^p \big(\E\|u_0\|_{\Xapcrit} + \E\|u\|_{\X(\mu)}^{ph}\big),
\end{align*}
\emph{
where $(\FS,\GS)$ is as in \eqref{eq:ABFG_def}. Finally, $c_0$ is also independent of $M_2$.}

Recall that $\Phi=f+\div(F)$. We only provide the details for the estimate of $f$. The estimates for $\div F$ and $\GS$ are similar.
Following the proof of the $\Phi$--estimates in \eqref{eq:estimate_rone_reaction_diffusion}, we obtain that
\begin{equation}
\label{eq:estimate_f_small}
\|f(\cdot,v)\|_{H^{2-\s,q}}
\leq c_0( M_1 + M_2\|v\|_{L^{\xi}}+ M_2\|v\|_{L^{h \xi}}^h),  \ \  v\in H^{2-\s,q},
\end{equation}
where $c_0$ is a constant independent of $(M_1,M_2,v)$ and $\xi$ is as in \eqref{eq:estimate_rone_reaction_diffusion}.

By Fatou's lemma it is enough to show the claim of Step 2 where $\mu$ is replaced by $\mu_n:=\sigma_n\wedge \mu$ for $n\geq 1$ and with constants independent of $n\geq 1$.
Hence, by using \eqref{eq:estimate_f_small} we have
\begin{align*}
\E\|f(\cdot,u)\|_{L^p(0,\mu_n,w_{\a_{\crit}};H^{2-\s,q})}^p
&\leq c_0 T M_1+ c_1 M_2\big( \E\|u\|_{L^{p}(0,\mu_n,w_{\a_{\crit}};L^{\xi})}^p+ \E\|u\|_{L^{ph}(0,\mu_n,w_{\a_{\crit}};L^{h \xi})}^{ph}).
\end{align*}
Next we conveniently estimate the lower order term $\E\|u\|_{L^{p}(0,\mu_n;L^{\xi})}^p$ appearing on the RHS of the above estimate.
Let $\lambda>0$ be arbitrary for the moment. Note that, by standard interpolation,
\begin{align*}
\E\|u\|_{L^{p}(0,\mu_n,w_{\a_{\crit}};L^{\xi})}^p
&\leq \frac{1}{\lambda} \E\|u\|_{L^{p}(0,\mu_n,w_{\a_{\crit}};L^{h\xi})}^p
+ C_1\E\|u\|_{L^{p}(0,\mu_n;L^{1})}^p\\
&\leq \frac{1}{\lambda} \E\|u\|_{L^{p}(0,\mu_n,w_{\a_{\crit}};L^{h\xi})}^p
+C_2\big( \E\|u\|_{L^{1}(0,\mu_n,w_{\a_{\crit}};L^{1})} +\E\|u\|_{L^{ph}(0,\mu_n,w_{\a_{\crit}};L^{1})}^{ph}\big),
\end{align*}
where $C_1,C_2$ are constants which depend only on $(p,c_1,\lambda,M_2,h,\xi,d)$ and we used \eqref{eq:def_X_space}.
Now, let $C_T$ be the constant of the embedding
\[L^{h p} (0,t,w_{\a_{\crit}};X_{\beta_1})\embed L^{p} (0,t,w_{\a_{\crit}};H^{-\delta+2\beta_1,q}) \embed  L^p(0,t,w_{\a_{\crit}};L^{h\xi})\]
for any $t\in (0,T]$. Then, the previous shows
$$
\E\|u\|_{L^{p}(0,\mu_n,w_{\a_{\crit}};L^{\xi})}^p\leq  \frac{C_T}{ \lambda} \E\|u\|_{\X(\mu_n)}^p
+C_2\big( \E\|u\|_{L^{1}(0,\mu_n,w_{\a_{\crit}};L^{1})} +\E\|u\|_{\X(\mu_n)}^{ph}\big) .
$$
To conclude, recall that, $\mu_n\leq \sigma_n\leq T$ a.s.\ and therefore
\begin{align*}
\E\|u\|_{L^{1}(0,\mu_n,w_{\a_{\crit}};L^{1})}
& \leq  \E\int_0^{T}\int_{\Tor^d} |u(t\wedge \sigma_n,x)| w_{\a_{\crit}}(t)\,\dd x \, \dd t \lesssim \E\|u_0\|_{\Xapcrit},
\end{align*}
where in the last inequality we used Step 1. Note that, by Step 1, the implicit constant in the above estimate is independent of $n\geq 1$ as desired. Putting together the previous estimates, one obtains the claim for $f(\cdot,u)$ by choosing $\lambda$ large enough. The remaining ones are similar.

\emph{Step 3: Let $K$ be as in Lemma \ref{l:linear_estimate_X}. Then there exists $\Constant>0$, independent of  $(M_1,u_0)$, such that for any $N\geq 1$ and any stopping time $\mu$ satisfying $0\leq  \mu \leq \sigma\wedge T$ and $\|u\|_{\X(\mu)}\leq N$ a.s.,}
\begin{align*}
\E \big[\psi_R(\|u\|_{\X(\mu)}^p)\big]\leq\E\|u_0\|_{\Xapcrit}+ \E\|u_0\|_{\Xapcrit}^p+ M_1^p, \ \ \text{with} \ \ \psi_R(x)=\frac{x}{R}-x^{h}.
\end{align*}

As in the previous step we may prove the claim with $\mu$ replaced by $\mu_n:=\sigma_n\wedge \mu$ since $\|u\|_{\X(\mu)}\leq N$ a.s.\ for some $N \geq 1$.
The estimates of Lemma \ref{l:linear_estimate_X} and Step 3 readily implies
\begin{align*}
\E &\|u\|_{\X(\mu_n)}^p
\leq K^p\big( \E \|u_0\|_{\Xapcrit}^p+ \E\|\FS(\cdot,u)\|_{L^p(0,\mu_n,w_{\crit};X_0)}^p  +\E\|\GS(\cdot,u)\|_{L^p( 0,\mu_n,w_{\crit};X_{1/2})}^p
\big)
\\ &\leq K^p c_0 M_1^p + K^p\big( \E\|u_0\|_{\Xapcrit} +c_1M_2\E\|u_0\|_{\Xapcrit}^p\big) + \frac{1}{2} \E \|u\|_{\X(\mu_n)}^p + K^pc_1 M_2 \E\|u\|_{\X(\mu_n)}^{ph}.
\end{align*}
Since $\|u\|_{\X(\mu_n)}\leq n$ a.s.\ by definition of $\sigma_n$, the term $\frac{1}{2} \E \|u\|_{\X(\mu_n)}^p$ can be absorbed on the LHS and hence
\begin{align*}
\E \|u\|_{\X(\mu_n)}^p
\leq 2K^p c_0 M_1^p +2K^p( \E\|u_0\|_{\Xapcrit} +c_1M_2\E\|u_0\|_{\Xapcrit}^p)  + 2K^p c_1 M_2 \E\|u\|_{\X(\mu_n)}^{ph}.
\end{align*}
Letting $n\to \infty$, the desired estimate follows after division by $R=2K^p\max\{c_0,1+c_1M_2\}$.

{\em Step 4: (A reduction). To prove Theorem \ref{t:global_small_data} (i.e.\ the implication \eqref{eq:smallness_implies_global}) it is sufficient to prove the existence of $C_{\varepsilon,T},r_T>0$ independent of $u_0$ such that}
\begin{equation}
\label{eq:alternative_implication_global_perturbation}
\E\|u_0\|_{\Xapcrit}+\E\big\|u_0\big\|_{\Xapcrit}^p + M_1^p
\leq C_{\varepsilon,T}
\ \
\Longrightarrow \ \
\P(\W)>1-\varepsilon,
\end{equation}
\emph{where}
$$
\W = \{\|u\|_{\X(\sigma\wedge T)}\leq r_{T}\}.
$$
Arguing as in Step 2 (cf.\ \eqref{eq:estimate_f_small} and the text before it), one can check that there exists $C_{*}$ depending only on $(M_1,M_2,K,c_0,c_1,r_T)$ such that
\begin{equation*}
\|\FS(\cdot,u)\|_{L^p(0,\mu,w_{\a_{\crit}};H^{2-\s,q})}
+
\|\GS(\cdot,u)\|_{L^p(0,\mu,w_{\a_{\crit}};H^{1-\s,q}(\ell^2))}  \leq C_{*} \ \ \text{on $\W$},
\end{equation*}
where $\mu\in [0,\sigma]$ is a stopping time.
Define the stopping time $\tau$ by
$$
\tau = \inf\big\{t\in [0,\sigma): \|\FS(\cdot,u)\|_{L^p(0,\mu,w_{\a_{\crit}};H^{2-\s,q})}+ \|\GS(\cdot,u)\|_{L^p(0,\mu,w_{\a_{\crit}};H^{1-\s,q}(\ell^2))}\geq C_{*}+1\big\}\wedge T,
$$
where we set $\inf\emptyset =\sigma\wedge T$. Then $\tau = \sigma\wedge T$ on $\W$.

By \cite[Theorem 1.2]{AV21_SMR_torus}, $(A,B)$ has stochastic maximal $L^p$-regularity. Since $(u,\sigma)$ is a $(p,\a_{\crit},q,\s)$-solution to \eqref{eq:reaction_diffusion_system}, as in \cite[Proposition 3.12(2)]{AV19_QSEE_1} it follows that a.s. on $[0,\tau)$
\begin{align*}
\dd u +A u\, \dd t = \one_{[0,\tau)} F(\cdot, u)\, \dd t + \big(Bu + \one_{[0,\tau)} G(\cdot, u)\big)\, \dd W_{\ell^2},
\end{align*}
and $u(0) = u_0$.
Now \cite[Theorem 1.2]{AV21_SMR_torus} (see also Theorem 5.2 there) also gives
\begin{equation}
\label{eq:regularity_small_implies_global_tau}
u\in L^p(\O;H^{\frac{\a_{\crit}}{p},p}(0,\tau,w_{\a_{\crit}};X_{1-\frac{\a_{\crit}}{p}}))\cap L^p(\O; C([0,\tau];\Xapcrit)).
\end{equation}
Using $\tau=\sigma\wedge T$ on $\W$, by Sobolev embedding \cite[Proposition 2.7]{AV19_QSEE_1} we obtain
\begin{equation}
\label{eq:upathproprtylocO}
\begin{aligned}
u\in H^{\frac{\a_{\crit}}{p},p}(0,\sigma\wedge T,w_{\a_{\crit}};X_{1-\frac{\a_{\crit}}{p}})
&\hookrightarrow L^p(0,\sigma\wedge T;X_{1-\frac{\a_{\crit}}{p}})\\
& = L^p(0,\sigma\wedge T;H^{\gamma,q}) \quad \text{a.s.\ on $\W$},
\end{aligned}
\end{equation}
where $\gamma = 1+\delta-2\frac{\a_{\crit}}{p}=\frac{2}{p}+\frac{d}{q}-\frac{2}{h-1}$. Let $\beta=\frac{d}{q}-\frac{2}{h-1}$,  and note that $\Xapcrit =B^{\beta}_{q,p}$, see \eqref{eq:Besov_spaces}. Thus \eqref{eq:regularity_small_implies_global_tau} also implies
\begin{equation}
\label{eq:upathproprtylocO_2}
u\in C([0,\sigma\wedge T];B^{\beta}_{q,p}) \ \ \text{a.s.\ on $\W$}.
\end{equation}
Hence, it follows that
\begin{equation*}
\begin{aligned}
&\P\big(\{\sigma\leq T\}\cap \W\big)\\
&\stackrel{(i)}{=}\lim_{s\downarrow 0}\P\Big(\Big\{s<\sigma\leq T,\,
 \sup_{t\in [s,\sigma)}\|u(s)\|_{B^{\beta}_{q,p}} + \|u\|_{L^p(s,\sigma;H^{\gamma,q})}<\infty\Big\}\cap\W \Big)\\
&\leq  \limsup_{s\downarrow 0}
\P\Big(s<\sigma\leq T ,\, \sup_{t\in [s,\sigma)}\|u(s)\|_{B^{\beta}_{q,p}}+ \|u\|_{L^p(s,\sigma;H^{\gamma,q})}<\infty\Big)\stackrel{(ii)}{=} 0.
\end{aligned}
\end{equation*}
Here in $(i)$ we used $\sigma>0$ a.s.\ (see Theorem \ref{t:reaction_diffusion_global_critical_spaces}) and \eqref{eq:upathproprtylocO}-\eqref{eq:upathproprtylocO_2}. In $(ii)$ we used Theorem \ref{t:blow_up_criteria}\eqref{it:blow_up_sharp} with $p_0 = p$, $q_0 = q$, $\gamma_0 = \gamma$ and $\beta_0=\beta$ (see also the comments below \eqref{eq:include_endpoint_T} on the set $\{\sigma=T\}$). Therefore, $\sigma >  T$ on $\W$ and therefore we showed that \eqref{eq:alternative_implication_global_perturbation} implies the claim of Theorem \ref{t:global_small_data}.

\textit{Step 5: Conclusion, i.e.\ \eqref{eq:alternative_implication_global_perturbation} holds}.
Let $\psi_R$ be as in Step 3. It is easy to check that
$\psi_R$ has a unique maximum on $\R_+$ attained in $\xm:= (Rh)^{-1/(h-1)}$ and it is given by $\psim:=\frac{\xm}{R}\frac{h-1}{h}$. Set $r_{\varepsilon,T} = \xm^{-1/p}$ and hence
$\W=\{\|u\|_{\X(\sigma)}\leq \xm^{-1/p}\}$. Define
\begin{align}
\label{eq:def_lambda_proof_small}
\mu&:=
\inf\{t\in [0, \sigma)\,:\, \|u\|_{\X(t)}\geq \xm^{-1/p}\}\wedge T,
\end{align}
where  $\inf\emptyset:=\sigma\wedge T$.
We prove \eqref{eq:alternative_implication_global_perturbation} with $C_{\varepsilon,T} = \frac{\varepsilon \psim}{2}$.
To derive a contradiction suppose that
\begin{equation}
\label{eq:smallness_initial_data}
\E\|u_0\|_{\Xapcrit}+\E\|u_0\|_{\Xapcrit}^p+M_1^p\leq\frac{\varepsilon \psim}{2}, \ \ \ \
\text{and}  \ \
\ \ \P(\W)\leq 1-\varepsilon.
\end{equation}
From the definition of $\W$ and \eqref{eq:def_lambda_proof_small}, we find that $\mu< \sigma\wedge T$ a.s.\ on $ \O\setminus \W$. Moreover,
\begin{align}
\label{eq:f_reach_max_1}
\psi_R(\|u\|_{\X(\mu)}^p)&=\psim \ \ \text{ a.s.\ on }\O\setminus \W, \ \  \text{and} \ \
  \psi_R(\|u\|_{\X(\mu)}^p) \geq 0,\   \text{ a.s.\ on } \W.
 \end{align}
Therefore,
\begin{align*}
\E\Big[\psi_R(\|u\|_{\X(\mu)}^p)\Big]
&=\E\Big[\psi_R(\|u\|_{\X(\mu)}^p)\one_{ \W}\Big]
+\E\Big[\psi_R(\|u\|_{\X(\mu)}^p)\one_{\O\setminus  \W}\Big]\\
&\stackrel{\eqref{eq:f_reach_max_1}}{\geq}
\P(\O\setminus  \W) \psim \stackrel{\eqref{eq:smallness_initial_data}}{\geq } \varepsilon \psim\stackrel{\eqref{eq:smallness_initial_data}}{\geq } 2\big( \E\|u_0\|_{\Xapcrit}^p+M_1^p\big).
\end{align*}
The latter contradicts Step 2. Thus $\P(\W)>1-\varepsilon$ as desired.
\end{proof}

\section{Extension to the one-dimensional case}\label{sec:1dcase}
\label{s:one_d}
Many of the results of the previous sections extend to the one-dimensional setting. However, different restrictions on the parameters will appear. The reason for this is that certain sharp Sobolev embeddings become invalid. An example is the condition on $\xi$ in \eqref{eq:estimate_rone_reaction_diffusion}: $-\frac{d}{\xi} = -\delta-\frac{d}{q}$. The latter can no longer hold for $\delta\in [1, 2)$ and $d=1$, and therefore one takes the best possible choice $\xi=1$, which in turn leads to other conditions on $h$ and $\delta$ in the Sobolev embedding $H^{\theta,q}\hookrightarrow L^{h\xi}$ used in \eqref{eq:estimate_rone_reaction_diffusion}. Similar changes are needed for Subset 2. As these restrictions lead to sub-optimal exponents, it is not really interesting to consider critical spaces anymore. Therefore, there is no need to state Theorem \ref{t:reaction_diffusion_global_critical_spaces} for $d=1$. However, we will include the conditions on the exponents under which the one-dimension variant of Proposition \ref{prop:reaction_diffusion_global} holds:
\begin{proposition}[Local existence, uniqueness, and regularity for $d=1$]
\label{prop:reaction_diffusion_globald=1}
Let Assumption \ref{ass:reaction_diffusion_global}$(p,q,h,\s)$ be satisfied for $d=1$.
Suppose that $q\geq 2$ and $\frac1q - \frac1h<2-\delta$ and that one of the following holds:
\begin{enumerate}[(1)]
\item\label{it1:reaction_diffusion_globald=1} $\delta+\frac1q >2$ and $\frac{1+\a}{p}\leq \frac{h}{h-1}\min\Big\{1-\frac{\delta}{2}, 1-\frac{\delta}{2} + \frac{1}{2h} - \frac{1}{2q}\Big\}$.
\item\label{it2:reaction_diffusion_globald=1} $\delta+\frac1q <2$ and
$\frac{1+\a}{p}\leq \frac{h}{h-1}\min\Big\{1-\frac{\delta}{2}, 1-\frac{\delta}{2} + \frac{1}{2h} - \frac{1}{2q}, 1-\frac{h-1}{2h} (\delta+\frac1q)\Big\}$.
\end{enumerate}
Then for any
$u_0\in L^0_{\F_0}(\O;B^{2-\reg-2\frac{1+\a}{p}}_{q,p})$,
\eqref{eq:reaction_diffusion_system} has a (unique) $(p,\a,\s,q)$-solution satisfying a.s.\ $\sigma>0$ and
\begin{equation*}
u\in L^p_{\rm loc}([0,\sigma),w_{\a};H^{2-\reg,q})\cap C([0,\sigma);B^{2-\reg-2\frac{1+\a}{p}}_{q,p}).
\end{equation*}
Moreover, $u$ instantaneously regularizes
\begin{align*}
u&\in H^{\theta,r}_{\rm loc}(0,\sigma;H^{1-2\theta,\zeta})  &\text{a.s.\ for all }\theta\in  [0,1/2), \ & r,\zeta\in (2,\infty),\\
u&\in C^{\theta_1,\theta_2}_{\rm loc}((0,\sigma)\times \Tor^d;\R^{\ell}) &\text{a.s.\ for all }\theta_1\in  [0,1/2), \ &\theta_2\in (0,1).
\end{align*}

Moreover, the assertions of Proposition \ref{prop:local_continuity_general} hold under these conditions as well.
\end{proposition}
We left out the case $\delta+\frac1q =2$ since it leads to slightly different conditions because one needs $\eta = 1+\varepsilon$ in this case, because the Sobolev embedding $L^1\hookrightarrow H^{1-\delta, q}$ does not hold for the $L^1$-endpoint (here $\eta$ is as in Substep 1b of Lemma \ref{l:estimate_nonlinearities}).

\begin{proof}
First we discuss the required changes in the proof of Lemma \ref{l:estimate_nonlinearities}.
Taking $\xi = 1$ in \eqref{eq:estimate_rone_reaction_diffusion} we can set $\theta = \max\{\frac1q - \frac1h,0\}$.
We need the condition $\frac1q - \frac1h<2-\delta$ to ensure $\FS_0$ is of lower order. This leads to the choice
\begin{align}\label{eq:beta1d=1}
\beta_1 = \max\Big\{\frac{1}{2q}-\frac{1}{2h},0\Big\} + \frac{\delta}{2}.
\end{align}

For $\FS_1$ and $\GS$ we consider two cases:

\emph{Case $\delta+\frac1q >2$}. In this case we choose $\eta=1$ and $\phi = \max\{0, \frac{1}{q} - \frac{2}{h+1}\}$. We need the condition $\frac{1}{q} - \frac{2}{h+1}<2-\delta$ to ensure that $\FS_1$ is of lower order, but the latter is automatically satisfied since $\frac{2}{h+1}>\frac1h$. This leads to
\[\beta_2 = \max\Big\{\frac{1}{2q}-\frac{1}{h+1},0\Big\} + \frac{\delta}{2}.\]
It turns out that the sub-criticality condition (see \eqref{eq:critical_condition_j})
\begin{equation}
\label{eq:critical_condition_jd=1}
\frac{1+\a}{p}\leq \frac{\rho_j+1}{\rho_j}(1-\beta_j) \ \ \ \text{ for }j\in \{1,2\},
\end{equation}
is most restrictive for $j=1$, and this leads to the condition as stated in \eqref{it1:reaction_diffusion_globald=1}.

\emph{Case $\delta+\frac1q<2$}. In this case we can take $\eta$, $\phi$ and $\beta_2$ as in the proof of Lemma \ref{l:estimate_nonlinearities}. Since $\delta+\frac1q <2<\frac{2h}{h-1}$, elementary computations show that the condition $\phi<2-\delta$ is automatically satisfied. This time the condition \eqref{eq:critical_condition_jd=1} gets an additional restriction as stated in \eqref{it2:reaction_diffusion_globald=1}. It only plays a role if $q<\frac{h-1}{2(\delta-1)}$.

Now the rest of the assertions follow in the same way as in Propositions \ref{prop:reaction_diffusion_global} and \ref{prop:local_continuity_general}.
\end{proof}

The following analogues of the previous results hold in the case $d=1$ as well:
\begin{remark}\label{rem:d=1}
Let the conditions of Proposition \ref{prop:reaction_diffusion_globald=1} be satisfied with exponents $(p, q, h, \delta,\kappa)$.
\begin{enumerate}
\item {\em (Blow-up criteria)}. Suppose that Assumption \ref{ass:reaction_diffusion_global}$(p_0, q_0, h_0, \delta_0)$ holds with $h_0\geq h$, and that Proposition \ref{prop:reaction_diffusion_globald=1}\eqref{it1:reaction_diffusion_globald=1} or \eqref{it2:reaction_diffusion_globald=1} hold for $(p_0, q_0, h_0, \delta_0, \kappa_0)$. Let $\beta_0= 2-\delta_0-2\frac{1+\kappa_0}{p_0}$ and $\gamma_0 = 2-\delta_0-\frac{2\kappa_0}{p_0}$.
Then for all $0<s<T<\infty$, Theorem \ref{t:blow_up_criteria}\eqref{it:blow_up_not_sharp}-\eqref{it:blow_up_sharp} for $d=1$ hold.
\item {\em (Positivity)}. The assertion of Theorem \ref{thm:positivity} holds for $d=1$ if the conditions of Theorem \ref{t:reaction_diffusion_global_critical_spaces} are replaced by the conditions of Proposition \ref{prop:reaction_diffusion_globald=1}.
\item In a similar way Theorems \ref{t:high_order_regularity} and \ref{t:global_small_data}  hold for $d=1$ in the setting of Proposition \ref{prop:reaction_diffusion_globald=1}. Here Assumption \ref{ass:admissibleexp} should be omitted  and $\kappa$ should be as in Proposition \ref{prop:reaction_diffusion_globald=1} instead of $\kappa_\crit$. Some changes are required in the arguments.
\end{enumerate}
\end{remark}

In Remark \ref{r:basic_assumptions}\eqref{it:d=1} we mentioned an alternative way to include the case $d=1$ by adding a dummy variable. However, this leads to additional restrictions on the parameters.

\section{Extensions to the case $p=q=2$}\label{sec:p=q=2}

In this section we explain how to extend the results of the previous sections to $p=q=2$ and $\kappa=0$. This end-point case follows in the same way as in \cite{AV22_variational}, where we cover the so-called variational setting which can be seen as an abstract version of the case $p=q=2$ and $\kappa=0$.
Its importance lies in the fact that it often allows to prove energy bounds which lead to global existence. All results in Sections \ref{s:main_results} and \ref{s:one_d} extend to $p=q=2$ and $\kappa=0$ under suitable restrictions which we explain below.

The variational framework is very effective in the weak setting (i.e.\ $\delta = 1$), where coercivity conditions are easy to check. The case $\delta\in (1, 2)$ allowed in Theorem \ref{t:reaction_diffusion_global_critical_spaces}, is typically not included as the fractional scale leads to difficulties with coercivity conditions. The results of this section (e.g.\ Proposition \ref{prop:globalp2q2rough}) might be combined to some of the results in \cite{AV22_variational} for instance to allow rougher initial data and/or to obtain higher order regularity (see Theorem \ref{t:reaction_diffusion_global_critical_spaces} and  \ref{t:high_order_regularity}, respectively).
However, one should be aware that using the case $p=q=2$ and $\kappa=0$ requires low dimension, or nonlinearities which do not grow too rapidly (see Subsection \ref{sss:scaling} and \cite[Subsection 5.2]{AVreaction-global}).

As in \cite[Subsection 5.3]{AV22_variational} one can check that Definition \ref{def:solution} can be extended to $p=q=2$, $\kappa=0$ and $\delta=1$ if Assumption \ref{ass:reaction_diffusion_global}\eqref{it:reaction_diffusion_global1},\eqref{it:ellipticity_reaction_diffusion}, and \eqref{it:growth_nonlinearities} hold and
there exists a constant $K$ such that
\begin{align}\label{eq:boundednesscoef}
|\am^{j,k}_i| + \|b_i^j\|_{\ell^2}\leq K, \ \ \text{for all} \  i,j,k \text{ and a.e.\ on }\R_+\times\Omega\times\T^d.
\end{align}

Note that the regularity conditions on the coefficients in Assumption \ref{ass:reaction_diffusion_global}\eqref{it:regularity_coefficients_reaction_diffusion} are left out.
In this section we often use the abbreviation $H^{s}=H^{s,2}(\Tor^d;\R^{\ell})$ for $s\in \R$.

\begin{proposition}[Local existence and uniqueness, and blow-up criteria for $p=q=2$]\label{prop:p=q=2}
Suppose that
\begin{align}\label{eq:condh}
h\in \left\{
  \begin{aligned}
    &(1,4], & \text{ if } & \ d=1, \\
    &(1,3), & \text{ if } & \ d=2,\\
    &\Big(1,\frac{4+d}{d}\Big], & \text{ if }&\ d\geq 3.
  \end{aligned}
\right.
\end{align}
Suppose that for all $i\in \{1,\dots,\ell \}$ parts \eqref{it:reaction_diffusion_global1},\eqref{it:ellipticity_reaction_diffusion}, and \eqref{it:growth_nonlinearities} of Assumption \ref{ass:reaction_diffusion_global}
hold, and \eqref{eq:boundednesscoef} holds.
Let $u_0\in L^0_{\F_0}(\O;L^2)$. Then there exists a unique $(2, 0, 1, 2)$-solution $(u,\sigma)$ to \eqref{eq:reaction_diffusion_system} satisfying $\sigma>0$ a.s.\ and
\[u\in L^2_{\rm loc}([0,\sigma);H^{1})\cap C([0,\sigma);L^2) \ \ \text{a.s.}\]
Moreover, for all $T<\infty$,
\begin{align}\label{eq:condglobalL2}
\P\Big(\sigma<T, \sup_{t\in [0,\sigma)} \|u(t)\|_{L^2} + \|u\|_{L^2(0,\sigma;H^{1})}<\infty\Big) =0.
\end{align}
\end{proposition}

Note that in $d=2$, one cannot reach the scaling invariant case $h=3$, cf.\ Subsection \ref{sss:scaling}.
\begin{proof}
This follows by the same reasoning as in \cite[Theorems 3.3, 3.4 and Section 5.3]{AV22_variational}.
\end{proof}

Of course a natural question whether under further conditions on the coefficients, the solution of Proposition \ref{prop:p=q=2} is a $(p,\kappa,\delta,q)$-solution and has higher regularity than stated in Proposition \ref{prop:p=q=2}. This indeed turns out to be the case.
\begin{proposition}[Regularity for $p=q=2$]\label{prop:p=q=2-reg}
Suppose that \eqref{eq:condh} holds with the additional restriction that $h<4$ for $d=1$. Suppose Assumption \ref{ass:reaction_diffusion_global}$(p,q,h,\delta)$ holds for some $\delta\in (1, 2)$.  Let $u_0\in L^0_{\F_0}(\O;L^2)$. Let $(u,\sigma)$ be the $(2, 0, 1, 2)$-solution to \eqref{eq:reaction_diffusion_system} provided by Proposition \ref{prop:p=q=2}. Then the regularity assertions \eqref{eq:reaction_diffusion_H_theta}-\eqref{eq:reaction_diffusion_C_alpha_beta} hold, i.e.
\begin{align*}
u&\in H^{\theta,r}_{\rm loc}(0,\sigma;H^{1-2\theta,\zeta})  \ \ \text{a.s.\ for all }\theta\in  [0,1/2), \ \  r,\zeta\in (2,\infty),\\
u&\in C^{\theta_1,\theta_2}_{\rm loc}((0,\sigma)\times \Tor^d;\R^{\ell}) \ \ \text{a.s.\ for all }\theta_1\in  [0,1/2), \ \ \theta_2\in (0,1).
\end{align*}
\end{proposition}

\begin{proof}
First consider $d\geq 3$. Then without loss of generality we can assume $h = 1+ \frac{4}{d}$. Fix $\varepsilon\in (0,1/2)$ is so small that $\delta_0 := \varepsilon+1<\delta$.

To prove the claim we will apply \cite[Proposition 6.8]{AV19_QSEE_2} with
\begin{equation}
\begin{aligned}
\label{eq:frac_12_alpha_grather_than_0}
Y_i=H^{-1+2i-\varepsilon},  \   X_i=H^{-1+2i}, \  p=2 ,  \ r\in (2,\infty), \ \frac12 = \frac{1+\alpha}{r} + \frac{\varepsilon}{2}.
\end{aligned}
\end{equation}
Note that since $\varepsilon\in (0,\frac{1}{2})$, \eqref{eq:frac_12_alpha_grather_than_0} yields $\alpha\in (0,\frac{r}{2}-1)$.  First note that the conditions of Proposition \ref{prop:reaction_diffusion_global} with variant \eqref{eq:reaction_diffusion_globali} hold with $(p,q,\kappa,\delta)$ replaced by $(r, 2, \alpha, \delta_0)$. Therefore,  Part (A) of the  proof of Proposition \ref{prop:reaction_diffusion_global} shows that the conditions of \cite[Proposition 6.8]{AV19_QSEE_2} are satisfied if we choose $r$ such that $\frac1r  = \max_{j\in \{1, 2\}}\beta_j -\frac12$, where $\beta_j$ is as in Lemma \ref{l:estimate_nonlinearities}. From the proof of \cite[Proposition 6.8]{AV19_QSEE_2} one sees that $(u,\sigma)$ coincides with the $(r,\alpha,\delta_0,2)$-solution. Therefore, the required regularity follows from Proposition \ref{prop:reaction_diffusion_global} (or equivalently the extrapolation result of \cite[Lemma 6.10]{AV19_QSEE_2}).

Next let $d=2$. Without loss of generality we can assume $h\in (2, 3)$. In this case we need a slight modification of Lemma \ref{l:estimate_nonlinearities}. To this end, let $1<\delta_0 \leq \min\{\delta, 5/3\}$ be fixed but arbitrary. The nonlinearity $\Phi_0$ satisfies the required estimates with $h<3$, $\beta_1 = \frac{\delta_0}{2}+\frac12 - \frac1h$. Indeed, to see this in \eqref{eq:estimate_rone_reaction_diffusion} one can take $\xi=1$ (using $\delta_0>1$), and $\theta = \frac{d}{q} -\frac{d}{h\xi} = 1-\frac{2}{h}<\frac13$. Note that we are in the case $q<\frac{d(h-1)}{\reg}$ and $\theta<2-\delta_0$ follows from $\delta_0\leq 5/3$. The estimates for $\Phi_1$ and $\Gamma$ can be done by taking the optimal choices for $\eta$ and $\phi$ in the Sobolev embeddings where we replace $\delta$ by $\delta_0$.

Note that in Step 1 of the proof of  Proposition \ref{prop:reaction_diffusion_global} we have
$q<\frac{d(h-1)}{\delta_0}$ and $\frac{1+\a}{p}+\frac{1}{2}(\delta_0+\frac{d}{q})\leq \frac{h}{h-1}$ with $d=p=q=2$ and $\kappa=0$ if we take $\delta_0 \leq \frac{h+1}{h-1}$.
Now we are in the situation that we can repeat the argument of the case $d\geq 3$, where we take $\delta_0 = 1+\varepsilon$ with $\varepsilon\in (0,1/2)$ small enough.

In the case $d=1$, we argue in a similar way as for $d=2$. We may suppose that $h\in (3, 4)$. We first check Proposition \ref{prop:reaction_diffusion_globald=1}\eqref{it2:reaction_diffusion_globald=1}. One can check that the minimum is  attained at the middle expression. Let $r>2$, $\alpha\in (0,\frac{r-1}{2})$ and $\delta_0\in (1, \delta\wedge \frac{7}{4}]$ be arbitrary but fixed. Using $h<4$
and that the right-hand side is strictly decreasing in $h$, we find that for $\delta_0$ small enough
\[\frac{1+\alpha}{r}<\frac12< \frac{h}{h-1} \Big(1-\frac{\delta_0}{2} + \frac{1}{2h} - \frac{1}{4} \Big).\]
Thus Proposition \ref{prop:reaction_diffusion_globald=1} is applicable with $(p,q,\kappa,\delta)$ replaced by $(r,2,\alpha,\delta_0)$. Recall from \eqref{eq:beta1d=1} that
$\beta_1 =  \max\Big\{\frac{1}{2q}-\frac{1}{2h},0\Big\} + \frac{\delta_0}{2}\in\big(\frac12, 1\big)$. Also recall that from the proof of Proposition \ref{prop:reaction_diffusion_globald=1} one can see that $\beta_2$ can be taken as in Lemma \ref{l:estimate_nonlinearities}. Therefore, we can repeat the argument of the case $d\geq 3$ once more.
\end{proof}

The following complements the blow-up criteria of Theorem \ref{t:blow_up_criteria} and of Corollary \ref{cor:blow_up_criteria}. In particular, it shows that the solutions provided by Proposition \ref{prop:reaction_diffusion_global} (or Theorem \ref{t:reaction_diffusion_global_critical_spaces}) are \emph{global} in time if one can obtain energy estimates in an $L^2$-setting. Here the initial data can be from space with lower smoothness than $L^2(\T^d;\R^\ell)$. Thus the result extends the class of initial data covered by Proposition \ref{prop:p=q=2} under some smoothness conditions on the coefficients.

\begin{proposition}[Global existence for rough initial data]\label{prop:globalp2q2rough}
Suppose that the conditions of Proposition \ref{prop:reaction_diffusion_global} are satisfied, in particular $u_0\in L^0_{\F_0}(\O;B^{2-\reg-2\frac{1+\a}{p}}_{q,p})$. Let $(u,\sigma)$ be the $(p,\kappa,\delta,q)$-solution obtained there.
Suppose that \eqref{eq:condh} holds with $h<4$ if $d=1$. Then, for all $0<s<T<\infty$,
\begin{align}\label{eq:condglobalL2s}
\P\Big(s<\sigma<T, \,  \sup_{t\in [s,\sigma)} \|u(t)\|_{L^2} + \|u\|_{L^2(s,\sigma;H^{1})}<\infty\Big) =0.
\end{align}
\end{proposition}
\begin{proof}
We extend the argument in Theorem \ref{t:blow_up_criteria} to $p=q=2$.
First note that $u$ satisfies the regularity stated in \eqref{eq:reaction_diffusion_H_theta_1} and \eqref{eq:reaction_diffusion_C_alpha_beta_1}. In particular, the $L^2$-norm and $H^{1}$-norm appearing in \eqref{eq:condglobalL2s} are well-defined.

Proposition \ref{prop:p=q=2} (up to translation) yields the existence of a $(2,0,1,2)$-solution $(v,\tau)$ on $[s,\infty)$ to \eqref{eq:reaction_diffusion_global_stochastic_s} with initial data $\one_{\{\sigma>s\}}u(s)$ which satisfies $\tau>s$ a.s., and by Proposition \ref{prop:p=q=2-reg} (recall that $h<4$ if $d=1$),
\begin{equation}
\label{eq:v_regularizesp=2}
v\in H^{\theta,r}_{\rm loc}(s,\tau;H^{1-2\theta,\zeta})\quad  \text{a.s.\ for all }\theta\in  [0,1/2), \  r,\zeta\in (2,\infty).
\end{equation}
Moreover, by Proposition \ref{prop:p=q=2} (up to translation),
\begin{equation}\label{eq:tausmallT}
\P\Big(\tau<T,\, \sup_{t\in [s, \tau)}\|v(t)\|_{L^2} +  \|v\|_{L^2(s,\tau;H^{1})} <\infty\Big)=0.
\end{equation}

We claim that
\begin{equation}
\label{eq:tau_sigma_u_v_equalityp=q=2}
\tau=\sigma \text{ a.s.\ on }\{\sigma>s\} \quad \text{ and }\quad u=v \text{ a.e.\ on }[s,\sigma)\times \{\sigma>s\}.
\end{equation}
Hence \eqref{eq:condglobalL2s} follows from \eqref{eq:tausmallT} and \eqref{eq:tau_sigma_u_v_equalityp=q=2}.

It remains to prove the claim \eqref{eq:tau_sigma_u_v_equalityp=q=2}. Since $(u|_{[s,\sigma)\times \V}, \one_{\V} \sigma+ \one_{\O\setminus\V} s)$ is a \emph{local} $(2,0,1,2)$-solution to \eqref{eq:reaction_diffusion_global_stochastic_s} with initial data with initial data $\one_{\{\sigma>s\}}u(s)$, the maximality of $(v,\tau)$ yields
\begin{equation}\label{eq:maximp2q}
\sigma\leq \tau \text{ a.s.\ on }\{\sigma>s\} \quad \text{ and }\quad u=v \text{ a.e.\ on }[s,\sigma).
\end{equation}
To conclude it is enough to show that $\P(s<\sigma<\tau)=0$. To this end we will apply the blow-up criteria \eqref{eq:blow_up_criteria_u}. Indeed, by \eqref{eq:v_regularizesp=2} and \eqref{eq:maximp2q} we have $u=v\in L^p_{\loc}((s,\sigma];H^{\gamma,q})$ a.s.\ on $\{s<\sigma<\tau\}$. Combining this with \eqref{eq:L_p_up_to_zero} we find $u\in L^p(0,\sigma;H^{\gamma,q})$ a.s.\ on  $\{s<\sigma<\tau\}$. Similarly, one can check that $\sup_{t\in [0,\sigma)}\|u(t)\|_{B^{\beta}_{q,p}}<\infty$ a.s.\ on $\{s<\sigma<\tau\}$, and therefore
\begin{align*}
\P(s<\sigma<\tau)
&= \P\Big(\{s<\sigma<\tau\} \cap \Big\{\sup_{t\in [0,\sigma)}\|u(t)\|_{B^{\beta}_{q,p}}+ \|u\|_{L^p(0,\sigma;H^{\g,p})}<\infty\Big\}\Big)\\
&\leq \P\Big(\sigma<T,\, \sup_{t\in [0,\sigma)}\|u(t)\|_{B^{\beta}_{q,p}}+ \|u\|_{L^p(0,\sigma;H^{\g,p})}<\infty\Big)\stackrel{\eqref{eq:blow_up_criteria_u}}{=}0.
\end{align*}
\end{proof}

\begin{remark}
From the proof of Proposition \ref{prop:p=q=2-reg} it follows that the compatibility result of Proposition \ref{prop:comp} extends to $p=q=2$ and $\delta=1$ under the restrictions on $h$ and $d$ stated in Proposition \ref{prop:p=q=2-reg}.
\end{remark}

\begin{remark}\label{rem:weakerh}
By splitting the locally Lipschitz and growth conditions on $f$, $F$ and $g$ stated in Assumption \ref{ass:reaction_diffusion_global}\eqref{it:growth_nonlinearities} into three different growth conditions with parameters $h_f, h_F$ and $h_g$ instead of $h$, one can further weaken the conditions in Propositions \ref{prop:p=q=2}-\ref{prop:globalp2q2rough}. Indeed, from \cite[Section 5.3]{AV22_variational} one sees that in Proposition \ref{prop:p=q=2}  it is enough to assume $h_F, h_g\leq \frac{d+4}{d}$ for $d\geq 1$. The assumption on $h_f$ remains as it was for $h$. This leads to a slightly weaker assumptions on $F$ and $g$ for $d\in \{1,2\}$. The same actually applies to the more general of Lemma \ref{l:estimate_nonlinearities}.
\end{remark}

\begin{remark}\label{rem:Hk}
One can also replace $(L^2,H^{1})$ by $(H^{s},H^{s+1})$ in the above results. This gives a wider range of nonlinearities which can be treated if $s$ is large, but at the same time this choice requires more restrictions on the regularity of the coefficients, the spatial smoothness of the nonlinearities $f,F,g$ and on the initial data (see e.g.\ \cite[Section 5.4]{AV22_variational}).
\end{remark}

\appendix
\section{A maximum principle for SPDEs}
\label{s:maximum}
In \cite{Kry13}, Krylov presented a maximum principle for linear scalar second order SPDEs, which are allowed to be degenerate. In the proof of the positivity result of Theorem \ref{thm:positivity} we need such a result in the non-degenerate setting, but with coefficients which have less smoothness. Below we extend the maximum principle to the case of non-smooth coefficients as one can use an approximation argument in the non-degenerate case.
As before Theorem \ref{thm:positivity}, here we say that $v\in \D'(\Tor^d)$ is positive (or $v\geq 0$) if $\l \varphi,v\r\geq 0$ for all test functions $\varphi$ satisfying $\varphi\geq 0$ on $\Tor^d$.

\begin{lemma}[Maximum principle for second order SPDEs of scalar type]\label{lem:maxprinciple}
Suppose that $a^{ij}, a^i, b^i, c:[0,T]\times \Omega\times\T^d\to \R$, $(\sigma^{ik})_{k\geq 1}, (\nu^k)_{k\geq 1}:[0,T]\times \Omega\times\T^d\to \ell^2$  are bounded and $\Progress\otimes \Borel(\T^d)$-measurable, and there is a $\gamma>0$ such that a.s.\
\begin{align}\label{eq:parabolcoeff}
\sum_{i,j=1}^d \Big(a^{ij}-\frac{1}{2} \alpha^{ij}\Big)
 \xi_i \xi_j
\geq  \gamma |\xi|^2 \ \ \text{on $[0,T]\times \T^d$,}
\end{align}
where $\alpha^{ij} = (\sigma^i, \sigma^j)_{\ell^2}$. Let $u_0\in L^2(\Omega;L^2(\T^d))$ and $f\in L^2(\Omega\times (0,T);H^{-1}(\T^d))$ be such that a.e.\ $u_0\geq 0$ and $f\geq 0$. Let $u\in L^2(\Omega;L^2(0,T;H^1(\T^d)))\cap L^2(\Omega;C([0,T];L^2(\T^d)))$ be the solution to
\begin{equation}
\label{eq:parabolic_lin_problem}
\left\{
\begin{aligned}
&\dd u - A u\, \dd t = f \, \dd t + \sum_{k\geq 1} B^k u \, \dd w^k_t,\\
&u(0)=u_0,
\end{aligned}
\right.
\end{equation}
where
\begin{align*}
A u &= \sum_{i,j=1}^d  \partial_i(a^{ij} \partial_j u) + \sum_{i=1}^d\partial_i (a^i u) + \sum_{i=1}^d b^i \partial_i u + cu, \ \ \text{and} \ \ B^k u = \sum_{i=1}^d \sigma^{ik} \partial_i u +  \nu^k u.
\end{align*}
Then a.s. for all $t\in [0,T]$, one has $u\geq 0$.
\end{lemma}

In the above solutions to \eqref{eq:parabolic_lin_problem} are understood in the sense of Defintion \ref{def:solution} with $q=p=2$ and $\a=0$.
A similar result holds for general domains $\Dom \subseteq \R^d$ with (for instance) Dirichlet boundary conditions.

\begin{proof}
For convenience of the reader we give the details of the approximation argument. Note that a unique solution exists by the classical variational setting (see e.g.\ \cite[Theorem 4.2.4]{LR15}) applied to the linear problem \eqref{eq:parabolic_lin_problem}. In case of smooth coefficients and smooth $f$, it follows from that $u\geq 0$ (see \cite{Kry13}). In order to prove $u\geq 0$ in the above setting, it suffices to construct $(u_n)_{n\geq 1}$ such that $u_n\geq 0$ and $u_n\to u$ in $L^2(\Omega;C([0,T];L^2(\T^d)))$.

To approximate $u$ we use a standard mollifier argument. Let $\rho\in C^\infty(\T^d)$ be such that $\rho\geq 0$ and $\int_{\T^d} \rho \dd x = 1$. Set $\rho_n(x) = n^d\rho(n x)$, $h_n = \rho_n*h$ for $h\in \{a^{ij}, a^i, b^i, c, \sigma^{ik}, \nu^k, \alpha^{ij}, f\}$ and
\begin{align*}
A_n v &= \sum_{i,j=1}^d  \partial_i(\wt{a}^{ij}_n \partial_j v) + \sum_{i=1}^d\partial_i (a^i_n v) + \sum_{i=1}^d b^i_n \partial^i v + c_nv, \ \ \text{and} \ \ B^k_n v = \sum_{i=1}^d \sigma^{ik}_n \partial_i v +  \nu^k_n v,
\end{align*}
where $\wt{a}^{ij}_n = a^{ij}_n +\frac{1}{2} (\sigma^i_n, \sigma^j_n)_{\ell^2} - \frac12 \alpha^{ij}_n$. Note that in general $(\sigma^i_n, \sigma^j_n)_{\ell^2} \neq \alpha^{ij}_n$, but the equality holds pointwise a.e.\ in the limit as $n\to \infty$, (possibly) up to a subsequence. Due to this seemingly unnatural definition
we can again check the parabolicity condition  \eqref{eq:parabolcoeff}:
\[\sum_{i,j=1}^d \big(\wt{a}^{ij}_n-\frac{1}{2} (\sigma^i_n, \sigma^j_n)_{\ell^2}\big)
 \xi_i \xi_j = \sum_{i,j=1}^d \big(a^{ij}_n-\frac12 \alpha^{ij}_n\big)
 \xi_j \xi_j = \Big[\sum_{i,j=1}^d \big(a^{ij}-\frac12 \alpha^{ij}\big)\xi_i \xi_j\Big]*\rho_n
\geq  \gamma |\xi|^2.
\]

Let $u_n\in Z:=L^2((0,T)\times \O;H^1(\T^d))\cap L^2(\Omega;C([0,T];L^2(\T^d)))$ be the unique solution to
\begin{equation*}
\left\{
\begin{aligned}
&\dd u_n - A_n u_n\,\dd  t = f_n \, \dd t + \sum_{k\geq 1} B^k_n u_n \, \dd w^k_t,\\
&u_n(0)=u_0.
\end{aligned}
\right.
\end{equation*}
Since the coefficients in the above linear SPDE are smooth, we apply the periodic case of \cite[Theorem 4.3]{Kry13} to obtain $u_n\geq 0$. It remains to show $u_n\to u$ in $L^2(\Omega;C([0,T];L^2(\T^d)))$.

Note that $v_n = u- u_n$ satisfies
\begin{equation*}
\left\{
\begin{aligned}
&\dd v_n - A_n v_n\, \dd t
= F_n \, \dd t + \sum_{k\geq 1} \big(B^k_n v_n  +  G^k_n\big) \dd w^k_t,\\
&v_n(0)=0,
\end{aligned}
\right.
\end{equation*}
where
\[F_n := (A - A_n)u  + f-f_n \ \ \text{and} \ \  G^k_n := (B^k - B_n^k)u.\]
Therefore, by standard regularity estimates (see \cite[Theorem 4.2.4]{LR15} and its proof),
\begin{align*}
\|v_n\|_Z& \leq C\|F_n\|_{L^2((0,T)\times \Omega;H^{-1}(\T^d))} + C\|(G_n^k)_{k\geq 1}\|_{L^2((0,T)\times \Omega;L^2(\T^d;\ell^2))}
\\ & \leq  C\|(A - A_n)u + f - f_n\|_{L^2((0,T)\times \Omega;H^{-1}(\T^d))} \\
&+ C\|((B^k - B_n^k )u)_{k\geq 1}\|_{L^2((0,T)\times \Omega;L^2(\T^d;\ell^2))}.
\end{align*}
Since for each $h\in \{a^{ij}, a^i, b^i, c, \sigma^{ik}, \nu^k, \alpha^{ij}, f\}$,  we have $h_{n_m}\to h$ a.e.\ for a suitable subsequence, and since $u\in L^2((0,T)\times \O;H^1(\T^d))$, it follows from the dominated convergence theorem that
\[\|(A - A_n)u\|_{L^2((0,T)\times \Omega;H^{-1}(\T^d))}\to 0, \ \ \text{and} \ \
\|((B^k - B_n^k )u)_{k\geq 1}\|_{L^2((0,T)\times \Omega;L^2(\T^d;\ell^2))}\to 0.
\]
Note that in the above we used that $u\in L^2((0,T)\times \O;H^1(\Tor^d))$ as $\g>0$ in \eqref{eq:parabolcoeff}.
For the inhomogeneity $f$, writing $\wt{f} = (1-\Delta)^{-1/2} f$, we have
\[\|f-f_n\|_{L^2((0,T)\times \Omega;H^{-1}(\T^d))}  \eqsim \|\wt{f} -\rho_n * \wt{f}\|_{L^2((0,T)\times \Omega;L^2(\T^d))}\to 0.\]
Combining the above we have $\|v_n\|_Z\to 0$, as required.
\end{proof}

\def\polhk#1{\setbox0=\hbox{#1}{\ooalign{\hidewidth
  \lower1.5ex\hbox{`}\hidewidth\crcr\unhbox0}}} \def\cprime{$'$}

\end{document}